\documentclass[a4paper, 12pt]{amsart}

\usepackage[HTML, hyperref, svgnames, table]{xcolor}

\usepackage[ unicode,
             colorlinks,
             allcolors = blue
           ]{hyperref}

\usepackage{amsmath,amsthm,indentfirst}
\usepackage{amssymb}
\usepackage{mathtools}
\usepackage{amsfonts,dsfont}
\usepackage{calrsfs}
\usepackage{amscd}
\usepackage{url}
\usepackage{enumerate}
\usepackage{tikz-cd}

\theoremstyle{plain}
\newtheorem{thm}{Theorem}
\newtheorem{cor}[thm]{Corollary}
\newtheorem{lem}[thm]{Lemma}
\newtheorem{prop}[thm]{Proposition}

\newtheorem{question}[thm]{Question}

\newtheorem{conjecture}[thm]{Conjecture}

\newtheorem{definition}[thm]{Definition}

\newenvironment{thm'}
{\addtocounter{thm}{-1}%
	\begin{thm}}
	{\end{thm}}

\newenvironment{cor'}
{\addtocounter{thm}{-1}%
	\begin{cor}}
	{\end{cor}}

\newenvironment{lem'}
{\addtocounter{thm}{-1}%
	\begin{lem}}
	{\end{lem}}

\newenvironment{prop'}
{\addtocounter{thm}{-1}%
	\begin{prop}}
	{\end{prop}}

\DeclareMathAlphabet{\mathbbold}{U}{bbold}{m}{n}
\def\bb1{\mathbbold{1}}
\def\bbz{\mathbb{Z}}
\def\bbq{\mathbb{Q}}
\def\bbf{\mathbb{F}}
\def\bbr{\mathbb{R}}
\def\bba{\mathbb{A}}
\def\bbb{\mathbb{B}}
\def\bbc{\mathbb{C}}

\def\bbh{\mathbb{H}}

\def\bbg{\mathbb{G}}
\def\bbt{\mathbb{T}}
\def\bbv{\mathbb{V}}

\def\bbp{\mathbb{P}}

\def\bbs{\mathbb{S}}
\def\bbw{\mathbb{W}}

\def\bbx{\mathbb{X}}
\def\bbgl{\mathbb{GL}}

\def\gcal{\mathcal{G}}

\def\ocal{\mathcal{O}}
\def\ocal{\mathcal{O}}

\def\kcal{\mathcal{K}}

\def\ccal{\mathcal{C}}

\def\tcal{\mathcal{T}}

\def\cal{\mathcal{H}}
\def\lcal{\mathcal{L}}
\def\pcal{\mathcal{P}}

\def\xcal{\mathcal{X}}
\def\hfr{\mathfrak{h}}
\def\afr{\mathfrak{a}}

\def\pfr{\mathfrak{p}}

\def\zfr{\mathfrak{z}}
\def\lfr{\mathfrak{l}}

\def\gfr{\mathfrak{g}}
\def\mfr{\mathfrak{m}}

\def\tfr{\mathfrak{t}}

\DeclareMathOperator\id{id}

\DeclareMathOperator\Spec{Spec}
\DeclareMathOperator\Hom{Hom}

\DeclareMathOperator\Aut{Aut}

\DeclareMathOperator\Cl{Cl}
\DeclareMathOperator\SL{SL}
\DeclareMathOperator\GL{GL}

\DeclareMathOperator\Ad{Ad}

\DeclareMathOperator\Lie{Lie}

\DeclareMathOperator\Gal{Gal}

\DeclareMathOperator\pr{pr}

\DeclareMathOperator\Tr{Tr}

\newcommand{\inv}[1]{{#1}^{-1}}
\newcommand{\grpgen}[1]{\langle{#1}\rangle}
\newcommand{\convolve}[2]{{#1}^{(#2)}}

\def\ugfr{\underline{\mathfrak{g}}}
\def\uzfr{\underline{\mathfrak{z}}}
\def\utfr{\underline{\mathfrak{t}}}
\def\uhfr{\underline{\hfr}}
\def\h{\hspace{1mm}}

\def\vare{\varepsilon}
\def\be{\begin{equation}}
\def\ee{\end{equation}}
\def\one{\mathds{1}}

\newcommand{\wt}[1]{\widetilde{#1}}
\newcommand{\wh}[1]{\widehat{#1}}
\def\wtx{\widetilde{X}}
\def\acts{\curvearrowright}
\def\vecy{\overrightarrow{y}}
\begin{document}
\title{Towards super-approximation in positive characteristic}
\author{Brian Longo and Alireza Salehi Golsefidy}
\address{Mathematics Dept, University of California, San Diego, CA 92093-0112}
\email{blongo@ucsd.edu}
\email{golsefidy@ucsd.edu}
\thanks{A.S-G. was partially supported by the NSF grants  DMS-1303121, DMS-1602137, DMS-1902090, and A. P. Sloan Research Fellowship. A.S-G. would like to thank the IAS for its hospitality; part of this work was done while A.S-G. was visiting the IAS. This article contains the main results proved in B.L.'s Ph.D. thesis which was done in the UCSD}
\date{08/19/2019}
\begin{abstract}
In this note we show that the family of Cayley graphs of a finitely generated subgroup of $\GL_{n_0}(\bbf_p(t))$ modulo some {\em admissible} square-free polynomials is a family of expanders under certain algebraic conditions. 

Here is a more precise formulation of our main result. For a positive integer $c_0$, we say a square-free polynomial is $c_0$-admissible if degree of irreducible factors of $f$ are distinct integers with prime factors at least $c_0$. Suppose $\Omega$ is a finite symmetric subset of $\GL_{n_0}(\bbf_p(t))$, where $p$ is a prime more than $5$. Let $\Gamma$ be the group generated by $\Omega$. Suppose the Zariski-closure of $\Gamma$ is connected, simply-connected, and absolutely almost simple; further assume that the field generated by the traces of $\Ad(\Gamma)$ is $\bbf_p(t)$. Then for some positive integer $c_0$ the family of Cayley graphs 
${\rm Cay}(\pi_{f(x)}(\Gamma),\pi_{f(x)}(\Omega))$ as $f$ ranges in the set of $c_0$-admissible polynomials is a family of expanders, where  $\pi_{f(t)}$ is the quotient map for the congruence modulo $f(t)$. 
\end{abstract}

\maketitle
\tableofcontents
\section{Introduction}\label{sectionintro}
\subsection{Statement of the main results}
Let $\Gamma$ be a subgroup of a compact group $G$. Suppose $\Omega$ is a finite symmetric (that means $\omega\in \Omega$ implies $\omega^{-1}\in \Omega$) generating set of $\Gamma$. Suppose $\overline{\Gamma}$ is the closure of $\Gamma$ in $G$ and $\pcal_{\Omega}$ is the probability counting measure on $\Omega$. Let 
\[
T_{\Omega}:L^2(\overline{\Gamma})\rightarrow L^2(\overline{\Gamma}), \h T_{\Omega}(f):=\pcal_{\Omega}\ast f:=\frac{1}{|\Omega|} \sum_{\omega\in \Omega} L_{\omega}(f),
\]
where $L_{\omega}(f)(g):=f(\omega^{-1}g)$.
Then it is well-known that $T_{\Omega}$ is a self-adjoint operator, $T_{\Omega}(\one_{\overline{\Gamma}})=\one_{\overline{\Gamma}}$ where $\one_{\overline{\Gamma}}$ is the constant function on $\overline{\Gamma}$, and the operator norm $\|T_{\Omega}\|$ is 1. So the spectrum of $T_{\Omega}$ is a subset of $[-1,1]$ and $T_{\Omega}$ sends the space $L^2(\overline{\Gamma})^{\circ}$ orthogonal to the constant functions to itself. Let $T_{\Omega}^{\circ}$ be the restriction of $T_{\Omega}$ to $L^2(\overline{\Gamma})^{\circ}$. Let
\[
\lambda(\pcal_{\Omega};G):=\sup\{|c||\h c \text{ is in the spectrum of } T_{\Omega}^{\circ}\}.
\]
If $\lambda(\pcal_{\Omega};G)<1$, we say the left action $\Gamma\acts G$ of $\Gamma$ on $G$ has {\em spectral gap}.

It is worth mentioning that, if $\Omega_1$ and $\Omega_2$ are two generating sets of $\Gamma$ and $\lambda(\pcal_{\Omega_1};G)<1$, then $\lambda(\pcal_{\Omega_2};G)<1$. So having spectral gap is a property of the action $\Gamma\acts G$, and it is independent of the choice of a generating set for $\Gamma$.

The following is the main result of this article:
\begin{thm}\label{t:SpectralGap}
	Let $\Omega$ be a finite symmetric subset of $\GL_{n_0}(\bbf_p[t,1/r_0(t)])$ where $p>5$ is prime and $r_0(t)\in \bbf_p[t]\setminus \{0\}$. Let $\Gamma$ be the group generated by $\Omega$. Suppose the Zariski-closure $\bbg$ of $\Gamma$ in $(\GL_{n_0})_{\bbf_p(t)}$ is a connected, simply-connected, absolutely almost simple group. Suppose the field generated by $\Tr(\Ad(\Gamma))$ is $\bbf_p(t)$. Then there is a positive integer $c_0$ such that 
	\[
		\sup_{\{\ell_i(t)\}_i\in I_{r_0,c_0}}\lambda(\pcal_{\Omega};\prod_{i=1}^{\infty} \GL_{n_0}(\bbf_p[t]/\langle \ell_i(t)\rangle))<1,
	\]
 where $\{\ell_i(t)\}_{i=1}^{\infty}\in I_{r_0,c_0}$ if and only if $\ell_i(t)$ are irreducible, $\ell_i(t)\nmid r_0(t)$, and $\{\deg \ell_i\}_{i=1}^{\infty}$ is a strictly increasing sequence consisting of integers more than $1$ with no prime factors less than $c_0$. 
\end{thm}
 It is well-known that Theorem~\ref{t:SpectralGap} has immediate application in the explicit construction of expanders. Let us quickly recall that a family of $d$-regular graphs $X_i$ is called a family of expanders if the size $|V(X_i)|$ of the set of vertices goes to infinity and there is a positive number $\delta_0$ such that for any subset $B$ of $V(X_i)$ we have
 \[
 \frac{|e(B,V(X_i)\setminus B)|}{\min(|B|,|V(X_i)\setminus B|)}>\delta_0,
 \]
 where $e(B,C)$ is the set of edges that connect a vertex in $B$ to a vertex in $C$. Expanders have a lot of applications in theoretical computer science (see \cite{HLW} for a survey on such applications).
 
 Now we can give the equivalent formulation of Theorem~\ref{t:SpectralGap} in terms of expander graphs (see \cite[Remark 15]{SG:SAI} or \cite[Section 4.3]{Lub}). 
 
\begin{thm'}\label{thmmainthm} 
	Let $\Omega$ be a finite symmetric subset of $\GL_{n_0}(\bbf_p[t,1/r_0(t)])$ where $p>5$ is prime and $r_0(t)\in \bbf_p[t]\setminus\{0\}$. Let $\Gamma$ be the subgroup generated by $\Omega$. Suppose the Zariski-closure $\bbg$ of $\Gamma$ in $(\GL_{n_0})_{\bbf_p(t)}$ is a connected, simply-connected, absolutely almost simple group. Suppose the field generated by $\Tr(\Ad(\Gamma))$ is $\bbf_p(t)$. Then there is a positive integer $c_0$ such that the family of Cayley graphs \[\{{\rm Cay}(\pi_{f(t)}(\Gamma),\pi_{f(t)}(\Omega))|\h f(t)\in S_{r_0,c_0}\}\] is a family of expanders where $S_{r_0,c_0}$ consists of square-free polynomials  $f(t)\in \bbf_p[t]$ with prime factors $\ell_i(t)$ such that (1) $\ell_i(t)\nmid r_0(t)$, (2) $\deg \ell_i>1$, (3) $\deg \ell_i \neq \deg \ell_j$ if $i\neq j$, and (4) $\deg \ell_i$ does not have a prime factor less than $c_0$ and $\pi_{f(t)}$ is induced by the quotient map $\pi_{f(t)}:\bbf_p[t,1/r_0(t)]\rightarrow \bbf_p[t]/\langle f(t)\rangle$. 
\end{thm'} 
 
 \subsection{What super-approximation is and an ultimate speculation}
 
 In order to put Theorem~\ref{t:SpectralGap} in the perspective of previous works, let us say what {\em super-approximation} is in a very general setting.  
 
 \begin{definition}
	Suppose $A$ is an integral domain, and $\Omega$ is a finite symmetric subset of $\GL_{n_0}(A)$. Let $\Gamma$ be the group generated by $\Omega$. Suppose $\ccal$ is a family of finite index ideals of $A$. We say $\Gamma$ has {\em super-approximation with respect to $\ccal$} if $\sup_{\afr\in \ccal}\lambda(\pi_{\afr}[\pcal_{\Omega}];\GL_{n_0}(A/\afr))<1$, where $\pi_{\afr}$ is the group homomorphism induced by the quotient map $A\rightarrow A/\afr$ and $\pi_{\afr}[\pcal_{\Omega}]$ is the push-forward of $\pcal_{\Omega}$ under $\pi_{\afr}$. We simply say $\Gamma$ has {\em super-approximation} if it has super-approximation with respect to the set of all the finite index ideals of  $A$.
\end{definition}

Because of several ground breaking results in the past decade (see~\cite{Hel1,Hel2,BGT,PS}, \cite{BG1}-\cite{BV}, \cite{Var}, \cite{SG:SAI}-\cite{SGV}), we have a very good understanding of super-approximation property for finitely generated subgroups of linear groups over $A:=\bbz[1/q_0]$ (a finitely generated subring of $\bbq$). In this case, it is proved that $\Gamma$ has super-approximation with respect to fixed powers of square-free ideals \cite{SGV,SG:SAII} and powers of prime ideals \cite{SG:SAI,SG:SAII} if and only if the connected component $\bbg^{\circ}$ of the Zariski-closure of $\Gamma$ in $(\GL_{n_0})_{\bbq}$ has trivial abelianization. Based on these results we have the following conjecture. 
\begin{conjecture}[Super-approximation conjecture over $\bbq$]
Suppose $\Omega$ is a finite symmetric subset of $\GL_{n_0}(\bbz[1/q_0])$, $\Gamma=\langle\Omega\rangle$, and $\bbg^{\circ}$ is the connected component of the Zariski-closure of $\Gamma$ in $(\GL_{n_0})_{\bbq}$. Then $\Gamma$ has super-approximation if and only if $\bbg^{\circ}$ has trivial abelianization.	
\end{conjecture}
 
In the beautiful survey by Lubotzky~\cite{Lub:survey}, he goes further and make an analogues conjecture (see~\cite[Conjecture 2.25]{Lub:survey}) for an arbitrary finitely generated integral domain $A$. Notice that such a conjecture implies that {\em super-approximation is a Zariski-topological property}; that means if two groups have equal Zariski-closures, then either both of them have super-approximation or neither have this property. It turns out that this conjecture is false in this generality (see~\cite[Example 5]{SGV}); there are two finitely generated subgroups of $\GL_{n_0}(\bbz[i])$ such that (1) they have equal Zariski-closures in $(\GL_{n_0})_{\bbq[i]}$, and (2) one of them has super-approximation and the other one does not. This shows that for an arbitrary integral domain $A$ one needs a refiner understanding of $\Gamma$ to determine if it has super-approximation. The key point is that super-approximation is about how well $\Gamma$ is distributed in its closure $\overline{\Gamma}$ in the compact group $\GL_{n_0}(\wh{A})$ where $\wh{A}:=\varprojlim_{|A/\afr|<\infty}A/\afr$ is the profinite closure of the ring $A$. When the field of fractions $Q(A)$ has a subfield $F$ such that $[Q(A):F]<\infty$, $\Gamma$ might satisfy some {\em hidden} polynomial relations over $F$ which {\em disappear} over $Q(A)$. Of course such polynomial relations are still satisfied in $\overline{\Gamma}$; and so these are vital in understanding the group structure of $\overline{\Gamma}$. To detect the mentioned hidden polynomial relations, one has to use Weil's restriction of scalars and view $\GL_{n_0}(Q(A))$ as the $F$-points of $R_{Q(A)/F}((\GL_{n_0})_{Q(A)})$. More or less what we are hoping for is to have a finitely generated ring $A_0$ and a group scheme $\gcal_0$ over $A_0$ such that $\overline{\Gamma}$ can be realized as an open subgroup of $\gcal_0(\wh{A_0})$ where $\wh{A_0}$ is the profinite closure of $A_0$. Strong approximation (see~\cite{MVW,Wei,Nor,PinkStrongApproximation}) gives us such a result under various extra algebraic conditions. This is partially responsible for some of the extra technical conditions in Theorem~\ref{t:SpectralGap} compared to the mentioned results over $\bbz[1/q_0]$; it will be explained later why we need some additional technical conditions. 

In light of this discussion, it makes sense to formulate a conjecture for super-approximation property of a finitely generated subgroup $\Gamma$ of $\GL_{n_0}(A)$ {\em based on group theoretic properties of its closure $\overline{\Gamma}$ in $\GL_{n_0}(\wh{A})$}. As Lubotzky says in his survey~\cite[Conjecture 2.25]{Lub:survey} this conjecture is quite a {\em fantasy} at this point.

\begin{conjecture}[Super-approximation conjecture for a finitely generated integral domain]\label{conj:ultimate}
Suppose $A$ is a finitely generated integral domain, $\Omega$ is a finite symmetric subset of $\GL_{n_0}(A)$, and $\Gamma$ is the subgroup generated by $\Omega$. Let $\wh{A}:=\varprojlim_{|A/\afr|<\infty} A/\afr$ be the profinite closure of $A$ and $\overline{\Gamma}$ is the closure of $\Gamma$ in $\GL_{n_0}(\wh{A})$. Then $\Gamma$ has super-approximation if and only if any open subgroup $\overline{\Lambda}$ of $\overline{\Gamma}$ has finite abelianization; that means $|\overline{\Lambda}/[\overline{\Lambda},\overline{\Lambda}]|<\infty$. 
\end{conjecture}
Let us make two remarks: (1) since $A$ is a finitely generated ring, for any maximal ideal $\mfr$ we have that $A/\mfr$ is a finitely generated ring and a field; and so $|A/\mfr|<\infty$ if $\mfr$ is a maximal ideal. Moreover, since $A$ is a finitely generated integral domain, it is a Jacobson ring which means intersection of its maximal ideals is zero. Hence $A$ can be (naturally) embedded into $\wh{A}$. Therefore it does make sense to talk about the closure of $\Gamma$ in $\GL_{n_0}(\wh{A})$. (2) Using the argument given in \cite[Proposition 8]{SG:SAII} one can get the ``only if" part of Conjecture~\ref{conj:ultimate}. It is worth mentioning that {\em super-approximation} (also known as {\em superstrong approximation}) has been found to be extremely instrumental in a wide range of problems; see \cite{ThingrpsandSSA} for a collection of its applications.
 
\subsection{Best related result prior to this work}  
 
The best known result on super-approximation for linear groups over a global function field, prior to this work, is due to Bradford~\cite{Bra}. In \cite{Bra}, under the extra assumption that the degree $\deg \ell_i$ of irreducible factors $\ell_i$ are prime, a version of Theorem~\ref{thmmainthm} for subgroups of $\SL_2(\bbf_p[t])$ is proved. Bradford also highlights  many of the subtleties involved in the positive characteristic case.

\subsection{Notation}\label{sectionnotation}
Throughout this paper for any group $G$ and a subgroup $H$, $Z(G)$ is the center of $G$, $C_G(H)$ is the centralizer of $H$ in $G$, and $N_G(H)$ is the normalizer of $H$ in $G$ as usual.  If $G$ and $H$ are algebraic groups, then these notions are considered in the category of algebraic groups.

For a finite subset $S$ of a group $G$, we denote by $\pcal_S$ the uniform probability measure supported on $S$; that means 
\begin{displaymath}
   \pcal_S(A) = |A\cap S|/|S|
\end{displaymath} 
for any $A\subseteq G$.

For any measure $\mu$ with finite support on $G$ and $g\in G$, we let $\mu(g):=\mu(\{g\})$.

For any two measures with finite support $\mu,\nu$ on $G$, the convolution of $\mu$ and $\nu$ is denoted by
$\mu\ast\nu$; and so 
\[(\mu\ast\nu)(g)=\sum_{h\in G} \mu(h)\nu(h^{-1}g).\]
The $l$-fold convolution of $\mu$ with itself is denoted by $\mu^{(l)}$, and $\widetilde{\mu}$ denotes the measure such that
$\widetilde{\mu}(g)=\mu(g^{-1}).$

For a measure $\mu$ with finite support on a group $G$ and a group homomorphism $\pi:G\rightarrow H$, we denote by $\pi[\mu]$ the push-forward of $\mu$ under $\pi$; that means $\pi[\mu](\overline{A}):=\mu(\pi^{-1}(\overline{A}))$ for any subset $\overline{A}$ of $H$.

For subsets $A,A_1,\dots, A_n$ of a group $G$, we write 
\[\textstyle \prod_{i=1}^n A_i:=\{a_1a_2\dots a_n|a_i\in  A_i\}\] for the product set of $A_1,\dots, A_n$ and we write
\[\textstyle{\prod_k} A:=\{a_1a_2\dots a_k|a_i\in A,\ 1\le i \le k\}\] for the set consisting of products of $k$ elements of $A$.  We denote by $\bigoplus_{i=1}^k G_i,$ the direct sum of the groups $G_1,\dots, G_k$.  

We use Vinogradov's notation $x\ll_A y$ to mean $|x|<Cy$ for some positive constant $C$ depending on the parameter $A$.  For any constant $\delta$, $K=\Theta_A(\delta)$ means $\delta\ll_A K\ll_A \delta$.  The subscript will be omitted from the above notation if either the constant is universal, or if the dependencies are clear from context.

For any positive integer $n$, $[1..n]$ denotes the set of integers that are at least $1$ and at most $n$.

We use $\pr_i:\bigoplus_{j\in I}  G_j\rightarrow G_i$ to denote the projection to the $i^{\rm th}$ factor.  For $J\subseteq I$, we identify the group $\bigoplus_{i\in J}G_i$ with its natural inclusion in $\bigoplus_{i\in I}G_i$.

For any prime $p$ we let $\overline{\bbf}_p$ be an algebraic closure of a finite field $\bbf_p$ of order $p$. For any prime $p$ and positive integer $n$, $\bbf_{p^n}$ denotes the unique finite subfield of $\overline{\bbf}_p$ that has order $p^n$.

For a field $F$, we let $F^{\times}:=F\setminus \{0\}$. For  $f(t)\in \bbf_q[t]\setminus\{0\}$, we let $N(f):=|\bbf_q[t]/\langle f(t)\rangle|$. For an irreducible polynomial $\ell(t)\in \bbf_q[t]$ we let $v_\ell:\bbf_q(t)\rightarrow \bbz\cup\{\infty\}$ be the $\ell$-valuation; that means for $r\in\bbf_q[t]\setminus \{0\}$ we let $v_\ell(r):=m$ if $\ell^m|r$ and $\ell^{m+1}\nmid r$, $v_{\ell}$ induces a group homomorphism from $\bbf_q(t)^{\times}$ to $\bbz$, and $v_\ell(r)=\infty$ if and only if $r=0$. We let $v_{\infty}$ be the valuation associated to $1/t$; that means $v_{\infty}(r/s):=\deg s-\deg r$ for any $r,s\in \bbf_q[t]\setminus\{0\}$. The set of valuations of $\bbf_q(t)$ is denoted by $V_{\bbf_q(t)}$. For any valuation $v$, the $v$-adic norm of $r\in \bbf_q(t)$ is defined as 
\[
|r|_v:=
\begin{cases}
	N(\ell)^{-v(r)} &\text{ if } v=v_\ell \text{ for some irreducible polynomial } \ell, \\
	q^{-v(r)} &\text{ if } v=v_{\infty}. 
\end{cases}
\] 
For any valuation $v$, the $v$-adic completion of $K:=\bbf_q(t)$ is denoted by $K_v$. The ring of $v$-adic integers is denoted by $\ocal_v$, and the residue field of $K_v$ is denoted by $K(v)$. For an irreducible polynomial $\ell$, we let $K(\ell):=\bbf_q[t]/\langle \ell\rangle\simeq K(v_\ell)$. For any valuation $v$ of $\bbf_q$, we let $\deg v:=[K(v):\bbf_q]$; and so we have $\log_q N(f)=\sum_{v\in V_{\bbf_q(t)}} v(f) \deg v$ for any $f\in \bbf_q[t]$.  For $r_0(t)\in \bbf_q[t]$, we let $D(r_0)$ to be either the set of irreducible factors of $r_0$ or $\{v_\ell\in V_{\bbf_q(t)}|\h \ell \text{ is irreducible, } \ell|r_0\}$. For a finite subset $S$ of valuations of $\bbf_q(t)$, we let $\|r\|_S:=\max_{v\in S}|r|_v$. In this note for $h\in \GL_{n_0}(\bbf_q[t,1/r_0(t)])$, we let $\|h\|:=\max_{i,j} \|h_{ij}\|_{D(r_0)\cup\{v_{\infty}\}}$ where $h_{ij}$ is the $ij$-entry of $h$. We notice that this norm depends on $r_0(t)$, and $r_0(t)$ should be understood from the context.

For a polynomial $r_0\in \bbf_p[t]\setminus\{0\}$ and positive integer $c_0$, we let $S_{r_0,c_0}$ be the set of all square-free polynomials $f(t)\in \bbf_p[t]$ with prime factors $\ell_i(t)$ such that (1) $\ell_i(t)\nmid r_0(t)$, (2) $\deg \ell_i>1$, (3) $\deg \ell_i \neq \deg \ell_j$ if $i\neq j$, and (4) $\deg \ell_i$ does not have a prime factor less than $c_0$.

For a ring $A$, $A^{\times}$ is the group of units of $A$ and $\Spec (A)$ denotes the associated affine scheme; that means the points of this space are prime ideals of $A$. If $\cal$ is a group scheme defined over a ring $A$ and $B$ is an $A$-algebra, then $\cal\otimes_A B$ denotes the group scheme on the fiber product $\cal\times_{\Spec A}\Spec B$.  For a group scheme $\cal$ defined over $A$ and $\ell\in A\setminus A^{\times}$, we let $\cal_{\ell}:=\cal\otimes_A A/\langle \ell\rangle$. For a ring $A$, $(\GL_n)_A$ denotes the $A$-group scheme given by the $n$-by-$n$ general linear group; so $(\GL_n)_A=(\GL_n)_{\bbz}\otimes_{\bbz} A$.

For an algebraic group $\bbg$, $R_u(\bbg)$ denotes its unipotent radical, and $\Lie \bbg$ is its Lie algebra.

\subsection{Outline of proof and the key differences with the characteristic zero case}\label{sectionoutline}
The general architecture of this article is as in Salehi Golsefidy-Varj\'{u}'s work \cite{SGV} where Bourgain-Gamburd's method~\cite{BG1} has been combined with Varj\'{u}'s multi-scale argument~\cite{Var}. By now there are many excellent surveys and lecture notes that explain the key ideas of the ground breaking result of Bourgain and Gamburd (see~\cite{Bre-survey,Hel-survey,Kow-survey,Tao-book}); so here we will be very brief on that part and focus on the main difficulties that were needed to be addressed. 

As in the characteristic zero case, we start with understanding the group structure of $\pi_{f}(\Gamma)$ for a square-free polynomial $f(t)\in \bbf_p(t)$. By Weisfeiler's strong approximation theorem~\cite{Wei}, we have that, if irreducible factors $\ell_i(t)$ of $f(t)$ have large degrees, then \[\pi_f(\Gamma)\simeq \bigoplus_{i=1}^n \bbg_{\ell_i}(\bbf_{N(\ell_i)})\]
 for some absolutely almost simple $\bbf_{N(\ell_i)}$-group $\bbg_{\ell_i}$ of dimension bounded by $n_0^2$. Notice that, since $\bbf_p(t)$ has many subfields, it is inevitable to have an assumption on the trace field of $\Gamma$ to get such a result; this is why we assume that $\bbf_p(t)$ is the field generated by $\Tr(\Ad(\Gamma))$. 
 
 There is a positive number $c_0$ depending on $n_0$ such that all the factors $\bbg_{\ell_i}(\bbf_{N(\ell_i)})$ are $c_0$-quasirandom in the sense of Gowers \cite{Gow}; this implies that for any irreducible representation $\rho$ of $\pi_f(\Gamma)$ we have that $\dim \rho\ge |{\rm Im}\h \rho|^{c_0}$. Based on Sarnak-Xue trick~\cite{SX1} (see~\cite{Gow,NP}), it would be enough to find a good upper bound for the trace of $(T_{\pi_f(\Omega)}^{\circ})^l$ for some positive integer $l=\Theta_{n_0}(\log |\pi_f(\Gamma)|)$. This trace can be controlled in terms of the $L^2$-norm of $\pi_f[\pcal_{\Omega}]^{(l)}$. Following Bourgain-Gamburd's treatment we look at the sequence of $\{\|\pi_f[\pcal_{\Omega}]^{(2^m)}\|_2\}_{m=1}^{\infty}$. It is easy to see that it is a decreasing sequence with a lower bound $\|\pcal_{\pi_f(\Gamma)}\|_2$ (the $L^2$-norm of the probability counting measure on $\pi_f(\Gamma)$). Roughly what Bourgain and Gamburd showed is that, if at some step $\|\pi_f[\pcal_{\Omega}]^{(2^m)}\|_2$ is still not close enough to the lower bound $\|\pcal_{\pi_f(\Gamma)}\|_2$ and does not get significantly smaller in the next step, there should be an algebraic reason: $\pi_f[\pcal_{\Omega}]^{(2^m)}$ should be concentrated on an approximate subgroup $X$; this roughly means $X$ is symmetric and almost close under multiplication.    (we refer the reader to the above cited surveys and lecture notes and \cite{Tao-approximate-subgroup} for a more thorough treatment of this subject; in this note we do not define approximate subgroups as they play an important role only at the background of our arguments). Breakthrough results of Breuillard-Green-Tao~\cite{BGT} and Pyber-Szab\'{o}~\cite{PS} (these generalize works of Helfgott~\cite{Hel1,Hel2}) say that an approximate subgroup of a finite simple group of Lie type with bounded rank is very close to being a subgroup. The multi-scale argument of Varj\'{u}~\cite{Var} gives us an axiomatic way to reduce understanding of approximate subgroups of a finite product of finite groups to the same question for each one of the factors (see~\cite[Section 3]{Var}). One of Varj\'{u}'s assumptions on the factors (see \cite[Condition (A5), section 3]{Var}) demands a type of {\em bounded hierarchy} for the subgroups of factors. The main idea of existence of such bounded hierarchy of subgroups relies on Nori's result~\cite{Nor} which roughly says a subgroup of $\GL_{n_0}(\bbf_p)$ is more or less the $\bbf_p$-points of an algebraic subgroup of $(\GL_{n_0})_{\bbf_p}$; and the mentioned hierarchy comes from the dimension of the associated algebraic subgroup. Clearly this type of statement is not true for subgroups of $\GL_{n_0}(\bbf_{p^d})$ when $d$ gets arbitrarily large; consider $\GL_{n_0}(\bbf_p)\subseteq \GL_{n_0}(\bbf_{p^2})\subseteq \cdots \subseteq \GL_{n_0}(\bbf_{p^d})$. So an important part of this work is to modify Varj\'{u}'s argument to work in our setting (notice that we are presenting the overview of the proof in a backward fashion; and so this part of proof appears towards the end of the article in Section~\ref{s:VarjuProductTheorem}). 
 
 So far under the contrary assumption we have that $\pi_f[\pcal_{\Omega}^{(l_0)}]$ is concentrated on a proper subgroup $H$ of $\pi_f(\Gamma)$ for some positive integer $l_0=\Theta_{n_0}(\log |\pi_f(\Gamma)|)$. Hence we need to have a good understanding of proper subgroups of $\pi_f(\Gamma)$ and escape them in logarithmic number of steps. Here is another important difference with the case of $A=\bbz[1/q_0]$ that we only partially address and is responsible for some of the additional technical assumptions in Theorem~\ref{t:SpectralGap}. For the case of $A=\bbz[1/q_0]$, we have to understand proper subgroups of $\GL_{n_0}(\bbf_\ell)$ where $\ell$ is a prime integer; and as it has been pointed out earlier, by a result of Nori~\cite{Nor} such groups are more or less $\bbf_\ell$-points of an algebraic subgroup. When $A=\bbf_p[t,1/r_0(t)]$, we need to understand subgroups of $\GL_{n_0}(\bbf_{N(\ell)})$ where $\ell(t)\in \bbf_p[t]$ is an irreducible polynomial that does not divide $r_0(t)$. Using work of Larsen and Pink~\cite{LP}, we prove (see Section~\ref{sectiondesubgroupdich}) that if $\bbg_0$ is an absolutely almost simple group of adjoint type defined over a finite field $\bbf_q$ and $H\subseteq\bbg_0(\bbf_q)$ is a proper subgroup, then either there exists a proper algebraic subgroup $\bbh$ (with controlled complexity) of $\bbg_0$ with $H\subseteq \bbh({\bbf}_q)$, or there exists a subfield $\bbf_{q'}$ and a model $\bbg_1$ of $\bbg_0$ defined over $\bbf_{q'}$ (that means we can and will identify $\bbg_1\otimes_{\bbf_{q'}}\bbf_q$ with $\bbg_0$) such that
\[[\bbg_1(\bbf_{q'}):\bbg_1(\bbf_{q'})]\subseteq H\subseteq \bbg_1(\bbf_{q'}).\]  
Subgroups of the former type are called {\em structural subgroups} while subgroups of the latter type are called {\em subfield type subgroups}. Currently we do not know how to escape subfield type subgroups, and this is why we need to add the extra technical assumptions on the largeness of prime factors of the degree of irreducible factors $\ell_i$ of $f$ in Theorem~\ref{t:SpectralGap}.

 In order to escape structural subgroups, we use similar ideas as in Salehi Golsefidy-Varj\'{u} \cite{SGV}; but since representations of a simple group over a positive characteristic field are not necessarily completely reducible, we face extra difficulties that need to be resolved. 

To be more precise we show that there is a polynomial $r_1(t)$ depending on $\Omega$ such that, if  $f(t)\in\bbf_p[t]$ is a square-free polynomial and $\gcd(f,r_1)=1$, then (1) $\pi_f(\Gamma)=\prod_{i=1}^n \pi_{\ell_i}(\Gamma)$ where $\ell_i$'s are irreducible factors of $q$ and $\pi_{\ell_i}(\Gamma)\simeq \bbg_{\ell_i}(\bbf_{N(\ell_i)})$ for an absolutely almost simple $\bbf_{N(\ell_i)}$-groups $\bbg_{\ell_i}$; (2) if $H\subseteq \pi_f(\Gamma)$ is a proper subgroup such that $\pi_{\ell_i}(H)$ is a structural subgroup of $\bbg_{\ell_i}(\bbf_{N(\ell_i)})$ for any $i$, then the set of {\em small lifts} of $H$,
\[\lcal_\delta(H):=\{h\in \bbg(\bbf_{p}[t,1/r_0(t)])\mid \pi_f(h)\in H\mbox{ and }\|h\|<[G:H]^\delta\}\]
 is contained in a proper algebraic subgroup of $\bbg$, where $\|h\|:=\max_{ij} \|h_{ij}\|_{D(r_1)\cup \{v_{\infty}\}}$ (when $\delta$ is small enough depending on $\Omega$). So we can escape a proper subgroup $H$ of $\pi_f(\Gamma)$ where $\pi_{\ell_i}(H)$ are structural subgroups if we manage to escape proper algebraic subgroups of $\bbg$. Following \cite{SGV}, we show that there are finitely many non-trivial irreducible representations $\{\rho_i:\bbg\rightarrow \GL(\bbv_i)\}_{i=1}^m$ and affine representations $\{\rho_j':\bbg\rightarrow {\rm Aff}(\bbw_j)\}_{j=1}^{m'}$ of $\bbg$ such that (1) the linear part of $\rho_j'$ is non-trivial and irreducible, (2) $\bbg(\overline{\bbf_p(t)})$ does not fix any point of $\bbw_j(\overline{\bbf_p(t)})$, (3) for any proper algebraic subgroup $\bbh$ of $\bbg$ there are either $i$ and $v\in \bbv_i(\overline{\bbf_p(t)})$ such that $\rho_i(\bbh(\overline{\bbf_p(t)}))[v]=[v]$ where $[v]$ is the line in $\bbv_i(\overline{\bbf_p(t)})$ that is spanned by $v$ or $j$ and $w\in \bbw_j(\overline{\bbf_p(t)})$ which is fixed by $\bbh(\overline{\bbf_p(t)})$ (see Proposition~\ref{propreps}). Notice that, since the representation $\wedge^{\dim\bbh}\Ad$ is not necessarily completely reducible, we had to use affine representations even for the case where $\bbg$ is (semi)simple;  this is an issue that can occur only in the positive characteristic case. Having this result we can apply the same {\em ping-pong} type argument as in~\cite[Proposition 21]{SGV} and find a finite symmetric subset $\Omega'$ of $\Gamma$ such that very few words in terms of $\Omega'$ fix a line in one of the irreducible representations $\rho_i$; and then we deduce that $\pcal_{\Omega'}^{(l)}(\bbh(\bbf_p(t)))\le e^{-O_{\Omega}(l)}$ for any proper algebraic subgroup $\bbh$ of $\bbg$. In order to be able to use $\Omega'$ instead of $\Omega$, we have to make sure that $\pi_f(\langle \Omega'\rangle)=\pi_f(\Gamma)$ when irreducible factors of $f$ have large degree. Unfortunately at this point, we cannot do this; and here is another place that the technical assumption on the largeness of prime divisors of the degree of irreducible factors of $f$ is needed. We suspect that this condition should not be needed here and the answer to the following question should be affirmative. 

 \begin{question}\label{ques:AdelicToplogicalTitsAlternative}
 	Let $\Omega$, $\Gamma$, and $\bbg$ be as in the hypotheses of Theorem~\ref{t:SpectralGap}. Let $K:=\bbf_{p}(t)$, $V_K$ be the set of valuations of $K$, $\ocal_v$ be the ring of integers of the completion $K_v$ of $K$ with respect to a valuation $v$, and $D(r_0):=\{v_{\ell}|\h \ell \text{ is an irreducible factor of } r_0\}$. Let $\overline{\Gamma}$ be the closure of $\Gamma$ in $\prod_{v\in V_K\setminus (D(r_0)\cup\{v_{\infty}\})}\GL_{n_0}(\ocal_v)$. Then there is a finite subset $\Omega'_0$ of $\Gamma$ such that
 	\begin{enumerate}
 	\item $\Omega'_0$ freely generates a subgroup $\Gamma'$ of $\Gamma$.
 	\item The closure $\overline{\Gamma'}$ of $\Gamma'$ in $\overline{\Gamma}$ is open.
 	\item For any proper algebraic subgroup $\bbh$ of $\bbg$ we have $\pcal_{\Omega'}^{(l)}(\bbh(\bbf_{p}(t)))\le e^{-O_{\Omega}(l)}$  where $\Omega':=\Omega'_0\sqcup \Omega_0'^{-1}$.	
 	\end{enumerate}
 \end{question}
It is worth mentioning that we do find $\Omega_0'$ that satisfies (1) and (3); but we cannot make sure that the trace field of $\Gamma'$ would be still $\bbf_{p}(t)$. Hence strong approximation does not imply (2). This issue does not occur over $\bbq$ as it does not have any non-trivial subfield. 

Overall we get the following result.

\begin{prop}[Escape from proper subgroups]\label{propmainescape} Let $\Omega$, $\Gamma$, and $\bbg$ be as in the hypotheses of Theorem \ref{t:SpectralGap}.  Then there is a symmetric set $\Omega'\subset \Gamma$, a square free polynomial $r_1$ divisible by $r_0$, and constants $c_0$ and $\delta_0$ depending only on $\Omega$ such that the following holds:\newline
	\indent For $f\in S_{r_1,c_0}$, suppose $H\le \pi_f(\Gamma)$ is a proper subgroup with the property that $\pi_\ell(H)$ is a structural subgroup of $\pi_\ell(\Gamma)$ for every irreducible factor $\ell$ of $f$.  Then for $l\gg_{\Omega} \deg f$ we have
	\[\pi_f[\convolve{\pcal}{l}_{\Omega'}](H)\le [\pi_f(\Gamma):H]^{-\delta_0}\text{, and } \pi_f(\langle \Omega'\rangle)=\pi_f(\Gamma).\]
\end{prop}
As you can see using Proposition~\ref{propmainescape} we can only show escape from proper subgroups with {\em structural factors}. 
On the other hand, roughly speaking an arbitrary proper subgroup $H$ can be embedded into a product of two groups, one with structural factors and the other with subfield subgroup factors. The extra technical condition on the largeness of prime factors of degrees of irreducible factors of $f$ implies that the subgroup with subfield factors is relatively small; so it can be disregarded, and we get the desired result.

\section*{Acknowledgments}
We would like to thank P. Varj\'{u} and M. Larsen for their quick replies to our questions in regard to their works. The second author is thankful to A. Mohammadi for many mathematical discussions related (and unrelated) to random walks in compact groups.     

\section{A refinement of a theorem by Larsen and Pink}\label{s:RefinementOfLarsenPink}
In this section, we point out how Larsen and Pink's work~\cite{LP} gives us a concrete understanding of proper subgroups of $\pi_f(\Gamma)$ where $\Gamma\subseteq\GL_{n_0}(\bbf_p[t,1/r_0(t)])$ is as in Theorem~\ref{t:SpectralGap} (see Theorem~\ref{thm:FinalRefinementOfLP}). To avoid referring reader to the {\em ideas} in that article, we present an argument that uses only a couple of results from \cite{LP} as a black-box. That said it is worth pointing out that most of the results in this section are hidden in the mentioned Larsen-Pink work. 

\subsection{General setting and strong approximation}\label{ss:storngapproximation} 
Let $\Omega\subset\GL_{n_0}(\bbf_{q_0}(t))$ be a finite symmetric set, and let $\Gamma=\grpgen{\Omega}$.  Since $\Omega$ is finite, there exists a square-free polynomial $r_0\in\bbf_{q_0}[t]$ such that $\Omega\subset\GL_{n_0}(\bbf_{q_0}[t,1/r_0(t)])$.  The set of polynomials in $n_0^2$ variables with coefficients in $\bbf_{q_0}(t)$ which vanish on $\Gamma$ define a flat group scheme $\gcal$ of finite type over $\bbf_{q_0}[t,1/r_0]$.  The Zariski closure $\bbg$ of $\Gamma$ in $(\bbgl_{n_0})_{\bbf_{q_0}(t)}$ can be viewed as the generic fiber 
\begin{equation}
\label{grpsch}\gcal\otimes_{\bbf_{q_0}[t,1/r_0]}\bbf_{q_0}(t)\end{equation}
of $\gcal$. After possibly passing to a multiple of $r_0$, we may assume $\gcal$ is a smooth group scheme over $\bbf_{q_0}[t,1/r_0]$ and that all of its fibers are of constant dimension.
For any polynomial $f\in\bbf_{q_0}[t]$ that is coprime to $r_0$, we let $\gcal_f:=\gcal\otimes_{\bbf_{q_0}[t,1/r_0]}\bbf_{q_0}[t]/\langle f\rangle$; and the {\em reduction modulo $f$ homomorphism} is denoted by $\pi_f:\gcal(\bbf_{q_0}[t,1/r_0])\rightarrow \gcal_f(\bbf_{q_0}[t]/\langle f\rangle).$

For an irreducible polynomial $\ell$ which does not divide $r_0$, let $K(\ell):=\bbf_{q_0}[t]/\langle \ell\rangle$. Then $\gcal_{\ell}$ is an absolutely almost simple $K(\ell)$-group; and possibly after passing to a multiple of $r_0$, we can and will assume that all $\gcal_{\ell}\otimes_{K(\ell)} \overline{K(\ell)}$ are of the same type $\Phi$ as $\ell$ ranges through irreducible polynomials in $\bbf_{q_0}[t]$ that do not divide $r_0$; this means there is an adjoint Chevalley $\bbz$-group scheme $\gcal^{\rm Che}$ (we refer the reader to \cite{Ste} for a thorough treatment of Chevalley group schemes) such that for any irreducible $\ell$ that does not divide $r_0$ we have a central isogeny 
\[
\gcal_{\ell}\otimes_{K(\ell)} \overline{K(\ell)}\rightarrow 
\gcal^{\rm Che}\otimes_{\bbz} \overline{K(\ell)}.
\]
By Weisfeiler's strong approximation theorem \cite[Theorem 1.1]{Wei}, after possibly passing to a multiple of $r_0$, we have that if $f$ is a square-free polynomial coprime to $r_0$, then
\begin{equation}
\label{eqnSAsurj}
\pi_f(\Gamma)=\gcal_f(\bbf_{q_0}[t]/\langle f\rangle);\end{equation}
and by the Chinese Remainder Theorem $\bbf_{q_0}[t]/\langle f\rangle\simeq \bigoplus_{\ell\mid f, \ell \text{ irred.}} K(\ell)$, which implies
\begin{equation}
\label{eqnSAprod}
\gcal_f(\bbf_{q_0}[t]/\langle f\rangle)\simeq \prod_{\ell\mid f, \ell \text{ irred.}}\gcal_{\ell}(K(\ell)).
\end{equation}
Throughout this paper, we may replace $r_0$ by the product of all irreducible polynomials of degree at most $C$ in $\bbf_{q_0}[t]$ for some $C\ll_\Omega 1$ as necessary.  For the remainder of this section, $f$ is a fixed square-free polynomial coprime to $r_0$.

In order to prove Proposition~\ref{propmainescape} we must understand proper subgroups of $\pi_f(\Gamma)$.  In light of (\ref{eqnSAprod}) and (\ref{eqnSAsurj}), we must study proper subgroups of $\gcal_\ell(K(\ell))$ as $\ell$ ranges through all irreducible factors of $f$.

\subsection{The dichotomy of proper subgroups of $\gcal_{\ell}(K(\ell))$}\label{sectiondesubgroupdich} In this section the mentioned theorem of Larsen-Pink is stated and based on that we define structure type and subfield type subgroups.

Let $\bbt$ be a maximal torus of $\bbg$ and let $L$ be a minimal splitting field of $\bbt$.  Then $L $ is a finite extension of $\bbf_{q_0}(t)$ of degree say $D'$.  Let $\gcal^{Che}$ be the adjoint Chevalley $\bbz$-group scheme of the same type $\Phi$ as $\bbg\otimes_{K} L$, where $K:=\bbf_{q_0}(t)$.  Then there exists a central $L$-isogeny
\begin{equation*} \bbg\otimes_{K}L\rightarrow \gcal^{\rm Che}\otimes_\bbz L.
\end{equation*}
After passing to a multiple of $r_0$, if needed, we can extend this isogeny to a central  $\ocal_L[1/r_0]$-isogeny
\begin{equation*}\label{eqnglobliso}
\phi:\gcal\otimes_{\bbf_{q_0}[t,1/r_0]}\ocal_L[1/r_0]\rightarrow \gcal^{\rm Che}\otimes_\bbz\ocal_L[1/r_0]
\end{equation*}
where $\ocal_L$ is the integral closure of $\bbf_{q_0}[t]$ in $L$. For an irreducible polynomial $\ell$ coprime to $r_0$, let $\lfr\in \Spec(\ocal_L)$ be in the fiber over $\langle \ell\rangle$; that means $\lfr\cap \bbf_{q_0}[t]=\langle \ell\rangle$. Then $K(\ell):=\bbf_{q_0}[t]/\langle \ell\rangle$ can be embedded into $L(\lfr):=\ocal_L/\lfr$, and 
 \[[L(\lfr):K(\ell)]\le [L:K]\ll_\bbg 1.\]  Hence, we obtain an induced central $L(\lfr)$-isogeny
\begin{equation*}\label{phip}\phi_\ell:\gcal_\ell\otimes_{K(\ell)} L(\lfr)\rightarrow \gcal^{\rm Che}\otimes_{\bbz} L(\lfr).
\end{equation*} 

With this preparation, we mention a theorem of Larsen and Pink which is key in understanding proper subgroups of $\gcal_\ell(K(\ell))$. 

\begin{thm}{\cite[Theorem 0.6]{LP}}\label{thmLP}  Let $\gcal_0^{\rm Che}$ be an adjoint Chevalley $\bbz$-group scheme with simple root system $\Phi_0$.  Then there exists a representation
\[\rho:\gcal_0^{\rm Che}\rightarrow (\GL_{n'_0})_{\bbz}\] with the following property: Let $H$ be a finite subgroup of $\gcal_{0,p}^{\rm Che}(\overline{\bbf}_p)$ where $\gcal_{0,p}^{\rm Che}=\gcal_0^{\rm Che}\otimes_\bbz \overline{\bbf}_p$ is the geometric fibre of $\gcal_0^{\rm Che}$ over $p$ where $p$ is a prime more than 3.  Then either there exists a proper subspace $W\subset (\overline{\bbf}_p)^{n_0'}$ that is stable under $\rho(H)$ but not $\rho(\gcal_{0,p}^{\rm Che}(\overline{\bbf}_p))$, or there exists a finite field $\bbf_q\subset\overline{\bbf}_p$ and a model $\bbg_0$ of $\gcal_{0,p}^{\rm Che}$ over $\bbf_q$ (that means an $\bbf_q$-group $\bbg_0$ such that $\bbg_0\otimes_{\bbf_q}\overline{\bbf}_{p}\cong \gcal_{0,p}^{\rm Che}$) such that the commutator subgroup of $\bbg_0(\bbf_q)$ is simple and 
\begin{equation}\label{eqnsubfieldsubs}[\bbg_0(\bbf_q):\bbg_0(\bbf_q)]\subseteq H\subseteq \bbg_0(\bbf_q).\end{equation}
\end{thm}  
\begin{definition}
Subgroups that satisfy the first condition are said to be of {\em structural type} while subgroups that satisfy the latter condition are said to be of {\em subfield type}.  

If for an irreducible polynomial $\ell$ that does not divide $r_0$, $H\subseteq \pi_\ell(\Gamma)\simeq\gcal_\ell(K(\ell))$ is a subgroup such that $\phi_\ell(H)$ is a subfield type subgroup (resp. structural type subgroup) of $\gcal_p^{\rm Che}(\overline{K(\ell)})$, then we call $H$ a {\em subfield} (resp. {\em structural}) {\em type subgroup} of $\pi_\ell(\Gamma)$.\end{definition}  

\subsection{Refiner description of subfield type subgroups of $\gcal_{\ell}(K(\ell))$}

In this section, we focus on subfield type subgroups of $\gcal_{\ell}(K(\ell))$; and we get a connection between the model $\bbg_0$ given in Theorem \ref{thmLP} and $\gcal_{\ell}$.  We prove a stronger result (see Proposition~\ref{propLPsubfields}) which is of independent interest. 
 
A subfield type subgroup $H$ of $\gcal^{\rm Che}(\overline{\bbf}_p)$ gives us a finite field $F_H$ and a model $\bbg_H$ of $\gcal^{\rm Che}\otimes_{\bbz} \overline{\bbf}_p$ over $F_H$. Proposition~\ref{propLPsubfields} implies that if $H_1\subseteq H_2$ are two subfield subgroups of $\gcal^{\rm  Che}(\overline{\bbf}_p)$ and $p$ is large enough, then $F_{H_1}\subseteq F_{H_2}$ and $\bbg_{H_1}$ is a model of $\bbg_{H_2}$ over $F_{H_1}$. This statement can be proved by the virtue of the argument given by Larsen and Pink. Here we give an independent self-contained proof. 

\begin{prop}\label{propLPsubfields}  For $i=1,2$, let $\bbg_i$ be an absolutely almost simple group defined over a finite field $\bbf_{q_i}\subseteq \overline{\bbf}_p$.  Suppose $\bbf_{q_i}$'s are of characteristic $p>5$, $q_1>9$, and that $\bbg_2$ is of adjoint type.  Suppose $\widetilde{\theta}:\bbg_{1}\otimes_{\bbf_{q_1}}\overline{\bbf}_{p}\rightarrow \bbg_{2}\otimes_{\bbf_{q_2}}\overline{\bbf}_{p}$ is an isogeny such that 
\[\widetilde{\theta}(\bbg_{1}(\bbf_{q_1}))\subseteq \bbg_2(\bbf_{q_2}).\]
Then $\bbf_{q_1}\subseteq \bbf_{q_2}$ and there exists an isogeny $\theta:\bbg_1\otimes_{\bbf_{q_1}} \bbf_{q_2}\rightarrow \bbg_2$ such that $\theta\otimes \id_{\overline{\bbf}_{q_1}}=\widetilde{\theta}$.
\end{prop}

 By a theorem of Lang \cite[Thm. 35.2]{Hum}, $\bbg_1$ is quasisplit; that means $\bbg_1$ has a Borel subgroup $\bbb_1$ defined over $\bbf_{q_1}$. By \cite[\textsection 6.5 (3)]{Bor}, there is a maximal $\bbf_{q_1}$-split torus $\bbs_1$ such that \[\bbb_1=C_{\bbg_1}(\bbs_1)\cdot R_u(\bbb_1),\] where $R_u(\bbb_1)$ is the unipotent radical of $\bbb_1$. Since $\bbb_1$ is a Borel subgroup, $\bbt_1=C_{\bbg_1}(\bbs_1)$ is a maximal torus. Since $\bbs_1$ is defined over $\bbf_{q_1}$, $\bbt_1$ is defined over $\bbf_{q_1}$.

Let $\widetilde{\bbg}_i=\bbg_i\otimes_{\bbf_{q_i}}\overline{\bbf}_p$ for $i=1,2$, $\widetilde{\bbs}_2=\widetilde{\theta}(\bbs_1\otimes_{\bbf_{q_1}} \overline{\bbf}_p)$, $\widetilde{\bbt}_2=\widetilde{\theta}(\bbt_1\otimes_{\bbf_{q_1}} \overline{\bbf}_p)$, and $\widetilde{\bbb}_2=\widetilde{\theta}(\bbb_1\otimes_{\bbf_{q_1}} \overline{\bbf}_p)$.  Let $\ugfr_i=\Lie(\bbg_i)$ for $i=1,2$; it is worth pointing out that we view $\ugfr_i$'s as functors from $\bbf_{q_i}$-algebras to Lie $\bbf_{q_i}$-algebras, and since $\bbg_i$'s are smooth $\bbf_{q_i}$-group schemes, $\ugfr_i(A)$ is naturally isomorphic to $\gfr_i\otimes_{\bbf_{q_i}} A$ where $\gfr_i:=\ugfr_i(\bbf_{q_i})$.  Notice that since $\widetilde{\theta}$ is an isogeny, we have an isomorphism

\[d\widetilde{\theta}: \ugfr_1( \overline{\bbf}_p)\rightarrow \ugfr_2( \overline{\bbf}_p)\] which satisfies the identity
\begin{equation}\label{adjaction}d\widetilde{\theta}(\Ad(g_1)(x_1))=\Ad(\widetilde{\theta}(g_1))(d\widetilde{\theta}(x_1)),\end{equation} for all $g_1\in \bbg_1(\overline{\bbf}_p)$ and  $x_1\in \ugfr_1(\overline{\bbf}_p)$.
By \cite[Cors. 9.2, 11.12]{Bor} and \cite[A.2.8]{CGP}, we have $\widetilde{\theta}(C_{\widetilde{\bbg}_1}(\bbs_1\otimes_{\bbf_{q_1}} \overline{\bbf}_p))=C_{\widetilde{\bbg}_2}(\widetilde{\bbs}_2)$ is an $\overline{\bbf}_p$-torus, $\widetilde{\bbt}_2=C_{\widetilde{\bbg}_2}(\wt{\bbs}_2)$ is a maximal $\overline{\bbf}_p$-torus, and $\Lie(\bbt_1)=C_{\ugfr_1}(\bbs_1)$.

For a torus $\bbs$ defined over a perfect field $F$, let $X^*(\bbs)$ be the group of characters of $\bbs$; that means $\Hom(\bbs\otimes_F \overline{F},(\GL_1)_{\overline{F}})$. It is well-known that $X^*(\bbs)$ is isomorphic to $\bbz^{\dim \bbs}$ as an abelian group and the absolute Galois group ${\rm Gal}(\overline{F}/F)$ acts linearly on $X^*(\bbs)$ (see \cite[Chapter III, \textsection 8]{Bor}). Suppose $\bbs$ is a subgroup of an algebraic group $\bbh$; then $\Phi(\bbh,\bbs)\subseteq X^*(\bbs)$ denotes the set of roots of $\bbh$ relative to $\bbs$. For $\alpha\in \Phi(\bbh,\bbs)$, let \[\underline{\hfr}_{\alpha}(A):=\{x\in \underline{\hfr}(A)|\h \forall s\in \bbs(A), \Ad(s)(x)=\alpha(s) x\}\] be the root space associated with $\alpha$. We notice that if $\alpha$ is defined over $F$ (this is equivalent to saying $\alpha$ is invariant under the action of the absolute Galois group ${\rm Gal}(\overline{F}/F)$), then $\underline{\hfr}_{\alpha}$ is defined over $F$.

Let us also recall that, $\wt{\theta}$ induces injective group homomorphism from $\wt{\theta}^*:X^*(\wt{\bbs}_2)\rightarrow X^*(\wt{\bbs}_1)$ and $\wt{\theta}^*:X^*(\wt{\bbt}_2)\rightarrow X^*(\wt{\bbt}_1)$.
\begin{lem}  $\widetilde{\theta}^\ast$ induces bijections
$\Phi(\widetilde{\bbg}_2,\widetilde{\bbs}_2)\rightarrow \Phi(\widetilde{\bbg}_1,\widetilde{\bbs}_1)$
and
$\Phi(\widetilde{\bbg}_2,\widetilde{\bbt}_2)\rightarrow \Phi(\widetilde{\bbg}_1,\widetilde{\bbt}_1).$
Moreover, $d\widetilde{\theta}$ induces isomorphisms
$\gfr_{1,\widetilde{\theta}^\ast\alpha}(\overline{\bbf}_{p})\rightarrow \gfr_{2,\alpha}(\overline{\bbf}_{p})$ for $\alpha\in \Phi(\widetilde{\bbg}_2,\widetilde{\bbs}_2)$ or $\Phi(\widetilde{\bbg}_2,\widetilde{\bbt}_2)$.

\end{lem}
\begin{proof}  We notice that the root space decomposition of $\underline{\gfr}_2$ relative to $\wt{\bbs}_2$ gives us
\[\underline{\gfr}_2(\overline{\bbf}_{q_2})=\underline{\tfr}_2(\overline{\bbf}_{q_2})\oplus\left(\oplus_{\beta\in\Phi(\widetilde{\bbg}_2,\widetilde{\bbs}_2)} \underline{\gfr}_{2,\beta}(\overline{\bbf}_{q_2})\right).\]  
Suppose $x_{2,\alpha}\in \underline{\gfr}_{2,\alpha}(\overline{\bbf}_p)$, and $x_1\in \underline{\gfr}_1(\overline{\bbf}_p)$ such that $d\widetilde{\theta}(x_1)=x_{2,\alpha}$.  By \eqref{adjaction}
 for every $s_1\in \widetilde{\bbs}_1(\overline{\bbf}_p)$ and $\alpha\in \Phi(\widetilde{\bbg}_2,\widetilde{\bbs}_2)$, we have \[d\widetilde{\theta}(\Ad(s)(x_1))=\Ad(\widetilde{\theta}(s_1))(d\widetilde{\theta}(x_1))=(\widetilde{\theta}^\ast\alpha)(s_1)d\widetilde{\theta}(x_1)=d\widetilde{\theta}((\widetilde{\theta}^\ast\alpha)(s_1)x_1);\] this implies $(\widetilde{\theta}^\ast\alpha)(s)x_1=\Ad(s)x_1$ as $d\widetilde{\theta}$ is an isomorphism.  Therefore, $\widetilde{\theta}^\ast(\alpha)\in\Phi(\widetilde{\bbg}_1,\widetilde{\bbs}_1)$ and $d\widetilde{\theta}(\ugfr_{q,\widetilde{\theta}^\ast(\alpha)}(\overline{\bbf}_p))\subseteq\ugfr_{2,\alpha}(\overline{\bbf}_p)$.  By comparing dimensions, we see that $\widetilde{\theta}^\ast$ induces a bijection from $\Phi(\widetilde{\bbg}_2,\widetilde{\bbs}_2)$ to $\Phi(\widetilde{\bbg}_1,\widetilde{\bbs}_1)$ and $d\widetilde{\theta}$ induces an isomorphism from $\ugfr_{1,\widetilde{\theta}^\ast(\alpha)}(\overline{\bbf}_p)$ to $\ugfr_{2,\alpha}(\overline{\bbf}_p)$.  The argument is similar for the second assertion.
\end{proof}

\begin{lem}\label{l:DimRootSpace}  For every $\alpha\in \Phi(\widetilde{\bbg}_1,\widetilde{\bbs}_1)$, $\dim\ugfr_{1,\alpha}\le 3$.
\end{lem}
\begin{proof}
Let's recall that if $\bbh$ is a quasi-split absolutely almost simple $k$-group, then there is a Galois extension $l$ of $k$ such that $\bbh\otimes_k l$ is a split group and ${\rm Gal}(l/k)$ can be embedded into the group of symmetries of the Dynkin diagram of $\bbh\otimes_k l$; in particular, $\Gal(l/k)$ is isomorphic to $\{1\}, \bbz/2\bbz, \bbz/3\bbz$, or $S_3$. By Lang's theorem \cite[Thm 35.2]{Hum}, $\bbg_i$ is quasisplit over $\bbf_{q_i}$ for $i=1,2$. Therefore by the above discussion and the fact that finite extensions of $\bbf_{q_1}$ are cyclic, we have that there is a Galois extension $F_1$ of $\bbf_{q_1}$ such that $\bbg_1\otimes_{\bbf_{q_1}} F_1$ splits and $|\Gal(F_1/\bbf_{q_1})|\le 3$. For each $\alpha\in \Phi(\widetilde{\bbg}_1,\widetilde{\bbs}_1)$, we have that
\[\dim\ugfr_{1,\alpha}=|\{\widetilde{\alpha}\in \Phi(\widetilde{\bbg}_1,\widetilde{\bbt}_1)|\h\widetilde{\alpha}\big|_{\widetilde{\bbs}_1}=\alpha\}|\] and $\Gal(F_1/\bbf_{q_1})$ acts transitively on the set \[\{\widetilde{\alpha}\in \Phi(\widetilde{\bbg}_1,\widetilde{\bbt}_1)|\h\widetilde{\alpha}\big|_{\widetilde{\bbs}_1}=\alpha\}\] which implies the lemma (see \cite[Proposition 15.5.3]{Spr}).
\end{proof}

\begin{prop}\label{prop:subfield} In the above setting, if $q_1>9$, then $\bbf_{q_1}\subseteq\bbf_{q_2}$.
\end{prop}

\begin{proof}Let $\{\alpha_1,\alpha_2,\dots, \alpha_r\}$ be a set of simple roots of $\bbs_1$, and $\{\alpha_1^\vee,\dots,\alpha_r^\vee\}$ the corresponding coroots.  Then for any $t_1,t_2,\dots, t_r\in \bbf_{q_1}$ , \[\Tr(\Ad(\wt{\theta}(\Pi_{i=1}^r\alpha_i^\vee(t_i))))\in \bbf_{q_2}\] since $\widetilde{\theta}(\bbg_1(\bbf_{q_1}))\subseteq \bbg_2(\bbf_{q_2})$.  On the other hand, by \eqref{adjaction}, we have
\[
\Tr(\Ad(\Pi_{i=1}^r\alpha_i^\vee(t_i)))=
\Tr(\Ad(\wt{\theta}(\Pi_{i=1}^r\alpha_i^\vee(t_i))));
\]
and so
\be\label{eq:polynomial}
\sum_{\beta\in\Phi(\bbg_1,\bbs_1)}\dim\ugfr_{1,\beta}\Pi_{i=1}^rt_1^{\langle\alpha_i^\vee,\beta\rangle}\in \bbf_{q_2}.
\ee
Notice that for each $i=1,2,\dots, r$, and any root $\beta$,  $\langle\alpha_i^\vee,\beta\rangle$ is a Cartan integer and hence is at most $3$ in absolute value.  By Lemma~\ref{l:DimRootSpace}, $\dim\ugfr_{1,\beta}\le3$.  The proposition will be proved with the following series of lemmas:

\begin{lem}\label{lem12} Suppose $f(t)\in \bbf_p[t^{\pm 1}]$ is a nonconstant polynomial, $(\deg_{t}f+\deg_{t^{-1}}f)^2<q$, and $f(\bbf_q)\subseteq \bbf_{q'}$; then $\bbf_q\subseteq \bbf_{q'}$.
\end{lem}
\begin{proof}  For each $a\in \bbf_{q'}$, there are at most $(\deg_tf+\deg_{t^{-1}}f)$ elements $b\in \bbf_q$ such that $f(b)=a$.  Hence, $|f(\bbf_q)|\ge q/(\deg_tf+\deg_{t^{-1}}f)$. Suppose $F$ is the field generated by $f(\bbf_q)$; then $\log_p|F|$ divides $\log_pq$ and $\log_pq\le \log_p|F|+\log_p(\deg_tf+\deg_{t^{-1}}f)$.  If $F\neq \bbf_q$, then the above argument implies $(1/2)\log_pq\le \log_p(\deg_tf+\deg_{t^{-1}}f)$.  This contradicts the assumption that $q>(\deg_tf+\deg_{t^{-1}}f)^2$.
\end{proof}

\begin{lem}\label{l:nonzero-value}
Suppose $f\in \bbf_p[t_1^{\pm1},\dots, t_r^{\pm1}]$ is a nonzero polynomial and \[\max_i(\deg_{t_i}f+\deg_{t_i^{-1}}f)+1<q;\] then $f(\bbf_{q}^\times,\dots, \bbf_q^\times)\neq 0$.
\end{lem}

\begin{proof}This can easily be proved by induction on $r$.
\end{proof}

\begin{lem}\label{l:ImageMultivariablePolynomial}  Suppose $f\in \bbf_p[t_1^{\pm1},\dots, t_r^{\pm1}]$ is a nonzero polynomial such that $f(\bbf_q^\times,\dots, \bbf_q^\times)$ is contained in $\bbf_{q'}$, and $\max_i(\deg_{t_i}f+\deg_{t_i^{-1}}f)^2<q$; then $\bbf_q\subseteq\bbf_{q'}$.
\end{lem}
\begin{proof}  Since $f$ is nonconstant, there exists some index $i_0$ where $\deg_{t_{i_0}^{\pm1}}f\neq 0$.  Without loss of generality we can and will assume that $i_0=r$. By Lemma~\ref{l:nonzero-value}, there is a choice of constants $a_1,\dots, a_{r-1}\in \bbf_q^{\times}$ such that $f(a_1,\dots, a_{r-1},t_r)$ is a nonconstant polynomial in $t_r$.  By Lemma \ref{lem12}, we are done.
\end{proof}
Proposition~\ref{prop:subfield} follows from \eqref{eq:polynomial} and Lemma~\ref{l:ImageMultivariablePolynomial}.
\end{proof}

We must now prove the existence of the isogeny $\theta$. 

\begin{prop}\label{propliedescent}  If $q_1>7$ and $p>5$, then $d\widetilde{\theta}$ induces an isomorphism between $\ugfr_1(\bbf_{q_2})$ and $\ugfr_2(\bbf_{q_2})$.
\end{prop}
We distinguish two cases depending on whether or not $\ugfr_1$ has a nontrivial center.
\begin{lem}\label{lem11}  Let $p>5$.  Suppose $\bbg$ is an absolutely almost simple $\bbf_q$-group and that $\bbg$ is not of type $A_{np-1}$ for some positive integer $n$.  Assume:
\begin{enumerate}
\item $M\subseteq\ugfr(\overline{\bbf}_p)$ is an $\bbf_{q'}$-subspace where $\bbf_q\subseteq \bbf_{q'}$,
\item $\dim_{\bbf_{q'}}M=\dim_{\overline{\bbf}_p}\ugfr(\overline{\bbf}_p)$, and
\item $M$ is $\bbg(\bbf_q)$-invariant.
\end{enumerate}
Then there exists $0\neq\lambda\in\overline{\bbf}_p$ such that $M=\lambda\ugfr(\bbf_{q'})$.
\end{lem}

\begin{proof} Since $\bbg$ is not of type $A_{np-1}$, $\ugfr(\overline{\bbf}_p)$ is a simple $\bbg(\overline{\bbf}_p)$-module.  By \cite[Corollary 4.6]{Wei}, $\ugfr(\overline{\bbf}_p)$ is a simple $\bbg(\bbf_q)$-module.  Let $\{\alpha_i\}$ be an $\bbf_{q'}$-basis of $\overline{\bbf}_p$; so we have
\[\ugfr(\overline{\bbf}_p)=\oplus_{i\ge0}\alpha_i\ugfr(\bbf_{q'}).\]  Let $\pr_i:M\rightarrow \alpha_i\ugfr (\bbf_{q'})$ be the projection morphism onto the $i^{th}$ component.  Since $M$ and $\ugfr(\bbf_{q'})$ are both $\bbg(\bbf_q)$-invariant, $\pr_i$ is an $\bbf_{q'}$-linear $\bbg(\bbf_{q})$-module homomorphism.  Again by \cite[Cor. 4.6]{Wei}, $\ugfr(\bbf_{q'})$ is a simple $\bbf_{q'}[\Ad(\bbg(\bbf_q))]$-module and hence $\pr_i$ is either trivial or surjective for each $i$.  Since $\dim_{\bbf_{q'}}M=\dim_{\bbf_{q'}}\ugfr(\bbf_{q'})$, either $\pr_i=0$ or $\pr_i$ is an isomorphism.

If $\pr_i$ and $\pr_j$ are isomorphisms, then $\pr_i\circ\inv{\pr_j}\in\Aut_{\bbg(\bbf_q)\text{-Mod}}(\ugfr(\bbf_{q'}))$.  Then there exists a nonzero element $\alpha_{i,j}\in\bbf_{q'}$ such that $\pr_i\circ\inv{\pr_j}(x)=\lambda_{i,j}x$ for all $x\in \ugfr(\bbf_{q'})$.  Hence if $j_0$ is a fixed index for which $\pr_{j_0}$ is an isomorphism, we have \[M=\left(\sum_i\alpha_i\lambda_{i,j_0}\right)\ugfr(\bbf_{q'}).\]
\end{proof}
In the case when $\bbg$ is of type $A_{np-1}$ we have the following:
\begin{lem}\label{lem-typeA}  Suppose $p>5$ and $\bbg$ is of type $A_{np-1}$ for some positive integer $n$.  Suppose $\bbf_q\subseteq \bbf_{q'}$ and suppose:

\begin{enumerate}
\item $M\subseteq \ugfr(\overline{\bbf}_p)$ is an $\bbf_{q'}$-subspace,
\item $M$ is $\bbg(\bbf_q)$-invariant, and
\item $\dim_{\bbf_{q'}}(M+\uzfr(\overline{\bbf}_p))/\uzfr(\overline{\bbf}_p)=\dim_{\overline{\bbf}_p}\ugfr(\overline{\bbf}_p)/\uzfr(\overline{\bbf}_p)$ where $\uzfr$ is the center of $\ugfr$.
\end{enumerate}
Then there exists $0\ne\lambda\in\overline{\bbf}_p$, such that $M+\uzfr(\overline{\bbf}_p)=\lambda\ugfr(\bbf_{q'})+\uzfr(\overline{\bbf}_p)$.
\end{lem}
\begin{proof}
In this case, $\ugfr(\overline{\bbf}_p)/\uzfr(\overline{\bbf}_p)$ is a simple $\bbg(\overline{\bbf}_p)$-module.  Again by \cite[Cor. 4.6]{Wei}, $\ugfr(\overline{\bbf}_p)/\uzfr(\overline{\bbf}_p)$ is a simple $\bbg(\bbf_q)$-module.  An argument similar to the proof of Lemma~\ref{lem11} establishes the claim.
\end{proof}
\begin{proof}[Proof of Proposition \ref{propliedescent}]
  Let $M=d\inv{\widetilde{\theta}}(\ugfr_2(\bbf_{q_2}))\subseteq \ugfr_1(\overline{\bbf}_p)$.  If $\bbg_1$ and $\bbg_2$ are not of type $A_{np-1}$, then Lemma \ref{lem11} finishes the proof.  So assume $\bbg_1$ is of type $A_{np-1}$.  Then $\dim_{\bbf_{q_2}}M=\dim_{\overline{\bbf}_p}\ugfr_1(\overline{\bbf}_p)$.  Notice $d\widetilde{\theta}$ induces an isomorphism between $\uzfr_1(\overline{\bbf}_p)$ and $\uzfr_2(\overline{\bbf}_p)$, and
\[\dim_{\bbf_{q_2}}(\ugfr_2(\bbf_{q_2})+\uzfr_2(\overline{\bbf}_p))/\uzfr_2(\overline{\bbf}_p)=\dim_{\overline{\bbf}_p}\ugfr_1(\overline{\bbf}_p)-1\] and hence
\[\dim_{\bbf_{q_2}}(M+\uzfr_1(\overline{\bbf}_p))/\uzfr_1(\overline{\bbf}_p)=\dim_{\overline{\bbf}_p}\ugfr_1(\overline{\bbf}_p)/\uzfr_1(\overline{\bbf}_p).\]  By Lemma~\ref{lem-typeA}, there exists $0\neq\lambda\in \bbf_{q_2}$ such that $M+\uzfr_1(\overline{\bbf}_p)=\lambda\ugfr_1(\bbf_{q_2})+\uzfr_1(\overline{\bbf}_p)$.  Since $[\ugfr_i(\bbf_{q_2}),\ugfr_i(\bbf_{q_2})]=\ugfr_i(\bbf_{q_2})$ for $i=1,2$, we have $[M,M]=\lambda^2\ugfr_1(\bbf_{q_2})$ and
\[[M,M]=\inv{d\widetilde{\theta}}([\ugfr_2(\bbf_{q_2}),\ugfr_2(\bbf_{q_2})])=\inv{d\widetilde{\theta}}(\ugfr_2(\bbf_{q_2}))=M.\]  Hence $M=[M,M]=\lambda^4\ugfr_1(\bbf_{q_2})=\lambda^2\ugfr_1(\bbf_{q_2}).$  This shows $\ugfr_1(\bbf_{q_2})=\lambda^2\ugfr_1(\bbf_{q_2})$ and hence $M=\ugfr_1(\bbf_{q_2})$.

\end{proof}
\begin{cor} $d\widetilde{\theta}$ induces isomorphisms between
\[\utfr_1(\bbf_{q_2})\mbox{ and } \ugfr_2(\bbf_{q_2})\cap \wt{\utfr}_2(\overline{\bbf}_p),\]
and
\[\ugfr_{1,\widetilde{\theta}^\ast(\beta)}(\bbf_{q_2})\mbox{ and } \ugfr_2(\bbf_{q_2})\cap \widetilde{\ugfr}_{2,\beta}(\overline{\bbf}_p),\ \forall \beta\in \Phi(\widetilde{\bbg}_2,\widetilde{\bbs}_2).\]
\end{cor}
\begin{proof}  By Proposition \ref{propliedescent} we have,
\[d\widetilde{\theta}(\ugfr_{1,\widetilde{\theta}^\ast(\beta)}(\bbf_{q_2}))\subseteq\ugfr_2(\bbf_{q_2})\cap \widetilde{\ugfr}_{2,\beta}(\overline{\bbf}_p),\] and similarly
\[d\widetilde{\theta}(\utfr_1(\bbf_{q_2}))\subseteq\ugfr_2(\bbf_{q_2})\cap \widetilde{\utfr}_2(\overline{\bbf}_p).\]
By comparing dimensions of $\ugfr_1(\bbf_{q_2})$ and
\[(\ugfr_2(\bbf_{q_2})\cap\widetilde{\tfr}_2(\overline{\bbf}_p))\oplus\left(\oplus_{\beta\in\Phi(\widetilde{\bbg}_2,\widetilde{\bbt}_2)}(\ugfr_2(\bbf_{q_2})\cap\widetilde{\ugfr}_{2,\beta}(\overline{\bbf}_p))\right)\] the result follows easily.
\end{proof}
\begin{proof}[Proof of Proposition \ref{propLPsubfields}] Notice that the Galois group $\Gal(\overline{\bbf}_p/\bbf_{q_2})$ acts naturally on $\bbg_1$, $\bbg_2$, and their Lie algebras.
The existence of such an isogeny \[\theta:\bbg_1\otimes\bbf_{q_2}\rightarrow \bbg_2\] is equivalent to $\widetilde{\theta}$ commuting with the action of  $\Gal(\overline{\bbf}_p/\bbf_{q_2})$.  More precisely, it suffices to show that for any $g_1\in\bbg_1(\overline{\bbf}_p)$ and $\sigma\in \Gal(\overline{\bbf}_p/\bbf_{q_2})$, $\sigma(\widetilde{\theta}(g_1))=\widetilde{\theta}(\sigma(g_1))$.    

Recall that by \eqref{adjaction} we have 
\be\label{eq:adj}
d\widetilde{\theta}(\Ad(g_1)(x_1))=\Ad(\widetilde{\theta}(g_1))(d\widetilde{\theta}(x_1))
\ee
for every $g_1\in\bbg_1(\overline{\bbf}_p)$ and $x_1\in \ugfr_1(\overline{\bbf}_p)$.
Since $d\widetilde{\theta}$ restricts to an isomorphism from  $\ugfr_1(\bbf_{q_2})$ to $\ugfr_2(\bbf_{q_2})$ by Proposition \ref{propliedescent}, we have
\be\label{eq:prop}
\sigma(d\widetilde{\theta}(\Ad(g_1)(x_1)))=d\widetilde{\theta}(\sigma(\Ad(g_1)(x_1))).
\ee
Since the adjoint representation of $\bbg_1$ is defined over $\bbf_{q_1}\subseteq\bbf_{q_2}$, we have
\be\label{eq:adj-field-defined}
\sigma(\Ad(g_1)(x_1))=\Ad(\sigma(g_1))(\sigma(x_1)).
\ee 
By \eqref{eq:adj}, \eqref{eq:prop}, and \eqref{eq:adj-field-defined}, we deduce that
\be\label{eq:galois-adj}
\begin{array}{rl}\sigma(d\widetilde{\theta}(\Ad(g_1)(x_1)))&=d\widetilde{\theta}(\Ad(\sigma(g_1))(\sigma(x_1)))\\
&=\Ad(\widetilde{\theta}(\sigma(g_1)))(d\widetilde{\theta}(\sigma(x_1))).
\end{array}\ee
Since $\bbg_2$ is defined over $\bbf_{q_2}$, by Proposition~\ref{propliedescent}  
\be\label{eq:inG-2}
\sigma(\Ad(\widetilde{\theta}(g_1))(d\widetilde{\theta}(x_1)))=\Ad(\sigma(\widetilde{\theta}(g_1)))(d\widetilde{\theta}(\sigma(x_1))).
\ee
Therefore by \eqref{eq:galois-adj} and \eqref{eq:inG-2}, we have
\[\Ad(\widetilde{\theta}(\sigma(g_1)))(d\widetilde{\theta}(\sigma(x_1)))=\Ad(\sigma(\widetilde{\theta}(g_1)))(d\widetilde{\theta}(\sigma(x_1)))\] and hence \[\Ad(\widetilde{\theta}(\sigma(g_1)))=\Ad(\sigma(\widetilde{\theta}(g_1))).\] Since $\bbg_2$ is an adjoint group, $\widetilde{\theta}(\sigma(g_1))=\sigma(\widetilde{\theta}(g_1))$ which proves the claim.  
\end{proof}

\subsection{Refiner description of structure type subgroups of $\gcal_\ell(K(\ell))$}
Suppose $H$ is a structural subgroup of $\gcal_{\ell}(K(\ell))$; it means there is a proper subgroup $\bbh_{\ell}$ of $\gcal_{\ell}\otimes_{K(\ell)} \overline{K(\ell)}$ such that $H\subseteq \bbh_{\ell}(\overline{K(\ell)})$. In this section, we use almost the full strength of Larsen and Pink's result to give a control on the {\em complexity} of $\bbh_{\ell}$ and its {\em field of definition}. 

\begin{definition}
Suppose $F$ is an algebraically closed field and $(\bba^n)_F$ is the affine space over $F$. The {\em complexity} of a Zariski closed subset $X$ of $F^n$ is the minimum of positive integers $D$ such that there are at most $D$ polynomials $p_i$ of degree at most $D$ in $F[x_1,\ldots,x_n]$ such that $X$ is the set of common zeros of $p_i$'s. 
\end{definition}
It is worth pointing out that one can use the language of algebraic geometry and use degree of the closure of $X$ in the projective space $\bbp^n$ to capture the above mentioned complexity of $X$; but we find it easier for the reader to work with the above mentioned quantity. 

\begin{prop}\label{prop:structural-subgroups} Suppose $\Gamma$, $\gcal$, $\gcal_{\ell}$, and $K(\ell)$ are as above; that means $\Gamma$ is a finitely generated subgroup of $\GL_{n_0}(\bbf_{q_0}[t,1/r_0(t)])$ where $q_0$ is a power of a prime $p>3$ and the field generated by $\Tr(\Gamma)$ is $\bbf_{q_0}(t)$, $\gcal$ is the Zariski-closure of $\Gamma$ in $(\GL_{n_0})_{\bbf_{q_0}[t,1/r_0(t)]}$, for any irreducible polynomial $\ell\in \bbf_{q_0}[t]$ that does not divide $r_0(t)$, let $K(\ell):=\bbf_{q_0}[t]/\langle \ell\rangle$ and $\gcal_{\ell}:=\gcal\otimes_{\bbf_{q_0}[t,1/r_0(t)]} K(\ell)$. Suppose $\bbg:=\gcal\otimes_{\bbf_{q_0}[t,1/r_0(t)]} \bbf_{q_0}(t)$ is an absolutely almost simple group, connected, simply connected group. Then if $H\subseteq \pi_{\ell}(\Gamma)$ is a proper structural subgroup for some irreducible polynomial $\ell$ with $\deg \ell\gg_{\Gamma} 1$, then there is a proper algebraic subgroup $\bbh$ of $\gcal_{\ell}$ such that 
 \begin{enumerate}
 \item 	the complexity of $\bbh$ is bounded by a function of $\Gamma$,
 \item $H\subseteq \bbh(K(\ell))\subsetneq \gcal_{\ell}(K(\ell))$.
 \end{enumerate} 
\end{prop}
\begin{proof}
As it has been mentioned earlier (see Section~\ref{ss:storngapproximation}), by Weisfeiler's strong approximation theorem there is a multiple $r_1$ of $r_0$ such that for any irreducible polynomial $\ell\in \bbf_{q_0}[t]$ that does not divide $r_1$, $\pi_{\ell}(\Gamma)=\gcal_\ell(K(\ell))$.  By the discussion at the beginning of Section~\ref{sectiondesubgroupdich}, there are a finite separable extension $L$ of $\bbf_{q_0}(t)$, a multiple $r_2$ of $r_1$, and a central $\ocal_L[1/r_2(t)]$-isogeny 
\[
\phi:\gcal\otimes_{\bbf_{q_0}[t,1/r_0(t)]} \ocal_L[1/r_2(t)]\rightarrow \gcal^{\rm Che}\otimes_{\bbz} \ocal_L[1/r_2(t)]
\] 	
where $\gcal^{\rm Che}$ is an adjoint Chevalley $\bbz$-group scheme and $\ocal_L$ is the integral closure of $\bbf_{q_0}[t]$ in $L$. By \cite[Theorem 0.5]{LP}, there is a scheme $\tcal$ of finite type over $\Spec \bbz$ and a closed group scheme $\cal$ of $\gcal^{\rm Che}\times_{\Spec \bbz} \tcal$ such that 
\begin{enumerate}
\item for any geometric point $s'$ of $\tcal$ over a geometric point $s$ of $\Spec \bbz$, the geometric fiber $\cal_{s'}$ is a proper subgroup of the geometric fiber $\gcal^{\rm Che}_s$ (here is the only place that we use the concept of geometric fiber; and so we do not give a precise definition of this concept. To illustrate what kind of objects these are, we only consider the example of a scheme $\xcal$ over $\Spec A$ where $A$ is a ring; for any $\pfr\in \Spec A$, we let $k(\pfr):=Q(A/\pfr)$ be the field of fractions of the integral domain $A/\pfr$, and then $\xcal \times_{\Spec A} \Spec(\overline{k(\pfr)})$ is a geometric fiber of $\xcal$. Vaguely if $\xcal$ is affine and given by polynomial equations with coefficients in $A$, we are looking at those polynomials modulo $\pfr\in \Spec A$ and then view them over the algebraic closure of the field of fractions of $A/\pfr$.) 
\item If $\overline{H}$ is a finite subgroup of $\gcal^{\rm Che}(\overline{\bbf}_p)$ and $s'\in \tcal$ is a point over $p\bbz$, then either $\overline{H}\subseteq \cal_{s'}(\overline{k(s')})$ where $k(s')$ is the residue field of $s'$ or there are a finite field $F_{\overline{H}}$ and a model $\bbg_{\overline{H}}$ of $\gcal^{\rm Che}\otimes_{\bbz} \overline{\bbf}_p$ over $F_{\overline{H}}$ such that 
\[
[\bbg_{\overline{H}}(F_{\overline{H}}),\bbg_{\overline{H}}(F_{\overline{H}})]\subseteq {\overline{H}} \subseteq \bbg_{\overline{H}}(F_{\overline{H}}).
\]
\end{enumerate}

By \cite[Proposition 2.3]{LP}, there is a representation $\rho:\gcal^{\rm Che}\rightarrow (\GL_{n_0})_{\bbz}$ with the following property: 

suppose $\overline{H}$ is a finite subgroup of $\gcal^{\rm Che}(\overline{\bbf}_p)$ such that a subspace of $\overline{\bbf}_p^{n_0}$ which is invariant under $\overline{H}$ should also be invariant under $\gcal^{\rm Che}(\overline{\bbf}_p)$; then $\overline{H}\not\subseteq \cal_{s'}(\overline{k(s')})$ if $s'$ is a geometric point over $p\bbz$.

For an irreducible polynomial $\ell$ that does not divide $r_2$, let $\lfr\in \Spec(\ocal_L)$ be in the fiber over $\langle \ell\rangle$. Set $L(\lfr):=\ocal_L/\lfr$. Let $\phi_{\ell}$ be the representation induced by the composite of $\rho$ and $\phi$ over $\lfr$:
\[
\phi_{\lfr}: \gcal_{\ell}\otimes_{K(\ell)} L(\lfr) \rightarrow (\GL_{n_0})_{L(\lfr)}.
\]

If $H\subseteq \gcal_{\ell}(K(\ell))$ is a proper structural subgroup, then by the above  mentioned results of Larsen-Pink there is a subspace $\wt{W}$ of $\overline{\bbf}_p^{n_0}=\overline{L(\lfr)}^{n_0}$ which is invariant under $H$ but not under $\gcal_{\ell}(\overline{L(\lfr)})$ (via the representation $\phi_{\lfr}$).  Since $\wt{W}$ is not invariant under $\gcal_{\ell}(\overline{L(\lfr)})$, the intersection of $\gcal_{\ell}\otimes_{K(\ell)}\overline{K(\ell)}$ with the stabilizer of $\wt{W}$ is a proper algebraic subgroup of $\gcal_{\ell}\otimes_{K(\ell)}\overline{K(\ell)}$. Hence the intersection of $\gcal_{\ell}\otimes_{K(\ell)}\overline{K(\ell)}$ with all the ${\rm Gal}(\overline{L(\lfr)}/L(\lfr))$-conjugates of the stabilizer of $\wt{W}$ has a descent to a proper subgroup $\wt\bbh$ of $\gcal_{\ell}\otimes_{K(\ell)}L(\lfr)$; and since $\phi_{\lfr}$ is defined over $L(\lfr)$ and $H$ leaves $\wt{W}$ invariant, $H\subseteq \wt\bbh(L(\lfr))$. For any $\sigma\in {\rm Gal}(L(\lfr)/K(\ell))$, let $\wt\bbh^{\sigma}$ be the corresponding subgroup of $\gcal_{\ell}\otimes_{K(\ell)} L(\lfr)$; and let $\bbh$ be the subgroup $\gcal_{\ell}$ that is the descent of $\bigcap_{\sigma\in {\rm Gal}(L(\lfr)/K(\ell))} \wt\bbh^{\sigma}$. Since $H\subseteq \gcal_{\ell}(K(\ell))\cap \wt{\bbh}(L(\lfr))$, we have that $H\subseteq \bbh(K(\ell))$. We notice that the complexity of the stabilizer of a subspace via $\phi_{\lfr}$ has a uniform upper bound which depends on $\rho$ and $\phi$ and it is independent of $\lfr$. Hence the complexity of $\wt{\bbh}$ is bounded as a function of $\Gamma$; moreover complexity does not change under the Galois action, which means the complexity of $\wt{\bbh}^{\sigma}$ is bounded by the same function of $\Gamma$. As $[L(\lfr):K(\ell)]\le [L:K]\ll_{\Gamma} 1$, we deduce that the complexity of $\bbh$ is bounded by a function of $\Gamma$.

Proposition 3.2 in \cite{LP} implies that, if $\deg \ell \gg_{\Gamma} 1$, then $\bbh(K(\ell))$ is a proper subgroup of $\gcal_{\ell}(K(\ell))$. For convenience sake we include its short proof here.  Since the complexity of $\bbh$ is bounded by a function of $\Gamma$, the number of its irreducible components is $O_{\Gamma}(1)$. Hence $|\bbh(K(\ell))|\ll_{\Gamma} |K(\ell)|^{\dim \bbh}$. On the other hand, since the geometric fiber of $\gcal_{\ell}$ is connected, by Lang-Weil \cite[Theorem 1]{LW}, $|\gcal_{\ell}(K(\ell))|\gg_{\Gamma} |K(\ell)|^{\dim \gcal_{\ell}}$ (It is worth pointing out that an explicit formula for $|\gcal_{\ell}(K(\ell))|$ based on invariant factors and $|K(\ell)|$ is known. So the mentioned result of Lang-Weil is not really needed; but it is more conceptual). Hence for $|K(\ell)|\gg_{\Gamma} 1$, $\bbh(K(\ell))$ is a proper subgroup of $\gcal_{\ell}(K(\ell))$.  
\end{proof}
\subsection{Refine version of the dichotomy of subgroups of $\gcal_{\ell}(K(\ell))$} Here we summarize what we have proved in the previous sections in regard to subgroups of $\pi_{\ell}(\Gamma)$.

\begin{thm}\label{thm:FinalRefinementOfLP}
Suppose $\Omega$, $\Gamma$, $\gcal$, $\gcal_{\ell}$, and $K(\ell)$ are as above; that means $\Gamma$ is a finitely generated subgroup of $\GL_{n_0}(\bbf_{q_0}[t,1/r_0(t)])$ where $q_0>7$ is a power of a prime $p>5$ and the field generated by $\Tr(\Gamma)$ is $\bbf_{q_0}(t)$, $\gcal$ is the Zariski-closure of $\Gamma$ in $(\GL_{n_0})_{\bbf_{q_0}[t,1/r_0(t)]}$, $\ell$ is an irreducible polynomial in $\bbf_{q_0}[t]$ that does not divide $r_0$, $K(\ell):=\bbf_{q_0}[t]/\langle \ell\rangle$, and $\gcal_{\ell}:=\gcal\otimes_{\bbf_{q_0}[t,1/r_0(t)]} K(\ell)$. Suppose $\bbg:=\gcal\otimes_{\bbf_{q_0}[t,1/r_0(t)]} \bbf_{q_0}(t)$ is an absolutely almost simple group, connected, simply connected group. Suppose $\deg \ell\gg_{\Gamma} 1$; then for a subgroup $H$ of $\pi_{\ell}(\Gamma)$ we have that either
\begin{enumerate}
\item {\em $H$ is a structural type subgroup}: there are a proper subgroup $\bbh$ of $\gcal_{\ell}$ and a polynomial $f_{H}\in K(\ell)[x_{11},\cdots,x_{n_0n_0}]$ such that
	\begin{enumerate}
	\item the complexity of $\bbh$ is bounded by a function of $\Gamma$, and $H\subseteq \bbh(K(\ell))\subsetneq \gcal_{\ell}(K(\ell))$.
	\item $\deg f\ll_{\Gamma} 1$,  $f_H(H)=0$, and for some $\gamma\in \Omega$, $f_H(\pi_{\ell}(\gamma))\neq 0$.
	\end{enumerate}
\item {\em $H$ is a subfield type subgroup}:
there are a subfield $F_{H}$ of $K(\ell)$ and an algebraic group $\bbg_H$ defined over $F_H$ such that 
\begin{enumerate}
\item $\bbg_H\otimes_{F_H} K(\ell)=\Ad(\gcal_{\ell})$,
\item $[\bbg_H(F_H),\bbg_H(F_H)]\subseteq \Ad H \subseteq \bbg_H(F_H)$.
\end{enumerate}	
\end{enumerate}	
\end{thm}
\begin{proof}
By Proposition~\ref{prop:structural-subgroups}, if $\deg \ell\gg_{\Gamma}1$ and $H$ is a structural type subgroup, there is a proper subgroup $\bbh$ of $\gcal_{\ell}$ such that the complexity of $\bbh$ is $O_{\Gamma}(1)$, $H\subseteq \bbh(K(\ell))\subsetneq \gcal_{\ell}(K(\ell))$. Suppose $\bbh$ is defined by polynomials $\{f_i\in K(\ell)[x_{11},\ldots,x_{n_0n_0}]|1\le i\ll_{\Gamma}1\}$, where $\deg f_i\ll_{\Gamma} 1$. Since $\gcal_{\ell}(K(\ell))\neq \bbh(K(\ell))$ and by strong approximation $\gcal_{\ell}(K(\ell))$ is generated by $\pi_{\ell}(\Omega)$, there is $\gamma\in \Omega$ and $f_i$ such that $f_i(\pi_{\ell}(\gamma))\neq 0$. This implies the claim if $H$ is a structure type subgroup.

If $H$ is subfield type subgroup, then there are a finite field $F_H\subseteq \overline{K(\ell)}$ and a model $\bbg_H$ of $\Ad\gcal_{\ell}\otimes_{K(\ell)} \overline{K(\ell)}$ over $F$ such that
\[
[\bbg_H(F_H),\bbg_H(F_H)]\subseteq \Ad H\subseteq \bbg_H(F_H).
\]  
Let $\wt\bbg_H$ be the simply connected cover of $\bbg_H$. Then $\wt\bbg_H$ is a model of $\gcal_{\ell}\otimes_{K(\ell)} \overline{K(\ell)}$; and so the adjoint homomorphism is a central isogeny \[\Ad: \wt\bbg_H\otimes_{F_H} \overline{K(\ell)}\rightarrow \Ad \gcal_{\ell} \otimes_{K(\ell)} \overline{K(\ell)}
\text{ and } 
\Ad(\wt{\bbg}_H(F_H))\subseteq \Ad H\subseteq \Ad(\gcal_{\ell})(K(\ell)).
\]
Hence by Proposition~\ref{propLPsubfields}, $F_H\subseteq K(\ell)$ and the adjoint homomorphism has a descent to $K(\ell)$, $\Ad:\wt{\bbg}_H\otimes_{F_H} K(\ell) \rightarrow \Ad \gcal_{\ell}$; and so $\bbg_H:=\Ad \wt\bbg_H$ satisfies the claim. 
\end{proof}

\subsection{A note on subfield type subgroups} In this section, we prove Proposition~\ref{prop:conjugates-subfield-type-subgroups} which will be used later in modifying Varj\'{u}'s multi-scale argument. 

\begin{prop}\label{prop:conjugates-subfield-type-subgroups}
Let $q$ be a power of a prime $p>5$, and $n\in \bbz^+$. Suppose $\bbh$ is an absolutely almost simple, connected, adjoint type $\bbf_q$-group. Then 
\[
T([\bbh(\bbf_q),\bbh(\bbf_q)],\bbh(\bbf_{q^n})):=\{g\in \bbh(\overline{\bbf}_p)|\h g^{-1} [\bbh(\bbf_q),\bbh(\bbf_q)] g\subseteq \bbh(\bbf_{q^n})\}=\bbh(\bbf_{q^n}).
\]	
\end{prop}
The main idea of the proof is similar to the proof of Proposition~\ref{propliedescent}; but as the proof is fairly short  we reproduce it here.
\begin{lem}\label{lem:absoultely-simple-modules}
Suppose $F$ is a field, $V$ is a finite-dimensional $F$-vector space, $H$ is a subgroup of ${\rm End}_F(V)$, and $V$ is an absolutely simple $H$-module; that means $V\otimes_F \overline{F}$ is a simple $\overline{F}[H]$-module where $\overline{F}$ is an algebraic closure of $F$ and $\overline{F}[H]$ is the $\overline{F}$-span of $H$ in ${\rm End}_{\overline{F}}(V\otimes_F \overline{F})$. Suppose $F\subseteq E\subseteq \overline{F}$ is an intermediate subfield. Let $F[H]$ be the $F$-span of $H$ in 
\[
{\rm End}_F(V)\subseteq {\rm End}_E(V\otimes_F E)\subseteq {\rm End}_{\overline{F}}(V\otimes_F \overline{F}).
\] 
If $W\subseteq V\otimes_F E$ is an $F[H]$-module and $\dim_F W=\dim_F V$, then there is $\lambda\in E$ such that $W=V\otimes \lambda$.
\end{lem}
\begin{proof}
First we notice that since $V$ is an absolutely simple $H$-module, by \cite[Theorem 7.5]{Lam} $F[H]={\rm End}_F(V)$; and so 
\be\label{eq:H-isomorphism}
{\rm End}_{F[H]}(V)=F. 
\ee
Suppose $\{\alpha_i\}_{i=1}^{\infty}$ is an $F$-basis of $E$. Then $V\otimes_F E=\bigoplus_{i=1}^{\infty} V\otimes \alpha_i$. For any $i$, let \[{\rm pr}_i:W\rightarrow V\otimes \alpha_i\] be the projection to the $i$-th summand according to this decomposition. We notice that, since  $l_{\alpha_i}:V\rightarrow V\otimes \alpha_i, l_{\alpha_i}(v):=v\otimes \alpha_i$ is an $F[H]$-module isomorphism, $V\otimes \alpha_i$ is a simple $F[H]$-module. Hence either ${\rm pr}_i(W)=0$ or ${\rm pr}_i:W\rightarrow V\otimes \alpha_i$ is a surjective $F[H]$-module homomorphism. As $\dim_F W=\dim_F V$, in the latter case ${\rm pr}_i$ is an $F[H]$-module isomorphism. Let $I:=\{i\in \bbz^+|\h {\rm pr}_i(W)\neq 0\}$. Then, for $i,j\in I$, 
\[
l_{\alpha_j}^{-1} \circ {\rm pr}_j \circ {\rm pr}_i^{-1}\circ l_{\alpha_i}: V\rightarrow V
\]
is an $F[H]$-module isomorphism. Therefore by \eqref{eq:H-isomorphism},  for $i,j\in I$, there is $a_{ij}\in F^{\times}$ such that 
\be\label{eq:various-components}
{\rm pr}_j\circ {\rm pr}_i^{-1}(v \otimes \alpha_i)=v \otimes a_{ij}\alpha_j.
\ee
Since $\dim_F W=\dim_F V<\infty$, by \eqref{eq:various-components} $I$ is finite. Let $i_0\in I$; then by \eqref{eq:various-components} we have
\[
\textstyle
W=\{\sum_{j\in I} v\otimes a_{i_0j} \alpha_j|\h v\in V\}=V\otimes (\sum_{j\in I} a_{i_0j}\alpha_j);
\]
and claim follows.
\end{proof}
\begin{proof}[Proof of Proposition~\ref{prop:conjugates-subfield-type-subgroups}]
	Since $p>5$, by \cite[Lemma 4.6]{Wei} 
	$\uhfr (\overline{\bbf}_{q})/\uzfr(\overline{\bbf}_q)$ 
	is a simple $H$-module, where $H=[\bbh(\bbf_q),\bbh(\bbf_q)]$, $\uhfr={\rm Lie}(\bbh)$, and $\uzfr$ is the center of $\uhfr$. Hence 
	\be\label{eq:finite-Lie-alg}
	(\uhfr(\bbf_{q^n})+\uzfr(\overline{\bbf}_q))/\uzfr(\overline{\bbf}_q)\subseteq \uhfr(\overline{\bbf}_{q})/\uzfr(\overline{\bbf}_q)
	\text{ is an absolutely simple $H$-module.}
	\ee 
  For $g\in T(H,\bbh(\bbf_{q^n}))$,  $\Ad(g) \uhfr(\bbf_{q^n})$ is $H$-invariant as we have $H\subseteq g\bbh(\bbf_{q^n})g^{-1}$. Since $\dim_{\bbf_{q^n}} \Ad(g)\uhfr (\bbf_{q^n})=\dim_{\bbf_{q^n}} \uhfr (\bbf_{q^n})$, by Lemma~\ref{lem:absoultely-simple-modules} there is $\lambda(g)\in \overline{\bbf}_{q}$ such that 
  \be\label{eq:Adg-scalar}
  \Ad(g) \uhfr (\bbf_{q^n})+\uzfr (\overline{\bbf}_q)=\lambda(g) \uhfr (\bbf_{q^n})+\uzfr (\overline{\bbf}_q).
  \ee
  Since $p>5$, $\uhfr (\bbf_{q^n})$ is a perfect Lie algebra. Therefore by \eqref{eq:Adg-scalar} we get that for any integer $m\ge 2$ we have
  \be\label{eq:Adg-scalar-2}
  \Ad(g) \uhfr (\bbf_{q^n})=\lambda(g)^m \uhfr (\bbf_{q^n}).
  \ee
  Notice that $\uhfr (\bbf_{q^n})$ and $\uhfr (\overline{\bbf}_q)$ are naturally isomorphic to $\hfr\otimes_{\bbf_q} \bbf_{q^n}$ and $\hfr\otimes_{\bbf_q} \overline{\bbf}_q$, respectively, where $\hfr=\uhfr (\bbf_{q})$; and so $\lambda(g)^m \uhfr (\bbf_{q^n})$ can be identified with $\hfr\otimes \lambda(g)^m \bbf_{q^n}$. Thus \eqref{eq:Adg-scalar-2} implies that $\lambda(g)\in \bbf_{q^n}$. Therefore $\Ad(g)\uhfr (\bbf_{q^n})=\uhfr (\bbf_{q^n})$, which means $g\in \bbh(\bbf_{q^n})$ as $\bbh$ is of adjoint form.
\end{proof}
\begin{cor}\label{cor:intersection-conjugate-subfield-type}
Let $q$ be a power of a prime $p>5$. Let $\bbh$ be a connected, almost simple, adjoint type $\bbf_q$-group.	 Suppose $n$ is a positive integer and $m$ is a positive divisor of $n$. Then, for any $g\in \bbh(\bbf_{q^n})\setminus \bbh(\bbf_{q^m})$, $g\bbh(\bbf_{q^m})g^{-1}\cap \bbh(\bbf_{q^m})$ is a structural subgroup of $\bbh(\bbf_{p^n})$.
\end{cor}
\begin{proof}
Suppose to the contrary that it is a subfield type subgroup. Then by Proposition~\ref{propLPsubfields} there is a subfield $F'$ of $\bbf_{q^m}$ and a model $\overline\bbh$ of $\bbh\otimes_{\bbf_q}\bbf_{q^m}$ over $F'$ such that 
\[
[\overline\bbh(F'),\overline\bbh(F')]\subseteq g\overline\bbh(\bbf_{q^m})g^{-1}\cap \overline\bbh(\bbf_{q^m}) \subseteq \overline\bbh(F').
\]
Therefore $g\in T([\overline\bbh(F'),\overline\bbh(F')], \overline\bbh(\bbf_{q^m}))$; and so by Proposition~\ref{prop:conjugates-subfield-type-subgroups} we have that $g$ is in $\overline\bbh(\bbf_{q^m})=\bbh(\bbf_{q^m})$, which is a contradiction. 
\end{proof}

\section{Escaping from the direct sum of structure type subgroups}\label{sectionproofofmainescape}

For a square-free polynomial $f$ (with large degree irreducible factors), we say a proper subgroup $H$ of $\pi_f(\Gamma)$ is {\em purely structural} if $\pi_{\ell}(H)$ is a structure type subgroup of $\pi_{\ell}(\Gamma)=\gcal_{\ell}(K(\ell))$ for any irreducible factor $\ell$ of $f$. The goal of this section is to prove Proposition~\ref{propmainescape}; that roughly means we show that there exists a symmetric set $\Omega'\subseteq \Gamma$ with the following property:  For any square-free polynomial $f\in \bbf_{q_0}[t]$ with large degree irreducible factors and for any purely structural subgroup $H$ of $\pi_f(\Gamma)$, the probability that an $l\sim \deg f$-step random walk lands in $H$ is small. 

\subsection{Small lifts of elements of a purely structural subgroup are in a proper algebraic subgroup}\label{sectionsmalllifts}

Let us recall that for any $h\in \GL_{n_0}(\bbf_{q_0}[t,1/r_0])$, 
\[
\|h\|:=\max_{v\in D(r_0)\cup \{v_{\infty}\},i,j} |h_{ij}|_{v},
\]
 where $h_{ij}$ is the $i,j$-entry of $h$ and $|\cdot|_{\ell}$ is the $\ell$-adic norm (see section~\ref{sectionnotation} for the definition of all the undefined symbols). For a subgroup $H$ of $\pi_f(\Gamma)$, let 
 \[
 \lcal_\delta(H):=\{h=(h_{ij})\in \Gamma|\pi_f(h)\in H\mbox{ and }\|h\|<[\pi_f(\Gamma):H]^\delta\},
 \]
In this section we show that, if $H$ is purely structural, then for some $\delta\ll_{\bbg} 1$, $\lcal_\delta(H)$ lies in a proper algebraic subgroup of $\bbg$. In light of Theorem~\ref{thm:FinalRefinementOfLP}, we follow the proof of \cite[Proposition 16]{SGV}. 

{\bf Standing assumptions.} In this section, we will be working with $\Omega$, $\Gamma$, $\gcal$, $\gcal_{\ell}$, $\bbg$, and $K(\ell)$ are as before; that means $\Gamma$ is a finitely generated subgroup of $\GL_{n_0}(\bbf_{q_0}[t,1/r_0(t)])$ where $q_0>7$ is a power of a prime $p>5$ and the field generated by $\Tr(\Gamma)$ is $\bbf_{q_0}(t)$, $\gcal$ is the Zariski-closure of $\Gamma$ in $(\GL_{n_0})_{\bbf_{q_0}[t,1/r_0(t)]}$, $\bbg$ is the generic fiber of $\gcal$, $\bbg$ is a connected, simply-connected, absolutely almost simple group, for an irreducible polynomial $\ell$, $K(\ell)$ is $\bbf_{q_0}[t]/\langle \ell\rangle$, and $\gcal_{\ell}$ is the fiber of $\gcal$ over $\langle \ell\rangle$. Here $f$ denotes a square free polynomial with the property that the dichotomy mentioned in Theorem~\ref{thm:FinalRefinementOfLP} holds for any of its irreducible factors. In particular, for any irreducible factor $\ell$ of $f$ and any proper subgroup $H_{\ell}$ of $\pi_{\ell}(\Gamma)$, we have that
\[
[\pi_{\ell}(\Gamma):H_{\ell}]\gg_{\Gamma} 
\begin{cases}
 |K(\ell)|^{\dim \bbg-\dim \bbh}\ge |K(\ell)| &\text{ if } H_{\ell} \text{ is a structure type subgroup,}\\
 |K(\ell)/F_H|^{\dim \bbg}\ge |K(\ell)| &\text{ if } H_{\ell} \text{ is a subfield type subgroup.}	
\end{cases}
\]
This implies that 
\be\label{eq:LowerboundIndex}
[\pi_{\ell}(\Gamma):H_{\ell}]\gg_{\Gamma} |\pi_{\ell}(\Gamma)|^{c_0}
\ee
for some positive number $c_0$ which depends only on $\bbg$. Moreover we assume, if $\ell$ and $\ell'$ are two different irreducible factors of $f$, then $\deg \ell\neq \deg \ell'$. This last condition is very restrictive and in a desired result it has to be removed. Removing this condition is in the spirit of Open Problem 1.4 in \cite{LinVar}.

We first start with approximating a proper subgroup $H$ of 
\[
\pi_f(\Gamma)\simeq \bigoplus_{\ell|f, \ell \text{ irred.}} \pi_{\ell}(\Gamma)
= \bigoplus_{\ell|f, \ell \text{ irred.}} \gcal_{\ell}(K(\ell)) 
\] 
with a subgroup in product form. This is done by a variant of \cite[Lemma 15]{SGV}. 
\begin{lem}\label{lemsubgroupproductform}  
Suppose $\{G_i\}_{i\in I}$ is a finite collection of finite groups with the following properties:
\begin{enumerate}
\item $G_i=\bigoplus_{j\in J_i} L_{ij}$ where $L_{ij}/Z(L_{ij})$ is simple.
\item $G_i$ is perfect; that means $G_i=[G_i,G_i]$.
\item For $i\neq j$, simple factors of $G_i/Z(G_i)$ and $G_j/Z(G_j)$ are not isomorphic.
\item There is a positive integer $c$ such that for any proper subgroup $H_i$ of $G_i$ we have $[G_i:H_i]\ge |G_i|^{c}$.	
\end{enumerate}
Then for any subgroup $H$ of $G_I:=\bigoplus_{i\in I}G_i$ we have 
\[
\prod_{i\in I} [G_i:\pr_i(H)]\ge[G_I:H]^c,
\]
where $\pr_i:G_I\rightarrow G_i$ is the projection to the $i$-th component. 
\end{lem}
\begin{proof} We proceed by strong induction on $|G_I|$. Let
\[
I_1:=\{i\in I|\h \pr_i(H)=G_i\}\text{, and } 
I_2:=\{i\in I|\h \pr_i(H)\neq G_i\}.
\]
{\bf Claim 1.} {\em We can assume that $I_1\neq \varnothing$.}

{\em Proof of Claim 1.} If $I_1=\varnothing$, then
\[
\prod_{i\in I} [G_i:\pr_i(H)]\ge \prod_{i\in I} |G_i|^c\ge |G_I|^c\ge [G_I:H]^c;
\]
and claim follows. So without loss of generality we can and will assume that $I_1\neq \varnothing$. 

{\bf Claim 2.} {\em The restriction to $H$ of the projection map $\pr_{I_1}$ to $G_{I_1}:=\bigoplus_{i\in I_1} G_i$ is surjective.}

{\em Proof of Claim 2.} We proceed by induction on $|I_1|$. The base of induction is clear. Suppose $\pr_{I'}(H)=G_{I'}$ for some subset $I'$ of $I$ and $\pr_i(H)=G_i$ for some $i\in I\setminus I'$. Let $\overline{H}:=\pr_{I'\cup\{i\}}(H)$. Then $\pr_i(\overline{H})=G_i$ and $\pr_{I'}(\overline{H})=G_{I'}$. Let $\overline{H}(I'):=\overline{H}\cap G_{I'}$ and $\overline{H}(i):=\overline{H}\cap G_i$. Then projections induce isomorphisms $\overline{H}/\overline{H}(I')\rightarrow G_i$ and $\overline{H}/\overline{H}(i)\rightarrow G_{I'}$. Hence we get the following commuting diagram
\be\label{eq:graph-hom}
\begin{tikzcd}
&\overline{H}/(\overline{H}(i)\oplus \overline{H}(I'))  \arrow[rd,"\simeq"]\arrow[ld,"\simeq"] &	
\\
G_i/\overline{H}(i) \arrow[rr,dashed,"\simeq"]&&
G_{I'}/\overline{H}(I').
\end{tikzcd}
\ee
If $\overline{H}(i)$ is a proper subgroup of $G_i$, then $\overline{H}(i)Z(G_i)$ is also a proper subgroup of $G_i$; this is because $[\overline{H}(i)Z(G_i),\overline{H}(i)Z(G_i)]=[\overline{H}(i),\overline{H}(i)]$ and $G_i$ is perfect. Therefore by \eqref{eq:graph-hom} a simple factor of $G_i/Z(G_i)$ is isomorphic to a simple factor of $G_j/Z(G_j)$ for some $j\in I'$; this contradicts our assumption. Hence $\overline{H}(i)=G_i$ and $\overline{H}(I')=G_{I'}$, which implies that $\overline{H}=G_{I'\cup\{i\}}$; and claim follows.

{\bf Claim 3.} {\em $[G:H]\le |G_{I_2}|$ and $\prod_{i\in I} [G_i:\pr_i(H)]\ge |G_{I_2}|^c$.} 

{\em Proof of Claim 3.} Let $H(I_2):=H\cap \ker \pr_{I_1}$ where $\pr_{I_1}:G\rightarrow G_{I_1}$ is the projection to $G_{I_1}$. Then by Claim 2, we have $|G_{I_1}|=[H:H(I_2)]$; and so 
\[
[G:H]=\frac{|G_{I_1}||G_{I_2}|}{|G_{I_1}||H(I_2)|}=\frac{|G_{I_2}|}{|H(I_2)|}\le |G_{I_2}|.
\]
We also have
\[
\prod_{i\in I} [G_i:\pr_i(H)]=\prod_{i\in I_2} [G_i:\pr_i(H)]\ge \prod_{i\in I_2} |G_i|^c=|G_{I_2}|^c,
\]
where we have the last inequality because of our assumption and $\pr_i(H)$ being a proper subgroup of $G_i$ for any $i\in I_2$.

Claim 3 implies that 
\[
\prod_{i\in I}[G_i:\pr_i(H)]\ge |G_{I_2}|^c \ge [G:H]^c;
\]
and claim follows.
\end{proof}

\begin{prop}\label{propsmalllifts}  Under the {\bf Standing assumptions} of this section, there exists a constant $\delta$ depending on $\Gamma$ such that the following holds:  Let $H\subseteq \pi_f(\Gamma)$ be a purely structural subgroup; that means $\pi_\ell(H)$ is a structural subgroup of $\pi_\ell(\Gamma)$ for each irreducible factor $\ell$ of $f$.  Then $\lcal_\delta(H)$ lies in a proper algebraic subgroup $\bbh$ of $\bbg$.
\end{prop}
\begin{proof}
 By Lemma \ref{lemsubgroupproductform} and \eqref{eq:LowerboundIndex}, there exists a positive constant $c_0$ which depends only on $\bbg$ such that \[[\pi_f(\Gamma):\bigoplus_{\ell\in D(f)}\pi_\ell(H)]\ge [\pi_f(\Gamma):H]^{c_0}.\]  
 
 If $\lcal_\delta(\bigoplus_{\ell\in D(f)}\pi_\ell(H))$ lies in a proper algebraic subgroup of $\bbg$, then so does $\lcal_{\delta/c_0}(H)$.  Therefore we can and will replace $H$ with $\bigoplus_{\ell\in D(f)}\pi_\ell(H)$.  Similarly, after replacing $f$ with the product of those irreducible factors satisfying $\pi_\ell(H)\neq\pi_\ell(\Gamma)$, we may assume $\pi_\ell(H)$ is a proper subgroup for each irreducible factor $\ell$ of $f$. By Theorem~\ref{thm:FinalRefinementOfLP}, there exists a constant $d_0:=d_0(\Gamma)$ such that for any $\ell\in D(f)$, there is a polynomial of degree at most $d_0$ and $\gamma_{\ell}\in \Omega$ such that $f_\ell(\pi_{\ell}(H))=0$ and $f_{\ell}(\pi_{\ell}(\gamma_{\ell}))=1$.  
 
First we show that $\lcal_{\delta}(H)$ lies in a {\em low complexity} proper algebraic {\em subset} of $\bbg$. To this end, we consider the degree $d_0$ monomial map

\[\Psi:\bbgl_{n_0}\rightarrow \bba_{d_1},\]
where \[d_1=\begin{small}\left(\begin{array}{c}
n_0^2+d_0 \\ 
d_0
\end{array} \right)\end{small}.\]

Let $d$ be the dimension of the linear span of $\Psi(\bbg(\bbf_{q_0}(t)))$. To show $\lcal_{\delta}(H)$ lies in a proper algebraic subgroup of $\bbg$, it suffices to prove that $\Psi(\lcal_\delta(H))$ spans a subspace of dimension less than $d$ if $\delta$ is sufficiently small. 

Suppose to the contrary that the linear span of $\Psi(\lcal_\delta(H))$ is $d$ dimensional.  Hence there is  a set of $d$ linearly independent elements $h_1,h_2,\dots, h_d$ of $\Psi(\lcal_\delta(H))$. 

Looking at the explicit formula for the number of elements of finite simple groups of Lie type~\cite[\textsection 11.1, \textsection 14.4]{Car}, we have $|\gcal_\ell(K(\ell))|\le |K(\ell)|^{\dim \bbg}=q_0^{\dim \bbg\cdot \deg \ell}$. Hence
\[
\pi_f(\Gamma)\le q_0^{\dim \bbg \cdot \deg f}.
\]
Thus for $h\in \lcal_{\delta}(H)$ we have
\[
\|h\|<[\pi_f(\Gamma):H]^{\delta}\le |\pi_f(\Gamma)|^{\delta}\le q_0^{\delta \dim \bbg\cdot \deg f}.
\]
This implies that the entries of the vectors $h_1,\ldots,h_d\in \bbf_{q_0}(t)^{d_1}$ are of the form $\frac{a}{\prod_{\ell\in D(r_0)} \ell^{e_{\ell}}}$ with $a\in \bbf_{q_0}[t]$, $\gcd(a, \prod_{\ell\in D(r_0)}\ell^{e_{\ell}})=1$, 
\begin{equation}\label{eqnsubdet1}
\deg a-\sum_{\ell\in D(r_0)}e_\ell \deg \ell< d_0\delta \dim\bbg \cdot\deg f,
\end{equation} 
and for each $\ell\in D(r_0)$
\begin{equation}\label{eqnsubdet2}
{e_\ell}\deg \ell< d_0\delta\dim\bbg\cdot \deg f.
\end{equation}
By the contrary assumption, the determinant $s(t)\in \bbf_{q_0}(t)$ of a $d$-by-$d$ submatrix of the matrix $X$ that has the vectors $h_1,\ldots,h_d$ in its rows is non-zero. By \eqref{eqnsubdet1} and \eqref{eqnsubdet2}, we have that $s(t)=\frac{a'}{\prod_{\ell\in D(r_0)} \ell^{e'_{\ell}}}$ for some $a'\in \bbf_{q_0}[t]$ and $e_{\ell}\in \bbz^{\ge 0}$ such that
\[
\deg a'\le \delta ((|D(r_0)|+1)dd_0\dim \bbg) \deg f\le \delta ((|D(r_0)|+1)|D(r_0)| dd_0\dim \bbg) \max_{\ell\in D(f)}\deg \ell.
\]
Hence for $\delta< ((|D(r_0)|+1)|D(r_0)| dd_0\dim \bbg)^{-1}$, there is an irreducible factor $\ell_0$ of $f$ such that $\deg a'<\deg \ell_0$; in particular, $\pi_{\ell_0}(s(t))\neq 0$. This implies that $\pi_{\ell_0}(h_1),\ldots,\pi_{\ell_0}(h_d)$ are $K(\ell_0)$-linearly independent in $K(\ell_0)^{d_1}$; and so the right kernel of $\pi_{\ell_0}(X)$ is zero. By the definition of $\lcal_{\delta}(H)$, we have that $\pi_{\ell_0}(h_i)\in \Psi(\pi_{\ell_0}(H))$. Since $f_{\ell_0}(\pi_{\ell_0}(H))=0$ and $\deg f\le d_0$, we have that the coefficients of $f_{\ell_0}$ form a column vector in the right kernel of $\pi_{\ell_0}(X)$, which is a contradiction.  

Therefore there is a proper algebraic subset $\bbx$ of $\bbg$ whose complexity is $O_{\Gamma}(1)$, and $\lcal_\delta(H)$ is a subset of $\bbx(\bbf_{q_0}(t))$. 
  By \cite[Proposition 3.2]{EMO} if $A\subseteq \bbg(\bbf_{q_0}(t))$ is a generating set of a Zariski-dense subgroup of $\bbg$, then there exists a positive integer $N$ depending on the complexity of $\bbx$ such that $\prod_N A\not\subseteq \bbx(\bbf_{q_0}(t))$.  It should be pointed out that the statement of \cite[Proposition 3.2]{EMO} is written for algebraic varieties and groups over $\bbc$.  Its proof, however, is based on a generalized B\'{e}zout theorem that has a positive characteristic counter part (see \cite[Pg. 519]{Sch}, \cite[Ex. 12.3.1]{Fu}, and \cite[III. Thm 2.2]{Da}). Altogether one can see that the proof of \cite[Proposition 3.2]{EMO} is valid over any algebraically closed field. Since 
  \[
  \textstyle \prod_N\lcal_{\delta/N}(H)\subseteq \lcal_\delta(H)\subseteq \bbx(\bbf_{q_0}(t)),
  \]
  we deduce that the group generated by $\lcal_{\delta/N}(H)$ is not Zariski-dense in $\bbg$; that means that $\lcal_{\delta/N}(H)$ lies in a proper algebraic subgroup of $\bbg$; and claim follows as $N=O_{\Gamma}(1)$.  
 \end{proof}

\subsection{Invariant theoretic description of proper positive dimensional subgroups of a simple group: the positive characteristic case}
In this section, we provide an {\em invariant theoretic} (or one can say a {\em geometric}) description of proper positive dimensional algebraic subgroups of an absolutely almost simple group over a field of {\em positive characteristic}. This is the positive characteristic counter part of \cite[Proposition 17, part (1)]{SGV}; and later it plays an important role in the proof of Proposition~\ref{propmainescape}. 

In this section we slightly deviate from our {\bf Standing assumptions}, and let $\bbg$ be a simply connected absolutely almost simple algebraic group defined over a positive characteristic {\em algebraically closed field $k$}. 

\begin{prop}\label{propreps}  Let $\bbg$ be an absolutely almost simple group defined over an algebraically closed field $k$ of positive characteristic.  Then there are finitely many group homomorphisms $\{\rho_i:\bbg\rightarrow (\bbgl)_{\bbv_i}\}_ {i=1}^d$ and $\{\rho'_j:\bbg\rightarrow {\rm Aff}(\bbw_j)\}_{j=1}^{d'}$ such that 
\begin{enumerate}
\item for any $i$, $\rho_i$ is irreducible and non-trivial.
\item for any $j$, $\rho'_j(g)(v):=\rho'_{{\rm lin},j}(g)(v)+w_j(g)$ where $\rho'_{{\rm lin},j}:\bbg\rightarrow \bbgl(\bbw_j)$ is irreducible and non-trivial, and $w_j(g)\in \bbw_j(k)$; and no point of $\bbw_j(k)$ is fixed by $\bbg(k)$ under the affine action given by $\rho_j'$.
\item for every positive dimensional closed subgroup $\bbh$ of $\bbg$, either there is an index $i$ and a non-zero vector $v\in \bbv_i(k)$ such that $\rho_i(\bbh(k))[v]=[v]$ where $[v]$ is the line in $\bbv_i(k)$ spanned by $v$, or there is an index $j$ and a point $w$ in $\bbw_j(k)$ such that $\rho_j'(\bbh(k))(w)=w$.
\end{enumerate}
\end{prop}

Let us remark that in the characteristic zero case any affine representation $\bbv$ of a semisimple group has a fixed point; here is a quick argument: suppose $g\cdot v:=\rho(g)(v)+c(g)$. We identify the affine space of $\bbv(k)$ with the hyperplane $\{(v,1)| v\in\bbv(k)\}$ of $W:=\bbv(k)\oplus k$; and so 
$
\wh\rho(g):=
\begin{pmatrix} 
\rho(g) & c(g) \\
0 & 1	 
\end{pmatrix}
$ 
is a group homomorphism and $(g\cdot v,1)=\wh{\rho}(g)(v,1)$. In the characteristic zero case any module is completely reducible; and so there is a line $[v]$ which is invariant under $\bbg(k)$ and $W=\bbv(k)\oplus [v]$. As $\bbg$ is semisimple, it does not have a non-trivial character. Hence any point on $[v]$ is a fixed point of $\bbg(k)$. As $[v]\not\subseteq \bbv(k)$, after rescaling, if needed, we can and will assume that $v=(v_0,1)$ for some $v_0\in \bbv(k)$. Therefore $\wh\rho(g)(v)=v$ implies that $g\cdot v_0=v_0$. 

In the positive characteristic case, however, there are affine transformations of $\bbg(k)$ that have no fixed points: there are irreducible representations $\bbv$ of $\bbg$ such that $H^1(\bbg(k),\bbv(k))\neq 0$. Hence there is a non-trivial cocycle $c:\bbg(k)\rightarrow \bbv(k)$. Since $c$ is a cocycle, $g\cdot v:=\rho(g)(v)+c(g)$ is a group action. If $g\cdot v_0=v_0$ for some $v_0$, then $c(g)=v_0-\rho(g)(v_0)$ which means $c$ is a trivial cocycle; and this contradicts our assumption. 

This said it is not clear to the authors if the mentioned affine representations are needed in Proposition~\ref{propreps} or not. 
 
\begin{question}\label{ques:subgroups-affine}
Suppose $\bbg$ is a connected, absolutely almost simple group and $\bbh$ is a positive dimensional proper subgroup	of $\bbg$. Is there a non-trivial irreducible representation $\rho:\bbg\rightarrow \bbgl(\bbv)$ of $\bbg$ and a non-zero vector $v\in \bbv(k)\setminus \{0\}$ such that $\rho(\bbh(k))([v])=[v]$?
\end{question}

As we will see in the proof of Proposition ~\ref{propreps}, the mentioned affine representations arise as submodules of wedge powers of the adjoint representation of $\bbg(k)$. When the characteristic of the field $k$ is large compared to the dimension of $\bbg$, all these representations are completely reducible; and so by a similar argument as in the characteristic zero case, one can see that such affine representations do not occur. Hence one gets a positive affirmative answer to Question~\ref{ques:subgroups-affine}.

\begin{proof}[Proof Proposition~\ref{propreps}] Since $\bbh$ is a proper positive dimensional subgroup, $\hfr:=\Lie(\bbh)(k)$ is a non-trivial proper subspace of $\gfr:=\Lie(\bbg)(k)$. Since $\bbg$ is an absolutely almost simple group, $\gfr/\zfr$ is a simple $G:=\bbg(k)$-module where $\zfr:=Z(\gfr)$ is the center of $\gfr$ and $\gfr$ is a perfect Lie algebra; that means $\gfr=[\gfr,\gfr]$. Therefore $(\hfr+\zfr)/\zfr$ is a proper subspace of $\gfr/\zfr$ and it is not $G$-invariant. Thus $\hfr$ is not invariant under $G$. From here we deduce that $l_H:=\wedge^{\dim_k \hfr}\hfr$ is not invariant under $G$, where $G$ acts on $\wedge^{\dim_k \hfr} \gfr$ via the representation $\wedge^{\dim_k \hfr} \Ad$. Suppose 
\[
0:=V_0\subset V_1 \subset \cdots \subset V_m:=\wedge^{\dim_k \hfr}\gfr
\]
is a composition factor of $\wedge^{\dim_k \hfr}\gfr$. Let $m'$ be the smallest index such that $l_H\subseteq V_{m'}$ as a $G$-module. Hence $l_H\not\subseteq V_{m'-1}$, which implies $l_H\oplus V_{m'-1}\subseteq V_{m'}$. 

{\bf Step 1.} (Composition factor is non-trivial) If $\dim_k V_{m'}/V_{m'-1}>1$, then $V_{m'}/V_{m'-1}$ is a non-trivial simple $G$-module that has a line which is $H$-invariant; here $H:=\bbh(k)$. 

{\bf Step 2.} (Triviality of the composition factor gives us an affine action whose linear part is irreducible) If $\dim_k V_{m'}/V_{m'-1}=1$, then $l_H\oplus V_{m'-1}=V_{m'}$. Let $V:=V_{m'-1}/V_{m'-2}$ and $W:=V_{m'}/V_{m'-2}$; and so $W/V$ is a one dimensional $G$-module. Since $\bbg$ has no non-trivial character, $G$ acts trivially on $W/V$. Suppose $w\in W\setminus V$; then for any $g\in G$, $c_w(g):=\rho_W(g)(w)-w\in V$. For $v\in V$ and $g\in G$, we let $g\cdot v:=\rho_V(g)(v)+c_w(g)$; then 
\begin{align*}
	g_1\cdot(g_2\cdot v)= & \rho_V(g_1)((g_2\cdot v))+c_w(g_1)\\
	= & \rho_V(g_1)(\rho_V(g_2)(v)+c_w(g_2))+c_w(g_1) \\
	= & \rho_V(g_1g_2)(v)+\rho_W(g_1)(\rho_W(g_2)(w)-w)+(\rho_W(g_1)(w)-w)\\
	= & \rho_V(g_1g_2)(v)+\rho_W(g_1g_2)(w)-\rho_W(g_1)(w)+\rho_W(g_1)(w)-w \\
	= & \rho_V(g_1g_2)(v)+(\rho_W(g_1g_2)(w)-w)\\
	= & \rho_V(g_1g_2)(v)+c(g_1g_2)=(g_1g_2)\cdot v.
\end{align*}
So $g\cdot v$ defines an affine action of $G$ on $V$. Suppose $x_H\in l_H\setminus\{0\}$; then $x_H=c_0 w+v_0$ for some $c_0\in k^{\times}$ and $v_0\in V$. For any $h\in H$, we have $\rho_W(h)(x_H)=x_H$, which implies that $c_0 (\rho_W(h)(w)-w)=v_0-\rho_V(h)(v_0)$. Therefore for any $h\in H$, 
\be\label{eq:coboundary}
c_w(h)=c_0^{-1}(v_0-\rho_V(h)(v_0)).
\ee 
Since $x_H$ is not fixed by $G$, there is $g_0\in G$ such that $\rho_W(g_0)(x_H)\neq x_H$, which implies 
\be\label{eq:not-equal}
c_w(g_0)\neq c_0^{-1}(v_0-\rho_V(g_0)(v_0)).
\ee
{\bf Step 3.} (Affine action has a fixed point) If the above affine action has a fixed point $v_1\in V$, then for any $g\in G$,
\be\label{eq:fixedpoint}
v_1=\rho_V(g)(v_1)+c_w(g). 
\ee
By \eqref{eq:coboundary} and \eqref{eq:fixedpoint}, for any $h\in H$, we have $v_1-\rho_V(h)(v_1)=c_0^{-1}(v_0-\rho_V(h)(v_0)),$ which implies
\be\label{eq:existence-of-fixed-point}
\rho_V(h)(c_0^{-1}v_0-v_1)=c_0^{-1}v_0-v_1.
\ee
By \eqref{eq:not-equal} and \eqref{eq:fixedpoint}, we have $\rho_V(g_0)(c_0^{-1}v_0-v_1)\neq (c_0^{-1}v_0-v_1)$. Therefore $\rho_V$ is a non-trivial irreducible representation of $G$ that has a non-zero vector fixed by $H$. 

{\bf Step 4.} (Affine action does not have a fixed point) Now suppose that the above affine action does not have a $G$-fixed point; then by \eqref{eq:coboundary} for any $h\in H$, 
\[
h\cdot (c_0^{-1}v_0)=\rho_V(c_0)^{-1}v_0)+c_w(h)=\rho_V(h)((c_0)^{-1}v_0))+c_0^{-1}(v_0-\rho_V(h)(v_0))=c_0^{-1}v_0,
\]
which means $H$ has a fixed point; and so claim follows.
\end{proof}

\subsection{Invariant theoretic description of small lifts of purely structural subgroups} In this section based on Proposition~\ref{propsmalllifts}  and Proposition~\ref{propreps}, we give an invariant theoretic understanding of small lifts of purely structural subgroups of $\pi_f(\Gamma)$ under the {\bf Standing assumption} (see the 2nd paragraph of Section \ref{sectionsmalllifts}).

\begin{prop}\label{prop:SmallLiftsInvariantTheoretic}
	Let $\Gamma, \bbg, f$ be as in the {\bf Standing assumption}. Then
	\begin{enumerate}
	\item there are local fields $\kcal_i$ and $\kcal_j'$ that are field extensions of $\bbf_{q_0}(t)$.
	\item there are homomorphisms $\rho_i:\bbg\otimes_{\bbf_{q_0}(t)} \kcal_i\rightarrow \mathbb{GL}(\bbv_i)$ and $\rho_j':\bbg\otimes_{\bbf_{q_0}(t)} \kcal_j'\rightarrow {\rm Aff}(\bbw_j)$ such that 
	\begin{enumerate}
	 \item $\rho_i$'s are non-trivial irreducible representations over a geometric fiber; that means after a base change to an algebraic closure of $\kcal_i$, $\rho_i$ is non-trivial and irreducible,
	 \item the linear parts $\rho_{{\rm lin},j}'$'s of the affine representations $\rho_j'$ are non-trivial irreducible representations over a geometric fiber,
	 \item $\bbg(\kcal_j')$ does not fix any point of $\bbw_j(\kcal_j')$.
	 \item $\rho_i(\Gamma)\subseteq \GL(\bbv_i(\kcal_i))$ and $\rho_{{\rm lin},j}'(\Gamma)\subseteq \GL(\bbw_j(\kcal_j'))$ are unbounded subgroups.
	\end{enumerate}
	\item there is $\delta>0$ depending on $\Gamma$ such that for any purely structural subgroup $H$ of $\pi_f(\Gamma)$ one of the following conditions hold:
	\begin{enumerate}
	\item the group generated by $\lcal_{\delta}(H)$ is a finite subgroup of $\Gamma$.
	\item for some $i$, there is a non-zero $v\in \bbv(\kcal_i)$ such that for any $h\in \lcal_{\delta}(H)$, $\rho_i(h)([v])=[v]$.
	\item for some $j$, there is $w\in \bbw_j(\kcal_j')$ such that for any $h\in \lcal_{\delta}(H)$, $\rho_j'(h)(w)=w$.
	\end{enumerate}
	\end{enumerate}
\end{prop}
\begin{proof}
Let $k$ be an algebraic closure of $\bbf_{q_0}(t)$; then by Proposition~\ref{propreps} the geometric fiber $\wt\bbg:=\bbg\otimes_{\bbf_{q_0}(t)}k$ of $\bbg$ has representations $\{\wt\rho_i\}_i$ and $\{\wt\rho_j'\}_j$ that can describe positive dimensional proper subgroups of $\wt\bbg$ (as in the statement of Proposition~\ref{propreps}). There is a finite Galois extension $L$ of $\bbf_{q_0}(t)$ such that $\wt\rho_i$ and $\wt\rho_j'$ have Galois descents $\wh\rho_i$ and $\wh\rho_j'$ to $\bbg\otimes_{\bbf_{q_0}(t)}L$. As $\Gamma$ is a discrete subgroup of $\prod_{v\in D(r_0)\cup\{v_{\infty}\}} \bbg(K_v)$ where $K_v$ is the $v$-adic completion of $\bbf_{q_0}(t)$, for any $i$ and $j$ there are some $v_i,v_j'\in D(r_0)\cup\{v_{\infty}\}$ and extensions $\nu_i,\nu_j'\in V_L$ of $v_i$ and $v_j'$, respectively, such that $\wh\rho_i(\Gamma)\subseteq \GL(\bbv_i(L_{\nu_i}))$ and $\wh\rho_{j}'(\Gamma)\subseteq \GL(\bbw_j(L_{\nu_j'}))$ are unbounded. So $\kcal_i:=L_{\nu_i}$, $\kcal_j':=L_{\nu_j'}$, $\rho_i:=\wh\rho_i\otimes {\rm id}_{\kcal_i}$, and $\rho_j':=\wh\rho_j\otimes {\rm id}_{\kcal_j'}$ satisfy parts (1) and (2). 

Let $\delta$ be as in Proposition~\ref{propsmalllifts}; then for any structural subgroup $H$ of $\pi_f(\Gamma)$, there is a proper subgroup $\bbh$ of $\bbg$ such that $\lcal_{\delta}(H)\subseteq \bbh(k)$. If $\bbh$ is zero-dimensional, then the group generated by $\lcal_{\delta}(H)$ is a finite group. If $\bbh$ is positive dimensional, then Proposition~\ref{propreps} implies that either (3.b) holds or (3.c); and claim follows. 
\end{proof}

\subsection{Ping-pong argument}
Let's recall that under the {\bf Standing assumptions} (see the 2nd paragraph in Section~\ref{sectionsmalllifts}), we want to show a random walk with respect to the probability counting measure on $\pi_f(\Omega)$ after $O(\deg f)$-many steps lands in a purely structural subgroup $H$ of $\pi_f(\Gamma)$ with {\em small probability}. Considering the lift of this random walk in $\Gamma$, we have to say that after $O(\delta_0 \deg f)$-many steps, the probability of landing in $\lcal_{\delta_0}(H)$ is small. By Proposition \ref{prop:SmallLiftsInvariantTheoretic}, it is enough to make sure that the probability of landing in a proper algebraic subgroup of $\bbg$ is small. In this section, we point out that the characteristic of the involved fields are irrelevant in the {\em ping-pong type} argument in \cite[Section 3.2]{SGV}, and we get similar statements in the global function field case. After having the needed {\em ping-pong players}, using Proposition~\ref{propreps} we end up getting a finite symmetric subset $\Omega_0$ such that a random walk with respect to the probability counting measure on $\Omega_0$ has an exponentially small chance of landing in a proper algebraic subgroup of $\bbg$. In this note, we do not repeat any of the proofs presented in \cite{Var, SGV}, and we refer the readers to those articles for the details of the arguments.

For a subset $\Omega'$ of a group and a positive integer $l$, we let
\[
B_l(\Omega'):=\{g_1\cdots g_l|\h g_i\in \Omega'\cup \Omega'^{-1}, g_i\neq g_{i+1}^{-1}\};  
\]
so the support of the $l$-step random-walk with respect to the probability counting measure on $\Omega'\cup \Omega'^{-1}$ is $\bigcup_{2k\le l} B_{l-2k}(\Omega')$. 

\begin{prop}\label{Propping1}  Let $\Gamma,\bbg$ be as in the {\bf Standing assumptions}. Let $\kcal_i$, $\kcal_j'$, $\rho_i$, and $\rho_j'$ be as in Proposition~\ref{prop:SmallLiftsInvariantTheoretic}. 
   Then there exists a subset $\Omega'\subset \Gamma$ that freely generates a subgroup $\Gamma'$ with the following properties:
\begin{enumerate}
\item For any $i$ and any non-zero vector $v\in \bbv_i(\kcal_i)$, 	
\[
|\{g\in B_\ell(\Omega')|\rho_i(g)([v])=[v]\}| < |B_\ell(\Omega')|^{1-c'}.
\]
\item For any $j$ and any point $w\in \bbw_j(\kcal_j')$
\[
|\{g\in B_\ell(\Omega')|\rho_j'(g)(w)=w\}| < |B_\ell(\Omega')|^{1-c'}.
\]
\end{enumerate}
where $c'$ is a constant depending only on $\Omega'$ and the representations.
\end{prop}
\begin{proof}
See proof of \cite[Proposition 20]{SGV}.
\end{proof}

\subsection{Escaping purely structural subgroups: finishing proof of Proposition \ref{propmainescape}}
This proof is almost identical to the proof of \cite[Proposition 7]{SGV}. Let $\Gamma, \gcal, \bbg,$ and $f$ be as in the {\bf Standing assumptions}. Let $\Omega'$ be the set given by Proposition \ref{Propping1}.  Suppose $H\subseteq\pi_f(\Gamma)$ is a purely structural subgroup. Let $\delta$ be as in Proposition~\ref{prop:SmallLiftsInvariantTheoretic}. 

As $\pi_f[\pcal_{\Omega'}]^{(l)}(H)^2\le \pi_f[\pcal_{\Omega'}]^{(2l)}(H)$, it is enough to prove the claim for even positive integers $l$. We notice that for any positive integer $l$
\[
\pi_f[\pcal_{\Omega'}]^{(2l)}(H)= \pcal_{\Omega'}^{(2l)}\left(\bigcup_{0\le k\le l}(\pi_f^{-1}(H)\cap B_{2l-2k}(\Omega'))\right);
\]
and for any $\gamma\in \pi_f^{-1}(H)\cap B_l(\Omega')$, $\|\gamma\|\le (\max_{w\in\Omega'} \|w\|)^l$. Hence for $l\ll_{\Omega'} \delta \log[\pi_f(\Gamma):H]$ and $\deg f\gg_{\Omega'} 1$ we have
\be\label{eq:random-walk-lift-bound}
\pcal_{\pi_f(\Omega')}^{(2l)}(H)\le \sum_{0\le k\le l} \pcal_{\Omega'}^{(2l)}(\lcal_{\delta}(H)\cap B_{2l-2k}(\Omega')).
\ee
We notice that, since $\Omega'=\Omega'_0\sqcup\Omega_0'^{-1}$ and $\Omega'_0$ freely generates a subgroup, for $\gamma, \gamma'\in B_{2r}(\Omega')$ we have $\pcal_{\Omega'}^{(2l)}(\gamma)=\pcal_{\Omega'}^{(2l)}(\gamma')$; let $P_l(r):=\pcal_{\Omega'}^{(2l)}(\gamma)$ for some $\gamma\in B_{2r}(\Omega')$. Hence by \eqref{eq:random-walk-lift-bound} we have
\be
\pcal_{\pi_f(\Omega')}^{(2l)}(H)\le \sum_{0\le r\le l} |\lcal_{\delta}(H)\cap B_{2r}(\Omega')| P_l(r).
\ee

Combining Propositions \ref{prop:SmallLiftsInvariantTheoretic} and \ref{Propping1}, we have
  \be\label{eq:upper-bound-number-of-elements-small-lifts}
  |\lcal_{\delta}(H)\cap B_{2r}(\Omega')|< |B_{2r}(\Omega')|^{1-c'}
  \ee
  where $c'$ is the constant from Proposition~\ref{Propping1}. 
  
 Let us recall a few well-known results related to random-walks (in a free group); for any $\gamma\in \langle \Omega'\rangle$, by Cauchy-Schwarz inequality, we have 
 \be\label{eq:return-to-identity}
 \pcal_{\Omega'}^{(2l)}(\gamma)=\sum_{\gamma'} \pcal_{\Omega'}^{(l)}(\gamma') \pcal_{\Omega'}^{(l)}(\gamma'^{-1}\gamma)\le \|\pcal_{\Omega'}^{(l)}\|_2^2=\sum_{\gamma'} \pcal_{\Omega'}^{(l)}(\gamma')\pcal_{\Omega'}^{(l)}(\gamma'^{-1})=\pcal_{\Omega'}^{(2l)}(I)
 \ee
 where $I$ is the identity matrix; and so $P_l(r)\le P_l(0)$ for any non-negative integer $r$. Since $P_{l_1}(0)P_{l_2}(0)\le P_{l_1+l_2}(0)$, we have that $\{\sqrt[l]{P_l(0)}\}_l$ is a non-decreasing sequence. Hence by Kesten's result~\cite[Theorem 3]{Kes}, we have
 \be\label{eq:Kesten}
 P_l(r)\le P_l(0)\le \left(\frac{2M-1}{M^2}\right)^l,
 \ee
 where $|\Omega'|=2M$. We also have $|B_{2r}(\Omega')|=2M(2M-1)^{2r-1}$. Therefore by \eqref{eq:random-walk-lift-bound}, \eqref{eq:upper-bound-number-of-elements-small-lifts}, and \eqref{eq:Kesten}, for $l=\Theta_{\Omega'}([\pi_f(\Gamma):H])$, we have
\begin{align*}
\pcal_{\pi_f(\Omega')}^{(2l)}(H)\le & \sum_{0\le r\le l/20} |\lcal_{\delta}(H) \cap B_{2r}(\Omega')| P_l(0) + \sum_{l/20<r\le l} |\lcal_{\delta}(H) \cap B_{2r}(\Omega')| P_l(r)
\\
\le &
\left( 1+ 2M\sum_{1\le r\le l/20}(2M-1)^{2r-1}\right) \left(\frac{2M-1}{M^2}\right)^l+ 	\sum_{l/20<r\le l} |B_{2r}(\Omega')|^{1-c'} P_l(r) 
\\
\le &
\frac{(2M)^{11l/10+1}}{M^{2l}}+(2M(2M-1)^{l/10})^{-c'} \sum_{l/20<r\le l} |B_{2r}(\Omega')| P_l(r) 
\\ 
\le & \frac{(2M)^{11l/10+1}}{M^{2l}}+(2M(2M-1)^{l/10})^{-c'} \le [\pi_f(\Gamma):H]^{-O_{\Omega'}(1)}.
\end{align*}
Suppose, for a positive integer $l$, the desired inequality holds for $2l$; then 
\[
\pcal_{\pi_f(\Omega')}^{(l)}(gH)^2\le \pcal_{\pi_f(\Omega')}^{(2l)}(H)\le [\pi_f(\Gamma):H]^{-\delta_0}
\]
 for any $g\in \pi_f(\Gamma)$. Hence for any $l'\ge l$ we have
\[
\pcal_{\pi_f(\Omega')}^{(l')}(H)=\sum_{g\in \pi_f(\Gamma)} \pcal_{\pi_f(\Omega')}^{(l'-l)}(g^{-1}) \pcal_{\pi_f(\Omega')}^{(l)}(gH)\le [\pi_f(\Gamma):H]^{-\delta_0/2}.
\]
So it remains to show for large enough $c_0$, if $f\in S_{r_1,c_0}$, then $\pi_f(\Gamma)=\pi_f(\langle \Omega'\rangle)$. 

By Propositions~\ref{prop:SmallLiftsInvariantTheoretic}
 and \ref{Propping1}, we have that the group $\Gamma'$ generated by $\Omega'$ is Zariski-dense in $\bbg$. Let $\bbf_p(s(t)/r(t))$ be the trace field of $\Gamma'$. Then by \cite[Theorem 0.2, Theorem 3.7, Proposition 4.2]{PinkStrongApproximation} (or \cite[Theorem 1.1]{Wei}) we have that, if $f$ is a square-free polynomial with large degree irreducible factors (in particular, we can and will assume that $\gcd(f,r)=1$), then
 \[
 \pi_f(\Gamma')\simeq \prod_{\ell|f} \gcal_{\ell}(\bbf_p[s(t)/r(t)]/\langle \ell\rangle).
 \]
 Notice that $\bbf_p[s(t)/r(t)]/\langle \ell\rangle$ can be embedded into $\bbf_p[t]/\langle \ell\rangle$, and the degree of this extension is at most $[\bbf_p(t):\bbf_p(s(t)/r(t))]=\max(\deg s,\deg t)$. Hence $[\bbf_p[t]/\langle \ell\rangle:\bbf_p[s(t)/r(t)]/\langle \ell\rangle]$ is a divisor of $\deg \ell$ that is at most $\max(\deg s,\deg r)$. So if all the prime divisors of $\deg \ell$ are more than $c_0:=\max(\deg s,\deg r)$, then 
 \[
 \bbf_p[s(t)/r(t)]/\langle \ell\rangle=\bbf_p[t]/\langle \ell\rangle;
 \]
and claim follows.


\section{A variation of Varj\'{u}'s Product Theorem}\label{s:VarjuProductTheorem}  
In \cite{Var}, Varj\'{u} introduced a technique on proving a multi-scale product result for the direct product of an infinite family of certain finite groups. He provided a series of conditions for each one of the factors for this {\em gluing} process to work. One of the important conditions is on the structure of the subgroups of each factor; it was assumed that subgroups can be divided into $O(1)$ families of different {\em dimensions}. This condition was modeled from Nori's theorem on description of subgroups of $\GL_n(\bbf_p)$, which roughly says that any such subgroup is {\em very close} to being the $\bbf_p$-points of an algebraic subgroup. As we discussed in Section~\ref{s:RefinementOfLarsenPink}, subgroups of $\GL_n(\bbf_{p^m})$ might be either {\em structural} or {\em subfield type}; and the subfield type subgroups cannot be grouped into an $O_n(1)$ family of subgroups. We, however, 
use the fact that intersection of two conjugate subgroups of subfield type is a structural subgroup (see Corollary~\ref{cor:intersection-conjugate-subfield-type}), and modify Varj\'{u}'s axioms and arguments accordingly (see Proposition~\ref{propvarjuproduct}).     

Most of Varj\'{u}'s arguments and results stay the same even after the modifications of the assumptions; but we reproduce some of those arguments. It should be pointed out that there is an error in the proof of \cite[Corollary 14]{Var}. Our modified axioms help us to resolve this issue; Varj\'{u} has also communicated to us a way to correct the proof {\em without} changing the original assumptions.    

\subsection{Modified assumptions} 
Before stating our modified assumptions, let us introduce a notation and recall the definition of {\em quasi-random} groups (this concept was introduced by Gowers~\cite{Gow}). For two subgroups $H$ and $H'$ of a finite group $G$ and a positive integer $L$, we write $H\preceq_L H'$ if $[H:H'\cap H]<L$. 

\begin{definition}\label{QRdef}  
For a positive constant $c$, we say a finite group $G$ is {\em $c$-quasi-random} if for any non-trivial irreducible representation $\rho$ of $G$ we have $\dim\rho>|G|^c$.
\end{definition}

Our set of axioms depend on two parameters $L$ and $\delta_0$, where $L$ is a positive integer and $\delta_0:\bbr^+\rightarrow \bbr^+$ is a function. 

{\bf Assumptions (V1)$_{L}$-(V3)$_{L}$ and (V4)$_{\delta_0}$}

\begin{enumerate}
\item[(V1)$_L$] $G$ is an almost simple group with $|Z(G)|<L$.
\item[(V2)$_L$] $G$ is $L^{-1}$-quasi-random (see Definition \ref{QRdef}).
\item[(V3)$_L$] There exists an integer $m<L$, and classes of proper subgroups $\cal_j$ for $1\le j\le m$ and $\cal_i'$ for $1\le i\le m'$ where $m'\le L \log |G|$ with the following properties:

\begin{enumerate}[(i)]
\item For each $i$, $\cal_i$ and $\cal_i'$ are closed under conjugation by elements in $G$.
\item $\cal_0=\{Z(G)\}$.
\item For each proper subgroup $H$ of $G$ there exist an index $i$ and a subgroup $H^\sharp\in \cal_i$ or $\cal_i'$ such that $H\preceq_L H^\sharp$.
\item For each $i$ and for each pair of distinct subgroups $H_1$, $H_2\in \cal_i$, there exists $j<i$ and a subgroup $H^\sharp\in \cal_j$ such that $H_1\cap H_2\preceq_L H^\sharp$. For any $H\in \cal_i$, there is $j$ and $H^\sharp\in \cal_j$ such that $N_G(H)\preceq_L H^\sharp$.
\item For each $i$ and for each pair of distinct subgroups $H_1'$ and $H_2' \in \cal_i'$, there exists $j$ and a subgroup $H^{\sharp}\in \cal_j$ such that $H_1'\cap H_2'\preceq_L H^\sharp$. For any $H\in \cal'_i$, $[N_G(H):H]\le L$.
\end{enumerate}
\item[(V4)$_{\delta_0}$] if $S\subseteq G$ is a generating set and $|S|<|G|^{1-\vare}$ for a positive number $\vare$, then $|S\cdot S \cdot S|\ge |S|^{1+\delta_0(\vare)}$. 
\end{enumerate}

\begin{prop}\label{propvarjuproduct} For $L\in \bbz^+$, $\delta_0:\bbr^+\rightarrow\bbr^+$, suppose $\{G_i\}_{i=1}^{\infty}$ is a family of pairwise non-isomorphic finite groups that satisfy assumptions (V1)$_{L}$-(V3)$_{L}$ and (V4)$_{\delta_0}$. Then for any $\vare>0$, there is $\delta>0$ such that for any $n\in \bbz^+$ and any symmetric subset $S$ of $G:=\bigoplus_{i=1}^nG_i$ satisfying
\begin{equation}
|S|<|G|^{1-\vare}\mbox{ and } \pcal_S(gH)<[G:H]^{-\vare}|G|^\delta \text{ for any subgroup } H \text{ of } G \text{ and } g\in G,
\end{equation}
we have 
\begin{equation*}
|\Pi_3S|\gg_\vare |S|^{1+\delta}.
\end{equation*}
\end{prop}
Let us reiterate that there are two key differences between Proposition \ref{propvarjuproduct} and \cite[Proposition 14]{Var}: (1)
In Varj\'{u}'s setting we have only $O(L)$ families of proper subgroups, and this parameter resembles dimension of an algebraic subgroup. In our setting, however, we have two types of families of proper subgroups, and only one of the types can have at most $O(L)$ families of proper subgroups. Theses types resemble the structural and the subfield type subgroups. For the structural subgroups we more or less use dimension of the underlying algebraic groups to parametrize them, and for subfield type subgroups the order of the subfield gives us the needed parametrization. It is clear that in this case the number of such possible families can grow as $|G|$ goes to infinity; but it does not get more than $\log |G|$. (2) We are assuming a {\em product type} result for each factor (see (V4)$_{\delta_0}$) instead of an $l^2$-flattening assumption for measures with {\em large} $l^2$-norm (see (A4) in \cite[Section 3]{Var}). This modification helps us resolve the mentioned error in \cite[Corollary 14]{Var}.   

\subsection{A detailed overview of Varj\'{u}'s proof}
Before getting to the multi-scale setting of Proposition~\ref{propvarjuproduct}, we recall Bourgain-Gamburd's result which gives us a way to measure how product of two random variables gets {\em substantially more random} unless there is an algebraic obstruction (see \cite[Proposition 2]{BG1} and \cite[Lemma 15]{Var}). 

\begin{lem}\label{lemmaBGapproxsub}  Let $\mu$ and $\nu$ be two probability measures on an arbitrary finite group $G$, and let $K$ be a real number greater than $2$.  If
$\|\mu\ast \nu\|_2> \frac{1}{K}\|\mu\|_2^{1/2}\|\nu\|_2^{1/2}$,
then there is a symmetric subset $A\subseteq G$ with the following properties:
\begin{enumerate}
\item (Size of $A$ is comparable with $\|\mu\|_2^{-2}$) $K^{-R}\|\mu\|_2^{-2}\le |A|\le K^{R}\|\mu\|_2^{-2}	$.
\item (An approximate subgroup) $|A\cdot A\cdot A|\le K^R|A|$.
\item (Almost equidistribution on $A$) $\min_{a\in A} (\wt{\mu}\ast \mu)(a)\ge K^{-R}|A|^{-1}$,
\end{enumerate}
where $R$ is a universal constant and $\wt{\mu}(g):=\mu(g^{-1})$.
\end{lem} 
One can use various forms of entropy to quantify how random a measure is. 

\begin{definition}
Suppose $X$ is a random variable on a finite set $S$ and has distribution $\mu$; then the (Shannon) {\em entropy} of $X$ is 
\[
H(X):=\sum_{s\in S}-\log (\bbp(X=s)) \bbp(X=s),
\]
where $\bbp(X=s)$ is the probability of having $X=s$. The {\em R\'{e}nyi entropy} of $X$ is 
\[
H_2(X):=-\log \left(\sum_{s\in S} \bbp(X=s)^2\right)=-\log \|\mu\|_2^2. 
\]	
We let $H_{\infty}(X):=-\log (\max_{s\in  {\rm supp}(X)} \bbp(X=s))$ and $H_0(X):=\log |{\rm supp}(X)|$, where ${\rm supp}(X)$ is the support of $X$. 

Suppose $Y$ is another random variable on $S$. Then {\em the entropy of $X$ conditioned to $Y$} is
\begin{align}
\notag H(X|Y):=&\sum_{y\in S} \bbp(Y=y) H(X|Y=y) \\
=& -\sum_{y\in S} \bbp(Y=y) \sum_{x\in S}\bbp(X=x|Y=y) \log\bbp(X=x|Y=y),
\end{align}
where $X|Y=y$ is the random variable $X$ conditioned to the random variable $Y$ taking a certain value $y$, and $\bbp(X = x|Y = y)$ is the probability of having $X = x$ conditioned to $Y = y$. The {\em R\'{e}nyi entropy of $X$ conditioned to $Y$} is 
\[
H_2(X|Y):=\sum_{y\in S} \bbp(Y=y) H_2(X|Y=y).
\] 
\end{definition}
 Here are some of the basic properties of entropy that will be used in this note.
\begin{lem}\label{lem:properties-entropy}
	Suppose $S$ is a finite set, and $X$ and $Y$ are random variables with values in $S$. Then
	\begin{enumerate}
	\item $H(X,Y)=H(X)+H(Y|X)$.
	\item $H(X)\ge H(X|Y)$.
	\item $H_0(X)\ge H(X)\ge H_2(X)\ge H_{\infty}(X)$.
	\item $H(X|f(Y))\ge H(X|Y)$ where $f$ is a function.
	\end{enumerate}
\end{lem}
\begin{proof}
	These are all well-known facts; for instance see \cite[Theorem 2.4.1, Theorem 2.5.1, Theorem 2.6.4, Lemma 2.10.1, Problem 2.1]{CT}.
\end{proof}
It is very intuitive to say that the product of two independent random variables with values in a group should be at least as random as the initial random variables. The next lemma says that this intuition is compatible with how various types of entropy measure the randomness of a distribution.

\begin{lem}\label{lem:product-entropy-trivial-bound}
	Suppose $X$ and $Y$ are two independent random variables with values in a group $H$ and finite supports. Then $H_i(XY)\ge \max(H_i(X),H_i(Y))$ for $i\in \{0,1,2,\infty\}$ where $H_1(X):=H(X)$. 
\end{lem}
\begin{proof}
Notice that ${\rm supp}(XY)={\rm supp}(X){\rm supp}(Y)$; and so $H_0(XY)\ge \max(H_0(X),H_0(Y))$. 

For any $h\in H$, we have 
\[
\bbp(XY=h)=\sum_{x\in H} \bbp(X=x)\bbp(Y=x^{-1}h) \le \max_{y\in H} \bbp(Y=y);  
\] 
and so $H_{\infty}(XY)\ge H_{\infty}(Y)$. By symmetry we get the claim for $i=\infty$.

Since the function $x^2$ is a convex function, we have 
\begin{align*}
	\bbp(XY=h)^2= & (\sum_{x\in H} \bbp(X=x)\bbp(Y=x^{-1}h))^2 \\
	\le & \sum_{x\in H} \bbp(X=x)\bbp(Y=x^{-1}h)^2.
\end{align*}
Therefore $\sum_{h\in H}\bbp(XY=h)^2\le \sum_{h\in H} \sum_{x\in H} \bbp(X=x)\bbp(Y=x^{-1}h)^2= \sum_{y\in H} \bbp(Y=y)^2$, which implies the claim for $i=2$.

By Lemma~\ref{lem:properties-entropy}, we have
\[
H(XY)\ge H(XY|Y)=H(X|Y)=H(X);
\]
and claim follows.
\end{proof}

Lemma~\ref{lemmaBGapproxsub} says how much the R\'{e}nyi entropy of product of two independent variables increases unless there is an algebraic obstruction: if $X$ and $Y$ are two independent random variables with values in a group $G$, then we have 
\be\label{eq:RenyEntropyIncrease}
H_2(XY)\ge \frac{H_2(X)+H_2(Y)}{2}+\log K,
\ee
 unless there is a symmetric subset $A$ of $G$ such that 
 \[
|\log |A|-H_2(X)|\le R\log K,\hspace{0.5cm}  
 |A\cdot A\cdot A|\le K^R|A|,
 \]
and for any $a\in A$ 
 \[
 \bbp(X'^{-1}X=a)\ge K^{-R} |A|^{-1},
 \]
 where $X'$ is a random variable with identical distribution as $X$ and it is independent of $X$. Based on this result one can prove a meaningful increase in the R\'{e}nyi entropy of product of two independent random variables with a {\em Diophantine type} condition with values in a group that has a {\em product type property} (similar to the condition (V4)$_{\delta_0}$).

\begin{definition}
Suppose $G$ is a finite group and $X$ is a random variable with values in $G$. We say $X$ is of $(\alpha,\beta)${\em -Diophantine type} if for any proper subgroup $H$ of $G$ with $|H|\ge |G|^{\alpha}$ and for any $g\in G$, we have $\bbp(X\in gH)\le [G:H]^{-\beta}$. 	
\end{definition}

\begin{lem}\label{lem:GrowthOfRenyiEntropySingleScale}
Suppose $G$ is a finite group and $X$ and $Y$ are two independent random variables with values in $G$. Suppose $G$ satisfies the following properties:
\begin{enumerate}
\item {\em (Quasi-randomness)} It is an $L^{-1}$-quasi-random group for some positive integer $L$.
\item {\em (Product property)} For every positive number $\vare$, there is a positive number $\delta_0:=\delta_0(\vare)$ such that if $A$ is a generating set of $G$ and $|A|<|G|^{1-\vare}$, then $|A\cdot A\cdot A|\ge |A|^{1+\delta_0}$.
\end{enumerate}
Suppose the random variable $X$ satisfies the following properties:
\begin{enumerate}
\item {\em (Diophantine condition)} For some $\alpha,\beta>0$, $X$ is of $(\alpha,\beta)$-Diophantine type. 
\item {\em (Initial entropy)} $\alpha' \log |G|\le H_2(X)$ for some $\alpha'>2\alpha$.
\item {\em (Room for improvement)} $H_2(X)\le (1-\alpha'') \log |G|$ for some $\alpha''>0$.
\end{enumerate}
Then 
\[
H_2(XY)\ge \frac{H_2(X)+H_2(Y)}{2}+\gamma_0 \log |G|,
\]
where $\gamma_0$ is a positive constant that only depends on $\alpha',\alpha'',\beta$, and the function $\delta_0$. 
\end{lem}
\begin{proof}
Suppose $H_2(XY)<\frac{H_2(X)+H_2(Y)}{2}+\gamma \log |G|$ for some $\gamma>0$; then by Bourgain-Gamburd's result and the above discussion there is a symmetric subset $A$ of $G$ such that 
\begin{align}
\label{eq:order-of-almost-subgroup} |\log |A|-H_2(X)| &\le R \gamma \log |G| & \text{(Controlling the order)} \\
\label{eq:almost-subgroup} |A\cdot A \cdot A|&\le |G|^{R\gamma} |A| & \text{(almost subgroup)} \\
\label{eq:almost-equidistributed} \forall a\in A, \bbp(X'^{-1}X=a)&\ge |G|^{-R\gamma} |A|^{-1} &\text{(almost equidistribution)}
\end{align}
where $X'$ is a random variable with identical distribution as $X$ and it is independent of $X$ and $R$ is an absolute positive constant. Let $H$ be the group generated by $A$. Then by \eqref{eq:almost-equidistributed} 
\be\label{eq:probability-of-hitting-H}
\bbp(X\in H)\ge \bbp(X'^{-1}X\in H)^{1/2}\ge \bbp(X'^{-1}X\in A)^{1/2}\ge |G|^{-R\gamma/2},
\ee
and by the lower bound on the R\'{e}nyi entropy of $X$
\be\label{eq:order-of-subgroup}
|H|\ge |A|\ge |G|^{-R\gamma}e^{H_2(X)}\ge |G|^{\alpha'/2}
\ee
for $\gamma\le \frac{\alpha'}{2R}$. Since $X$ is of $(\alpha,\beta)$-Diophantine type and $\alpha'>2\alpha$, by \eqref{eq:order-of-subgroup} and \eqref{eq:probability-of-hitting-H} we get 
\be\label{eq:Diophantine-condition}
[G:H]^{-\beta}\ge \bbp(X\in H)\ge |G|^{-R\gamma/2}.
\ee
Since $G$ is an $L^{-1}$-quasi-random group, we $[G:H]\ge |G|^{1/L}$ if $H$ is a proper subgroup; and so by \eqref{eq:probability-of-hitting-H} and \eqref{eq:Diophantine-condition} we get
\[
\gamma \ge \frac{\beta}{2RL}.
\]
Therefore for $\gamma\le \frac{\beta}{4RL}$, we have $G=H$, which means $A$ is a generating set of $G$.

By the upper bound on the R\'{e}nyi entropy of $X$ and \eqref{eq:order-of-almost-subgroup} we have
\[
|A|\le |G|^{1-\alpha''}|G|^{R\gamma}\le |G|^{1-\frac{\alpha''}{2}}
\]
for $\gamma\le \frac{\alpha''}{2R}$. Hence by the product property of $G$ there is $\delta_0:=\delta_0(\alpha''/2)$ such that
\be\label{eq:propduct-propety-implies}
|A\cdot A\cdot A|\ge |A|^{1+\delta_0}.
\ee 
 By \eqref{eq:almost-subgroup} and \eqref{eq:propduct-propety-implies}  we deduce that
 \[
 |G|^{R\gamma}\ge |A|^{\delta_0}; 
 \]
 together with \eqref{eq:order-of-almost-subgroup} and the lower bound on the R\'{e}yi entropy of $X$ we get
 \[
 |G|^{R\gamma}\ge |A|^{\delta_0}\ge (|G|^{-R\gamma}e^{H_2(X)})^{\delta_0} \ge |G|^{\alpha'\delta_0/2}.
 \]
 Hence we deduce that for $\gamma=\gamma_0:=\min(\frac{\alpha'\delta_0(\alpha''/2)}{4R}, \frac{\beta}{4RL})$, we have
 \[
 H_2(XY)\ge \frac{H_2(X)+H_2(Y)}{2}+\gamma \log|G|;
 \]
 and claim follows.
\end{proof}

Now suppose $X_{j}:=(X_j^{(i)})_{i=1}^n$'s are i.i.d. random variables with values in $G:=\bigoplus_{i=1}^n G_i$  and distribution $\pcal_A$. We notice that (see Lemma~\ref{lem:properties-entropy})
\[
\textstyle
\log |\prod_l A|=H_0(X_{1}\cdots X_{l})\ge H(X_{1}\cdots X_{l});
\]  
and by the mentioned basic properties of entropy (see Lemma~\ref{lem:properties-entropy}) we have
\begin{align*}
	H(X_{1}\cdots X_{l})= &
	\sum_{j=1}^n H(X_1^{(j)}\cdots X_l^{(j)}| X_1^{(1)}\cdots X_l^{(1)},\ldots,X_{1}^{(j-1)}\cdots X_{l}^{(j-1)})\\
	\ge &
	\sum_{j=1}^n H(X_1^{(j)}\cdots X_l^{(j)}| X_i^{(k)}, 1\le i\le l, 1\le k\le j-1)\\
	\ge &
	\sum_{j=1}^n H_2(X_1^{(j)}\cdots X_l^{(j)}| X_i^{(k)}, 1\le i\le l, 1\le k\le j-1).	
	\end{align*}

At this point we are almost at the setting of Bourgain-Gamburd's result, and we would like to apply Lemma~\ref{lem:GrowthOfRenyiEntropySingleScale}. By (V2)$_{L}$ and (V4)$_{\delta_0}$, $G_j$ does satisfy Lemma~\ref{lem:GrowthOfRenyiEntropySingleScale}'s conditions; but the random variables $X_j^{(i)}$'s do not necessarily satisfy the required conditions. Here are the steps that we take to get the desired conditions:

{\bf Step 1}. By a regularization argument, we find a subset $A$ of $S$ such that 

(a) for any $(g_1,\ldots,g_{j-1})\in \bigoplus_{k=1}^{j-1} G_k$ the conditional random variables $X_i^{(j)}|X_i^{(k)}=g_k, 1\le k\le j-1$ are uniformly distributed in their support. 

(b) $H(X_i^{(j)}|X_i^{(k)}=g_k, 1\le k\le j-1)$ is the same for any $(g_1,\ldots,g_{j-1}) \in \pr_{[1..j-1]}(A)$ where $\pr_I:\bigoplus_{k=1}^n G_k\rightarrow \bigoplus_{k\in I} G_k$ is the projection map. 

(c) (Initial entropy) For any $(g_1,\ldots,g_{j-1}) \in \pr_{[1..j-1]}(A)$, either $H(X_i^{(j)}|X_i^{(k)}=g_k, 1\le k\le j-1)=0$ or $H(X_i^{(j)}|X_i^{(k)}=g_k, 1\le k\le j-1)\ge \alpha \log |G_j|$. 

(d) $\log |A|>\log |S|-2\alpha \log |G|$.

This process (more or less) gives us the {\em initial entropy} condition. 

{\bf Step 2}. At this step, we focus on the {\em scales} where the entropy is already large and does not have much {\em room for improvement}. In the influential work~\cite{Gow} where Gowers defined quasi-random groups, he proved the following result (see \cite[Theorem 3.3]{Gow} and also \cite[Corollary 1]{NP}).
\begin{thm}\label{thm:Gowers}
Suppose $G$ is an $L^{-1}$-quasi-random group	. Suppose $X_1,X_2,X_3$ are three independent random variables with values in $G$. If \[\frac{H_0(X_1)+H_0(X_2)+H_0(X_3)}{3}>(1-\frac{1}{3L})\log |G|,\] then $H_0(X_1X_2X_3)=\log |G|$.
\end{thm}
We apply Theorem~\ref{thm:Gowers} for the conditional random variables $X_i^{(j)}|X_i^{(k)}=g_i^{(k)}, 1\le k\le j-1$ for 
$(g_i^{(1)},\ldots,g_{i}^{(j-1)}) \in \pr_{[1..j-1]}(A)$ and at the {\em scales} where 
\be\label{eq:large-entropy}
H(X_i^{(j)}|X_i^{(k)}=g_i^{(k)}, 1\le k\le j-1)\ge (1-\frac{1}{3L})\log |G|,
\ee
and deduce that $\bigoplus_{i\in I_{\rm l}}G_i=\pr_{I_{\rm l}}(A\cdot A\cdot A)$ where $I_{\rm l}$ consists of $j$'s such that \eqref{eq:large-entropy} holds. Next we let $I_{\rm s}:=[1..n]\setminus I_{\rm l}$; and define the following metric on $\bigoplus_{i\in I_{\rm s}}G_i$
\[
d(g,g'):=\sum_{i\in I_{\rm s}, \pr_i(g)\neq \pr_i(g')} \log |G_i|. 
\]
Let $T:=\max \{d(g_{\rm s},1)|\h g_{\rm s}\in \pr_{I_{\rm s}}(\prod_9 S\cap \{1\}\oplus \bigoplus_{i\in I_{\rm s}}G_i)\}$. Then one gets a $T$-almost group homomorphism $\psi:\bigoplus_{i\in I_{\rm l}}G_i\rightarrow \bigoplus_{i\in I_{\rm s}}G_i$. By a result of Farah~\cite{Far} on approximate homomorphisms, $\psi$ should be close to a group homomorphism. Based on this and certain Diophantine property of $S$, one can deduce that 
\[
\textstyle
\exists (1,g_{\rm s})\in \prod_9 S\cap \{1\}\oplus \bigoplus_{i\in I_{\rm s}}G_i, d(g_{\rm s},1)\gg \vare^2 \log |G|.
\]
Now considering $H:=C_G((g_{\rm s},1))$ and using the assumed upper bound of $\pcal_S(gH)$, one gets a strong lower bound for $|\prod_{14} S|$  unless almost all the {\em scales} do have {\em room for improvement}. 

{\bf Step 3}. At this step we focus on the {\em scales} where there is an {\em initial entropy} and {\em room for improvement} as required in Lemma~\ref{lem:GrowthOfRenyiEntropySingleScale}. The last condition that is needed is a {\em Diophantine type} condition. Varj\'{u} (essentially) proves the following result in order to deal with this issue.  
\begin{prop}\label{prop:generic-Diophantine-property}
	Suppose $L$ is a positive integer, $G$ is a finite group that satisfies properties (V1)$_L$-(V3)$_L$. Let $m$ be as in (V3)$_{L}$. Suppose $X_1,\ldots,X_{2^{m+1}}$ are independent random variables with values in $G$ and $H_{\infty}(X_i)\ge \alpha' \log |G|$ for some positive number $\alpha'$ and any index $i$. For $\overrightarrow{y}:=(y_1,\ldots,y_{2^{m+1}-1})\in \bigoplus_{i=1}^{2^{m+1}-1} G$, let $X_{\vecy}:=X_1 y_1 X_2y_2\cdots y_{2^{m+1}-1} X_{2^{m+1}}$. Suppose $Y_1,\ldots,Y_{2^{m+1}-1}$ are i.i.d. random variables with values in $G$. Suppose $Y_i$'s are of $(\alpha,\beta)$-Diophantine type for some positive numbers $\alpha$ and $\beta$ such that $\beta\ge 4\alpha$; further, assume that for any $g\in G$ and $H\in \bigcup_{i=1}^{m} \cal_i$, $\bbp(Y_1\in gH)\le [G:H]^{-\beta}$. Then assuming $|G|\gg_{\alpha',\beta,L} 1$, we have
	\[
	\bbp((Y_1,\ldots,Y_{2^{m+1}-1})=\vecy \text{ such that } X_{\vecy} \text{ is not of } (0,\beta'/2)\text{-Diophantine type}) \le |G|^{-\frac{\beta}{4L}} 
	\] 
	where $\beta':=\frac{1}{8^{m+1}}\min(\frac{\beta}{5L},\frac{\alpha'}{2})$.
\end{prop}
Using Proposition~\ref{prop:generic-Diophantine-property}, Lemma~\ref{lem:product-entropy-trivial-bound}, and Lemma~\ref{lem:GrowthOfRenyiEntropySingleScale}, one gets the following result.

\begin{prop}\label{prop:VarjuMutation}
Suppose $L$ is a positive integer and $\delta_0:\bbr^+\rightarrow\bbr^+$ is a function. Suppose $G$ satisfies conditions (V1)$_{L}$-(V3)$_{L}$	 and (V4)$_{\delta_0}$. Let $m$ be as in the condition (V3)$_{L}$. Suppose random variables $X_1,\ldots,X_{2^{m+1}+1}$ satisfy the following properties:
\begin{enumerate}
\item {\em (Initial entropy)} $\alpha' \log |G|\le H_{\infty}(X_i)$ for some $\alpha'>0$ and any index $i$.
\item {\em (Room for improvement)} $H_2(X_i)\le (1-\alpha'')\log |G|$ for some $\alpha''>0$.
\end{enumerate}
Suppose the i.i.d. random variables $Y_1,\ldots,Y_{2^{m+1}-1}$ satisfy the following property:
\begin{center}
{\em (Diophantine condition)} For some $0\le \alpha<\min(\beta/4,\alpha'/2)$, $Y_1$ is of $(\alpha,\beta)$-Diophantine type; and for any $H\in \bigcup_{i=1}^m \cal_i$ and $g\in G$, $\bbp(Y_1\in gH)\le [G:H]^{-\beta}$.	
\end{center}
Then assuming $|G|\gg_{\alpha',\alpha'',\beta,L,\delta_0}1$ we have
\[
H_2(X_1Y_1X_2\cdots Y_{2^{m+1}-1}X_{2^{m+1}}X_{2^{m+1}+1}|Y_1,\ldots,Y_{2^{m+1}-1})\ge \min_i H_2(X_i)+ \gamma \log |G|
\]
where $\gamma$ is a positive constant that only depends on $\alpha',\alpha'', \beta, L$, and the function $\delta_0$.
\end{prop}
Finally Varj\'{u} finds a subset $B$ of $S$ such that, if $Y=(Y^{(1)},\ldots,Y^{(n)})$ is a random variable with distribution $\pcal_B$, then for {\em lots of} $i$'s $Y^{(i)}$ is of $(0,\vare')$-Diophantine type where $\vare'\gg_{\vare,L} 1$ ($\vare$ and $L$ are given in Proposition~\ref{propvarjuproduct}); overall one gets 
\[\textstyle \log |\prod_{2^{m+2}} S|-\log |S|\gg_{\vare,L} \log|S|.\] One can finish the proof of Proposition~\ref{propvarjuproduct} using \cite[Lemma 2.2]{Hel1} which says 
\[\textstyle (k-2)(\log |\prod_3 S|-\log |S|)\ge \log|\prod_{k} S|-\log |S|\]
for any integer $k\ge 3$. 
\subsection{Regularization and a needed inequality}\label{ss:regular-subset} Let $L$, $\delta_0$, and $\{G_i\}_{i=1}^{\infty}$ be as in the statement of Proposition~\ref{propvarjuproduct}. Since $G_i$'s are pairwise non-isomorphic, $\lim_{i\rightarrow \infty}|G_i|=\infty$.

\begin{lem}\label{lem:required-inequality}
 Suppose $m:\bbr^+\rightarrow \bbz^+$ is a function. If claim of Proposition~\ref{propvarjuproduct} holds for $\vare$, $\delta(\vare)$, and the subfamily $\{G_i|\h 1\le i, |G_i|> m(\vare)\}$, then Proposition~\ref{propvarjuproduct} holds with $\delta(\vare)/2$ for $\delta$ and a possibly larger implied constant in the final claimed inequality.
\end{lem}

Let us remark that $\delta$ also depends on $L$ and $\delta_0$; but we are assuming that those are fixed in the entire section.

\begin{proof}[Proof of Lemma~\ref{lem:required-inequality}]
Suppose $S$ is a symmetric subset of $G:=\bigoplus_{i=1}^n G_i$ such that 
\be\label{eq:assumption}
|S|<|G|^{1-\vare} \text{ and } \pcal_S(gH)<[G:H]^{-\vare}|G|^{\delta(\vare)/2}
\ee
for any subgroup $H$ of $G$ and $g\in G$. Let 
\[
N:=\bigoplus_{|G_i|\le m(\vare), 1\le i\le n} G_i.
\]
Since $G_i$'s are pairwise non-isomorphic, $|N|<f(\vare)$ for some function $f:\bbr^+\rightarrow \bbz^+$. Let $\overline{S}:=\pi_N(S)$ where 
\[
\pi_N:G\rightarrow \overline{G}:=\bigoplus_{|G_i|> m(\vare), 1\le i\le n} G_i
\]
 is the natural projection. For any subgroup $\overline{H}$ of $\overline{G}$ and $\overline{g}\in \overline{G}$, by \eqref{eq:assumption} we have
 \[
 \pcal_S((\overline{g},1)\overline{H}\oplus N)<[G:\overline{H}\oplus N]^{-\vare/2}|G|^{\delta(\vare)/2};
 \]
 and so 
 \[
 \frac{1}{|N|}\pcal_{\overline{S}}(\overline{g}\overline{H})\le 
 \frac{|\overline{S}|}{|S|}\pcal_{\overline{S}}(\overline{g}\overline{H})
 <
 [\overline{G}:\overline{H}]^{-\vare}
 |\overline{G}|^{\delta(\vare)/2}|N|^{\delta(\vare)/2}.
 \]
This implies that 
\[
\pcal_{\overline{S}}(\overline{g}\overline{H})\le [\overline{G}:\overline{H}]^{-\vare}
 |\overline{G}|^{\delta(\vare)/2}f(\vare)^{1+\delta(\vare)/2}.
\]
If $|\overline{G}|>f(\vare)^{\frac{1+\delta(\vare)/2}{\delta(\vare)/2}}$, then we get
$
\pcal_{\overline{S}}(\overline{g}\overline{H})\le [\overline{G}:\overline{H}]^{-\vare}
 |\overline{G}|^{\delta(\vare)}.
$ 
Therefore by our assumption 
\[
|\overline{S}|^{1+\delta(\vare)} \le C(\vare)|\overline{S}\cdot \overline{S}\cdot \overline{S}|.
\]
Hence
\[
|S|^{1+\delta(\vare)}\le f(\vare)^{1+\delta(\vare)} C(\vare) |S\cdot S\cdot S|.
\]
If $|\overline{G}|<f(\vare)^{\frac{1+\delta(\vare)/2}{\delta(\vare)/2}}$, then $|S|\le |G|<f(\vare)^{1+\frac{1+\delta(\vare)/2}{\delta(\vare)/2}}$. Overall we get
\[
|S|^{1+\delta(\vare)}\le C'(\vare) |S\cdot S\cdot S|,
\]
where $C'(\vare):=\max\{f(\vare)^{2+\frac{2+\delta(\vare)}{\delta(\vare)/2}}, f(\vare)^{1+\delta(\vare)} C(\vare)\}$; and claim follows.
\end{proof}
We show that for small enough $\vare$ we can take \be\label{eq:initial-delta}
\delta(\vare):=\min\{\vare^5,1\}/8L.
\ee
For the given $\delta(\vare)$ and a positive valued function $C''(\vare)$, we let 
\[
m(\vare):=\sup\{x\in \bbr^+|\h C''(\vare) \log x\ge x^{\delta(\vare)^2}\}.
\] 
By Lemma~\ref{lem:required-inequality}, we can and will assume that 
\be\label{eq:required-inequality}
C''(\vare) \log |G_i|< |G_i|^{\delta(\vare)^2}
\ee
for any $i$. Throughout the proof of proposition~\ref{propvarjuproduct} we will be assuming inequalities of the type given in \eqref{eq:required-inequality}. 

As it is discussed in the beginning of \cite[Section 3.2]{Var}, passing to the groups $G_i/Z(G_i)$, using an argument similar to Lemma~\ref{lem:required-inequality} and based on an inequality of type \eqref{eq:required-inequality}, we can and will assume that $G_i$'s are simple groups.

For any non-empty subset $I$ of $[1..n]$, we let $G_I:=\bigoplus_{i\in I}G_i$; sometimes we view $G_I$ as a subgroup of $G_J$ when $I\subseteq J$. We let $G_{\varnothing}=\{1\}$. For any $I\subseteq J\subseteq [1..n]$, we let $\pr_I:G_J\rightarrow G_I$ be the natural projection map.
\begin{definition}\label{def:regular}
A subset $A$ of $\bigoplus_{i=1}^n G_i$ is called $(m_0,\ldots,m_{n-1})$-regular if for any $0\le k< n$ and $\overline{x}\in \pr_{[1..k]}(A)$ we have
\[
|\{x\in \pr_{[1..{k+1}]}(A)| \pr_{[1..k]}(x)=\overline{x}\}|=m_k.
\]	
\end{definition}
For a random variable $X$ with values in $\bigoplus_{i=1}^n G_i$, we write $X=(X_1,\ldots,X_n)$ and get random variables $X_i$ with values in $G_i$. 
\begin{lem}\label{lem:random-variables-regular-sets}
Suppose $A\subseteq \bigoplus_{i=1}^n G_i$ is an $(m_1,\ldots,m_n)$-regular subset. Let $X$ be a random variable with respect to the probability counting measure on $A$. Then
\begin{enumerate}
\item $\pr_{[1..k]}(X)$ is a random variable with respect to the probability counting measure on $\pr_{[1..k]}(A)$.
\item For any $(a_1,\ldots,a_n)\in A$, the conditional probability measure 
	\[\bbp(X_k|X_1=a_1,\ldots, X_{k-1}=a_{k-1})\] is a probability counting measure on a set of size $m_k$.	
\end{enumerate}
\end{lem}
\begin{proof}
Both of the above claims are easy consequences of the fact that $A$ is a regular set (See \cite[Lemma 22]{SG:sum-product}).
\end{proof}
The filtration $\{1\}=G_{\varnothing} \subseteq G_{\{1\}}\subseteq \cdots \subseteq G_{[1..i]} \subseteq \cdots \subseteq G_{[1..n]}$ gives us a rooted tree structure, where the vertices at the level $i$ are the elements of $G_{[1..i]}$; and the {\em children} of $(a_1,\ldots,a_i)$ are elements of $(a_1,\ldots,a_i)\oplus G_{i+1}$. To a non-empty subset $A$ of $G_{[1..n]}$, we associate the rooted subtree consisting of paths from the root to the elements of $A$. So a subset $A$ is $(m_0,\ldots,m_{n-1})$-regular precisely when the vertices at the level $i$ of the associated rooted tree of $A$ has exactly $m_i$ children.

 As it has been discussed in \cite[Section 3.2]{Var} by \cite[Lemma 5.2]{BGS} and inequality~\eqref{eq:required-inequality} (see also \cite[A.3]{BG3} and \cite[Section 2.2]{SG:SAI}) we get that there is a $(D_0,\ldots,D_{n-1})$-regular subset $A$ of $S$ such that the following holds.
 \begin{enumerate}
 \item For any $i$, either $D_i>|G_i|^{\delta}$ or $D_i=1$.
 \item $|A|>(\prod_{i=1}^n|G_i|)^{-2\delta} |S|$.	
 \end{enumerate}
 
 \subsection{Scales with no room for improvement} 
 This section is identical to \cite[Section 3.4]{Var}. The change in the assumptions has no effect in this part of the proof. We have decided to include the proofs for the convenience of the reader. 
 
 Let $I_{\rm l}:=\{i\in [0..n-1]|\h D_i\ge |G_i|^{1-1/(3L)}\}$, and $I_{\rm s}:=[0..n-1]\setminus I_{\rm l}$. Suppose $X=(X_1,\ldots,X_n)$ is the random variable with respect to the probability counting measure on $A$. 

\begin{lem}\label{lem:scales-with-no-room-for-improvement}
	In the above setting, $\pr_{I_{\rm l}}(A\cdot A\cdot A)=G_{I_{\rm l}}$.
\end{lem}
\begin{proof}
(See the beginning of Section 3.4 in \cite{Var}) Let's recall that $A$ is a $(D_0,\ldots,D_{n-1})$-regular set. Suppose $I$ is a subset $[0..n-1]$ such that for any $i\in I$, $D_i>|G_i|^{1-1/(3L)}$. By induction on $I$, we prove that $\pr_I(A\cdot A\cdot A)=G_I$. The base of induction follows from Theorem~\ref{thm:Gowers}. Suppose $I=\{i_1,\ldots,i_{m+1}\}$. By the induction hypothesis, 
\[
\pr_{\{i_1,\ldots,i_m\}}(A\cdot A\cdot A)=G_{I\setminus\{i_{m+1}\}}.
\]  
So for any $(g_{i_1},\ldots,g_{i_m})\in G_{I\setminus\{i_{m+1}\}}$, there are $a_1,a_2,a_3\in A$ such that 
\be\label{eq:first-m-components}
\pr_{I\setminus\{i_{m+1}\}}(a_1a_2a_3)=(g_{i_1},\ldots,g_{i_m}).
\ee 
Let 
\[
A(a_j):=\{a\in A|\h \pr_{[1..i_{m+1}-1]}(a)=\pr_{[1..i_{m+1}-1]}(a_j)\}.
\]
Then $|\pr_{i_{m+1}}(A(a_j))|=D_{i_{m+1}}>|G_i|^{1-1/(3L)}$; and so by Theorem~\ref{thm:Gowers}, we have
\be\label{eq:last-component}
\pr_{i_{m+1}}(A(a_1)A(a_2)A(a_3))=\pr_{i_{m+1}}(A(a_1))\pr_{i_{m+1}}(A(a_2))\pr_{i_{m+1}}(A(a_3))=G_{i_{m+1}}.
\ee
By \eqref{eq:first-m-components} and \eqref{eq:last-component}, we have that 
\[
\pr_I(A\cdot A\cdot A)=G_I;
\]
and claim follows.
\end{proof}
Let $I_{\rm s}:=[0..n-1]\setminus I_{\rm l}$, and for $g,g'\in G_{I_{\rm s}}$, let
\[
d(g,g'):=\sum_{i\in I_{\rm s}, \pr_i(g)\neq \pr_i(g')} \log |G_i|.
\]
It is easy to see that $d(.,.)$ defines a metric on $G_{I_{\rm s}}$. Let
\[
\textstyle
T:=\max \{d(g_{\rm s},1)|\h g_{\rm s}\in \bigcup_{i=1}^3 \pr_{I_{\rm s}}((\prod_{3i} S) \cap (\{1\}\oplus G_{I_{\rm s}}))\}; 
\]
here we rearranging components of $G_{[1..n]}$ and identifying it with $G_{I_{\rm l}}\oplus G_{I_{\rm s}}$. For any $g_{\rm l}\in G_{I_{\rm l}}$, let $\psi(g_{\rm l})\in G_{I_{\rm s}}$ be such that 
\[
(g_{\rm l},\psi(g_{\rm l}))\in A\cdot A\cdot A;
\]
notice that by Lemma~\ref{lem:scales-with-no-room-for-improvement} there is such a $\psi(g_{\rm l})$.
\begin{lem}\label{lem:approximate-hom}
In the above setting $\psi:G_{I_{\rm l}}\rightarrow G_{I_{\rm s}}$ is a $T$-approximate homomorphism; that means for any $g,g'\in G_{I_{\rm l}}$ we have
\[
d(\psi(gg'),\psi(g)\psi(g'))\le T\text{ and, } 
d(\psi(g^{-1}),\psi(g)^{-1})\le T.
\]	
\end{lem}
\begin{proof}
For $g,g'\in G_{I_{\rm l}}$, we have $(g,\psi(g)),(g',\psi(g')),(gg',\psi(gg'))\in A\cdot A\cdot A$; and so 
\[
\textstyle
(1,\psi(gg')\psi(g)^{-1}\psi(g')^{-1})\in \prod_9 S \cap (\{1\}\oplus G_{I_{\rm s}})
\text{ and, }
(1,\psi(g^{-1})\psi(g)) \in \prod_6 S \cap (\{1\}\oplus G_{I_{\rm s}});
\] 	
and claim follows as $d(g,g')$ is $G_{I_{\rm s}}$-bi-invariant. 
\end{proof}
By \cite[Theorem 2.1]{Far}, there is a group homomorphism $\wt{\psi}:G_{I_{\rm l}}\rightarrow G_{I_{\rm s}}$ such that for $g\in G_{I_{\rm l}}$
\be\label{eq:close-to-hom}
d(\psi(g),\wt{\psi}(g))\le 24 T.
\ee
\begin{lem}\label{lem:one-element-in-subgroup}
In the above setting, let $H$ be the graph of $\wt{\psi}$; then for any $g\in S$, there is $I_{\rm s}(g)\subseteq I_{\rm s}$ such that the following holds:
\begin{enumerate}
\item $g\in H G_{I_{\rm s}(g)}$ where $G_{I_{\rm s}(g)}$ is viewed as a subgroup of $G_{[1..n]}$. 
\item $|G_{I_{\rm s}(g)}| \le 2^{25 T}$.	
\end{enumerate}
\end{lem}
\begin{proof}
Suppose $g=(g_{\rm l},g_{\rm s})$ for some $g_{\rm l}\in G_{I_{\rm l}}$ and $g_{\rm s}\in G_{I_{\rm s}}$; then $d(\psi(g_{\rm l}),g_{\rm s})\le T$. By \eqref{eq:close-to-hom}, we have $d(\wt{\psi}(g_{\rm l}),\psi(g_{\rm l}))\le 24 T$; and so 
\be\label{eq:distance-from-H}
d(g_{\rm s},\wt{\psi}(g_{\rm l}))\le 25 T. 
\ee
Let $h:=(g_{\rm l},\wt{\psi}(g_{\rm l}))\in H$; and consider $h^{-1}g=(1,\wt{\psi}(g_{\rm l})^{-1} g_{\rm s})$. Let 
\[
I_{\rm s}(g):=\{j\in I_{\rm s}|\h \pr_j(g_{\rm s})\neq \pr_j(\wt\psi(g_{\rm l}))\};
\]
and so $h^{-1}g\in G_{I_{\rm s}(g)}$. By \eqref{eq:distance-from-H}, we have
\[
\sum_{j\in I_{\rm s}(g)} \log |G_j|\le 25 T,
\]
which implies that $|G_{I_{\rm s}(g)}|\le 2^{25T}$; and the claim follows.
\end{proof}
\begin{lem}\label{lem:finiding-small-height-element-with-large-centralizer}
	In the above setting, under the assumptions that $\delta\ll \vare^2\ll 1$ (as in \eqref{eq:initial-delta}) and an inequality of type \eqref{eq:required-inequality} hold, either $|\prod_3 S|>|G_{[1..n]}|^{1-\vare+\delta}$ or $T\gg \vare^2 \log |G_{[1..n]}|$; that means either $|\prod_3 S|>|G_{[1..n]}|^{1-\vare+\delta}$ or there is 
	\[
	\textstyle
	(1,g_{\rm s})\in \bigcup_{i=1}^3 (\prod_{3i} S) \cap (\{1\}\oplus G_{I_{\rm s}})
	\]
	such that $d(g_{\rm s},1)\gg \vare^2 \log |G_{[1..n]}|$.
\end{lem}
This can be interpreted as the existence of an element with {\em small height} and {\em large centralizer}; it has some conceptual similarities with \cite[Proposition 57]{SG:SAI}.  
\begin{proof}[Proof of Lemma~\ref{lem:finiding-small-height-element-with-large-centralizer}]
	By Lemma~\ref{lem:one-element-in-subgroup}, we have
	\[
		S\subseteq \bigcup_{I'\subseteq I_{\rm s}, |G_{I'}|\le 2^{25T}} H G_{I'}.
	\]
	Therefore we have
	\begin{align}
	\notag
	1=&\pcal_{S}(\bigcup_{I'\subseteq I_{\rm s}, |G_{I'}|\le 2^{25T}} H G_{I'})
	\le  
	\sum_{I'\subseteq I_{\rm s}, |G_{I'}|\le 2^{25T}} \pcal_S(H G_{I'})
	\\
	\notag
	\le &
	2^{|I_{\rm s}|} 2^{25T} [G_{[1..n]}:H]^{-\vare}|G_{[1..n]}|^{\delta}
=
	2^{|I_{\rm s}|} 2^{25T} |G_{I_{\rm s}}|^{-\vare} |G_{[1..n]}|^{\delta}
	\\
	\label{eq:lower-bound-union-subgroups}
	\le & 
	2^{25 T} |G_{I_{\rm s}}|^{-\vare/2} |G_{[1..n]}|^{\delta}. 
	&
	(2^{|I_{\rm s}|}\le |G_{I_{\rm s}}|^{\vare/2} 
	\text{by Inequality~\eqref{eq:required-inequality}})
	\end{align}
Let's assume that $|\prod_3 S|\le |G_{[1..n]}|^{1-\vare+\delta}$; then 
\[
\textstyle
|G_{I_{\rm l}}|\le |\prod_3 S|\le |G_{[1..n]}|^{1-\vare+\delta},
\]
which implies
\be\label{eq:scales-with-small-entropy-are-large}
|G_{[1..n]}|^{\vare/2}\le |G_{[1..n]}|^{\vare-\delta}\le |G_{I_{\rm s}}|.
\ee
By \eqref{eq:lower-bound-union-subgroups} and \eqref{eq:scales-with-small-entropy-are-large}, we have
\[
2^{25T}\ge |G_{[1..n]}|^{(\vare/2)^2-\delta} \ge |G_{[1..n]}|^{\vare^2/8};
\]
and the claim follows.
\end{proof}
\begin{lem}\label{lem:lower-bound-number-of-conjugates}
In the above setting, for any $g\in G_{[1..n]}$, we have
\[
|\{sgs^{-1}|\h s\in S\}|\ge |\Cl(g)|^{\vare}|G_{[1..n]}|^{-\delta},
\]
where $\Cl(g)$ is the set of conjugacy classes of $g$ in $G_{[1..n]}$.
\end{lem}
\begin{proof}
 Because of the bijection between conjugates of $g$ and cosets of the centralizer $C_{G_{[1..n]}}(g)$ of $g$ in $G_{[1..n]}$, we have that 
 \[
 |\{sgs^{-1}|\h s\in S\}|=|\underbrace{\{sC_{G_{[1..n]}}(g)|\h s\in S\}}_{\ccal(g;S)}|.
 \] 
 On the other hand,
 \[
 1=\pcal_S(\bigcup_{s\in S} sC_{G_{[1..n]}}(g))=\pcal_S(\bigcup_{\overline{s}\in \ccal(g;S)}\overline{s})
 \le \sum_{\overline{s}\in \ccal(g;S)} \pcal_S(\overline{s});
 \]
 and by our assumption 
 \[
 \pcal_S(\overline{s})\le [G_{[1..n]}:C_{G_{[1..n]}}(g)]^{-\vare}|G_{[1..n]}|^{\delta}=|\Cl(g)|^{-\vare} |G_{[1..n]}|^{\delta},
 \]
 for any $\overline{s}\in \ccal(g;S)$. Hence we have
 \[
 |\Cl(g)|^{\vare} |G_{[1..n]}|^{-\delta} \le |\ccal(g;S)|; 
 \]
 and the claim follows.
\end{proof}
\begin{prop}\label{proplargeindices} 
In the above setting, either $|\prod_3 S|>|G_{[1..n]}|^{1-\vare+\delta}$ or 
\[
\textstyle
|\prod_{14} S|\ge |G_{[1..n]}|^{\Theta_L(\vare^3)}|G_{I_{\rm l}}|.
\]	
\end{prop}
\begin{proof}
If $|\prod_3 S|\le|G_{[1..n]}|^{1-\vare+\delta}$	, then by Lemma~\ref{lem:finiding-small-height-element-with-large-centralizer} there is 
\[
	\textstyle
	(1,g_{\rm s})\in \bigcup_{i=1}^3 (\prod_{3i} S) \cap (\{1\}\oplus G_{I_{\rm s}})
\]
such that $d(g_{\rm s},1)\gg \vare^2 \log |G_{[1..n]}|$. Notice that 
\be\label{eq:conjugacy-class}
|\Cl(1,g_{\rm s})|=[G_{[1..n]}:C_{G_{[1..n]}}(1,g_{\rm s})]=\prod_{i\in I_{\rm s}}[G_i:C_{G_i}(\pr_i g_{\rm s})]\ge 2^{d(g_{\rm s},1)/L}\ge |G_{[1..n]}|^{\Theta_L(\vare^2)}.
\ee
Let us recall that there is a function $\psi:G_{I_{\rm l}}\rightarrow G_{I_{\rm s}}$ such that graph $H_{\psi}$ of $\psi$ is a subset of $\prod_3 S$. Since $\Cl(1,g_{\rm s})\subseteq \{1\}\oplus G_{\rm s}$, we have 
\be\label{eq:growth-after-conjugation}
\textstyle
|\Cl(1,g_{\rm s})H_{\psi}|=|\Cl(1,g_{\rm s})||G_{I_{\rm l}}|.
\ee
By \eqref{eq:conjugacy-class}, \eqref{eq:growth-after-conjugation}, and Lemma~\ref{lem:lower-bound-number-of-conjugates}, we have
\[
\textstyle
|\prod_{14} S|\ge |G_{[1..n]}|^{\Theta_L(\vare^3)}|G_{[1..n]}|^{-\delta} |G_{I_{\rm l}}| \ge |G_{[1..n]}|^{\Theta_L(\vare^3)}|G_{I_{\rm l}}|;
\]
and claim follows.
\end{proof}

\subsection{Combining Diophantine property of a distribution with entropy of another one}
The main goal of this section is to prove Proposition~\ref{prop:generic-Diophantine-property}. So this section is all about a single scale. Roughly speaking we start with two distributions on a finite group that satisfies (V1)$_{L}$-(V3)$_{L}$; we assume one of them has a certain Diophantine property and the other one has an entropy proportional to the entropy of the uniform distribution. We will show lots of certain convolutional distributions have both of these properties at the same time.  

In this section, $G$ is a finite group that satisfies (V1)$_{L}$-(V3)$_{L}$, and $m\le L$ and $m'\le L \log |G|$ are positive integers given in (V3)$_{L}$.

The next lemma says if we have a Diophantine type property for a distribution $\nu$ for subgroups of given complexity, then not many subgroups of the next level of complexity can fail a Diophantine type property of a similar order. 

We notice that because of (V3)$_{L}$-(v) the extra type $\{\cal_i'\}_{i=1}^{m'}$  of subgroups do not cause any problem and Varj\'{u}'s argument works in our setting as well.   

\begin{lem}\label{lem:exceptional-subgroups}
Suppose $G$ is a finite group that satisfies (V3)$_{L}$ and $m\le L$ is a positive integer given in (V3)$_{L}$. Suppose $\nu$ is a probability measure on $G$, $1\le k\le m$ is an integer, and $0<p,p'<1$ with the following properties. 
\begin{enumerate}
\item For any $H\in \bigcup_{i=1}^k \cal_i$ and for any $g\in G$, $\nu(gH)<p$. 
\item $p'>\sqrt{2Lp}$.	
\end{enumerate}
If $k<m$, let 
\[
E_{k+1}(\nu;p,p'):=\{H\in \cal_{k+1}|\h \wt{\nu}\ast \nu(H)>p'\}.
\] 
If $k=m$, for any $i\in [1..m']$, let 
\[
E'_i(\nu;p,p'):=\{H\in \cal'_i|\h \wt\nu\ast \nu(H)>p'\}.
\]
Then $|E_{k+1}(\nu;p,p')|$ and $|E'_i(\nu;p,p')|$ are less than $\sqrt{\frac{2}{Lpp'}}$.
\end{lem}
\begin{proof} (See \cite[Towards the end of proof of Lemma 18]{Var})
First we consider the case $k<m$. For two distinct elements $H,H'\in E_{k+1}(\nu;p,p')$, there is $H^\sharp\in \cal_j$ for some $j\le k$ such that $[H\cap H':H^\sharp\cap H\cap H']\le L$. Hence $\nu(g(H\cap H'))\le Lp$, which implies
\be\label{eq:volume-intersection}
\wt\nu\ast \nu(H\cap H')\le Lp.
\ee
For any $1\le l\le |E_{k+1}(\nu;p,p')|$, suppose $H_1,\ldots,H_l$ are distinct elements of $E_{k+1}(\nu;p,p')$; then 
\be\label{eq:inclusion-exclusion}
1\ge \wt\nu\ast \nu(\bigcup_{i=1}^l H_i)
\ge \sum_{i=1}^l \wt\nu\ast \nu(H_i)- \sum_{1\le i<j\le l} \wt\nu\ast \nu(H_i\cap H_j) \ge lp'- {l\choose 2} Lp.
\ee
Let $f(x):=-\frac{Lp}{2}x^2+(p'+\frac{Lp}{2}) x-1$; then by \eqref{eq:inclusion-exclusion} for any $l\in [1..|E_{k+1}(\nu;p,p')|]$, $f(l)\le 0$. We notice that
\begin{align*}
	f\left(\frac{2}{p'+Lp/2}\right)=
	&
	-\frac{Lp}{2}\left(\frac{2}{p'+Lp/2}\right)^2+(p'+Lp/2) \left(\frac{2}{p'+Lp/2}\right)-1
\\
	=&
	-\frac{Lp}{2}\left(\frac{2}{p'+Lp/2}\right)^2+1.
\end{align*}
Since $p'>\sqrt{2Lp}$, we have $p'+Lp/2> \sqrt{2Lp}$, which implies
$\frac{p'+Lp/2}{2}> \sqrt{\frac{Lp}{2}}$. Hence 
\[
f\left(\frac{2}{p'+Lp/2}\right)>0.
\]
By the concavity of $f$, $f(-\infty)=-\infty$, $1\le 2/(p'+Lp/2)$, and the above discussion we deduce that
\[
|E_{k+1}(\nu;p,p')|< \frac{2}{p'+Lp/2}\le \sqrt{\frac{2}{Lpp'}};
\]
and claim follows in this case. 

For the case of $k=m$, we notice that for any two distinct elements $H,H'\in \cal'_i$, there is $H^\sharp\in \cal_j$ for some $j\le m$ such that $[H\cap H':H^\sharp\cap H\cap H']\le L$. So an identical argument as in the previous case works here as well.
\end{proof}

Let us first recall the setting of Proposition~\ref{prop:generic-Diophantine-property}; we will be working in this setting for the rest of this section.
\begin{enumerate}
\item $G$ is a finite group that satisfies (V2)$_L$, (V3)$_{L}$ and $|G|\gg_{\alpha',\beta,L} 1$ (the implied constant will be specified later);
\item $X_1,\ldots,X_{2^{m+1}}$ are independent random variables such that $H_{\infty}(X_i)\ge \alpha' \log |G|$; 
\item  for any $l$-tuple $\vecy:=(y_1,\ldots,y_l)$, $X_{\vecy}:=X_1y_1X_2 \cdots y_{l}X_{l+1}$; in addition we let \[X'_{\vecy}:=X_{l+2}y_1X_{l+3}\cdots y_{l}X_{2l+2};\] 
\item $Y_1,\ldots,Y_{2^{m+1}-1}$ are i.i.d. random variables with values in $G$ such that for any $H$ in $\bigcup_{i=1}^m \cal_i$ and $g\in G$, $\bbp(Y_1\in gH)\le [G:H]^{-\beta}$; and moreover for any subgroup $H$ with order at least $|G|^{\alpha}$ and any $g\in G$, we have the same inequality; that means $\bbp(Y_1\in gH)\le [G:H]^{-\beta}$.
\item $\beta\ge 4\alpha$; in fact it is enough to assume $(1-\frac{1}{L})\beta\ge \alpha$.
\end{enumerate}
\begin{lem}\label{lem:inductive-set-up-structural}
In the above setting, let $\beta_0:=\min(\frac{\beta}{5L},\frac{\alpha'}{2})$, $\beta_k:=\frac{\beta_0}{8^k}$ and $p_k:=(2^k-1)|G|^{-\beta/(2L)}$; then for any $1\le k\le m$ 
\[
\textstyle
\bbp((Y_1,\ldots,Y_{2^k-1})=\vecy 
\text{ such that } 
\exists H\in \bigcup_{j=0}^k \cal_j, \exists g\in G, \bbp(X_{\vecy}\in gH)\ge |G|^{-\beta_k})\le p_k.
\]	
\end{lem}
\begin{proof}
(See \cite[Proof of Lemma 18]{Var}) We proceed by induction on $k$. In order to deal with the base of induction in the same venue as in the induction step, we consider the case of $k=0$ as well; in the sense that we show why we have  
$\bbp(X_1\in gZ(G))<|G|^{-\beta_0}$ for any $g\in G$.

Since $H_{\infty}(X)\ge \alpha'\log|G|$ and $|Z(G)|\le L$, we have 
$\bbp(X_1\in gZ(G))<L |G|^{-\alpha'}\le |G|^{-\alpha'/2}$ (the second inequality holds as $|G|\gg_{\alpha',L} 1$); and this implies the case of $k=0$. 

Next we focus on the induction step; let
\be\label{eq:def-exceptional-left}
\mathcal{E}_k:=\{\vecy \in \bigoplus_{i=1}^{2^k-1} G|\h \exists H\in \bigcup_{i=0}^k \cal_i, \exists g\in G, \bbp(X_{\vecy} \in gH)\ge |G|^{-\beta_k}\}, 
\ee
and 
\be\label{eq:def-exceptional-right}
\mathcal{E}'_k:=\{\vecy \in \bigoplus_{i=1}^{2^k-1} G|\h \exists H\in \bigcup_{i=0}^k \cal_i, \exists g\in G, \bbp(X'_{\vecy} \in gH)\ge |G|^{-\beta_k}\}.
\ee
By the induction hypothesis, we have that these are {\em exceptional sets}:
\be\label{eq:induction-hypo-exceptional-sets}
\bbp((Y_1,\ldots,Y_{2^k-1})\in \mathcal{E}_k)\le p_k
\text{ and }
\bbp((Y_{2^k+1},\ldots,Y_{2^{k+1}-1})\in \mathcal{E}'_k)\le p_k.
\ee
Suppose $\vecy:=(\vecy_{\rm l},y,\vecy_{\rm r})\in \mathcal{E}_{k+1}$ where $\vecy_{\rm l}$ and $\vecy_{\rm r}$ are the left and the right $2^k-1$ components of $\vecy$, respectively, and $y\in G$. Then 
$X_{\vecy}=X_{\vecy_{\rm l}}yX'_{\vecy_{\rm r}}$ and there are $H\in \bigcup_{i=0}^{k+1}\cal_i$ and $g\in G$ such that
\be\label{eq:being-exceptional}
|G|^{-\beta_{k+1}}\le \bbp(X_{\vecy_{\rm l}}yX'_{\vecy_{\rm r}}\in gH)
=\sum_{j=1}^{[G:H]}\bbp(X_{\vecy_{\rm l}}\in gHg_j)\bbp(X'_{\vecy_{\rm r}}\in y^{-1}g_j^{-1}H),
\ee
where $\{g_j\}_{j=1}^{[G:H]}$ is a set of right coset representatives of $H$. Let 
\[
I_{\rm l}:=\{j\in [1..[G:H]]| \bbp(X_{\vecy_{\rm l}}\in gHg_j)\le \bbp(X'_{\vecy_{\rm r}}\in y^{-1}g_j^{-1}H)\}
\]
 and 
 \[
 I_{\rm r}:=\{j\in [1..[G:H]]| \bbp(X_{\vecy_{\rm l}}\in gHg_j)> \bbp(X'_{\vecy_{\rm r}}\in y^{-1}g_j^{-1}H)\}.
 \]
  Let $q_{\rm l}:=\max_{j\in I_{\rm l}} \bbp(X_{\vecy_{\rm l}}\in gHg_j)$ and $q_{\rm r}:=\max_{j\in I_{\rm r}} \bbp(X'_{\vecy_{\rm r}}\in y^{-1}g_j^{-1}H)\}$; then
\begin{align*}
\sum_{j\in I_{\rm l}}\bbp(X_{\vecy_{\rm l}}\in gHg_j)\bbp(X'_{\vecy_{\rm r}}\in y^{-1}g_j^{-1}H)& \le q_{\rm l}, \text{ and}
\\
\sum_{j\in I_{\rm r}}\bbp(X_{\vecy_{\rm l}}\in gHg_j)\bbp(X'_{\vecy_{\rm r}}\in y^{-1}g_j^{-1}H)& \le q_{\rm r}.
\end{align*}
Therefore by \eqref{eq:being-exceptional}, we have
\[
\frac{1}{2}|G|^{-\beta_{k+1}}\le \max(q_{\rm l},q_{\rm r}),
\]
which implies that there is $j_0$ such that 
\be\label{eq:cosets-with-large-prob}
\frac{1}{2}|G|^{-\beta_{k+1}}\le \bbp(X_{\vecy_{\rm l}}\in gHg_{j_0})
\text{ and }
\frac{1}{2}|G|^{-\beta_{k+1}} \le \bbp(X'_{\vecy_{\rm r}}\in y^{-1}g_{j_0}^{-1}H).
\ee
For a random variable $U$ with values in $G$, let $\wt{U}$ be a random variable independent of $U$ with a distribution similar to $U^{-1}$. Then by \eqref{eq:cosets-with-large-prob} we have
\be\label{eq:subgroups-large-prob}
\bbp(\wtx_{\vecy_{\rm l}}X_{\vecy_{\rm l}}\in g_j^{-1}Hg_j)\ge \frac{1}{4}|G|^{-2\beta_{k+1}}
\text{ and }
\bbp(X'_{\vecy_{\rm r}}\wtx'_{\vecy_{\rm r}}\in y^{-1}g_j^{-1}Hg_jy)\ge \frac{1}{4}|G|^{-2\beta_{k+1}}.
\ee
It $\vecy_{\rm l}\not\in \mathcal{E}_k$, then $\bbp(X_{\vecy_{\rm l}}\in \overline{g}\overline{H})<|G|^{-\beta_k}$ for any $\overline{H}\in \bigcup_{i=0}^k \cal_i$ and any $\overline{g}\in G$. This implies that
\be\label{eq:exceptional-subgroup-left}
g_j^{-1}Hg_j \in E_{k+1}(\lambda_{\vecy_{\rm l}}; |G|^{-\beta_k}, \frac{1}{4}|G|^{-2\beta_{k+1}}),
\ee
where $\lambda_{\vecy_{\rm l}}$ is the distribution of the random variable $X_{\vecy_{\rm l}}$ and $E_{k+1}$ is the set defined in Lemma~\ref{lem:exceptional-subgroups}. By a similar argument, if $\vecy_{\rm r}\not\in \mathcal{E}'_k$, then 
\be\label{eq:exceptional-subgroup-right}
y^{-1}g_j^{-1}Hg_jy \in E_{k+1}(\wt{\lambda'}_{\vecy_{\rm r}}; |G|^{-\beta_k}, \frac{1}{4}|G|^{-2\beta_{k+1}}),
\ee
where $\wt{\lambda'}_{\vecy_{\rm r}}$ is the distribution of the random variable $\wtx'_{\vecy_{\rm r}}$. So far by \eqref{eq:exceptional-subgroup-left}, \eqref{eq:exceptional-subgroup-right}, 
 and we have
\begin{align}
\notag
\bbp((\vecy_{\rm l},y,\vecy_{\rm r})\in \mathcal{E}_{k+1})\le
&
 \bbp(\vecy_{\rm l}\in \mathcal{E}_k)+\bbp(\vecy_{\rm r}\in \mathcal{E}'_k)+
	\bbp((\vecy_{\rm l},y,\vecy_{\rm r})\in \mathcal{E}_{k+1}), \vecy_{\rm l}\not\in \mathcal{E}_k, \vecy_{\rm r}\not\in \mathcal{E}'_{k})
\\
\notag
\le &
2p_k+\sum_{\vecy_{\rm l}\not\in \mathcal{E}_k, \vecy_{\rm r}\not\in\mathcal{E}'_k} \bbp((Y_1,\ldots,Y_{2^k-1})=\vecy_{\rm l})\bbp((Y_{2^k+1},\ldots,Y_{2^{k+1}-1})=\vecy_{\rm r})
\\
\label{eq:initial-upper-bound-prob}
&	\left(\sum_{H_1,H_2}\bbp(Y_{2^k}^{-1}H_1Y_{2^k}=H_2)\right)
\end{align}
where $H_1$ ranges in $E_{k+1}(\lambda_{\vecy_{\rm l}};|G|^{-\beta_k},\frac{1}{4}|G|^{-2\beta_{k+1}})$ and $H_2$ ranges in 
		$E_{k+1}(\wt{\lambda'}_{\vecy_{\rm r}}; |G|^{-\beta_k}, \frac{1}{4}|G|^{-2\beta_{k+1}})$. For a given $H_1$ and $H_2$ that are conjugate of each other and are in $\cal_i$ for some $i$ there is $g'\in G$ such that
\[
\bbp(Y_{2^k}^{-1} H_1 Y_{2^k}=H_2)=\bbp(Y_{2^k}\in g'N_G(H_1));
\]
and by our assumption there is $H^\sharp\in \cal_j$ for some $j$ such that $N_G(H_1)\preceq_L H^\sharp$. Hence
\be\label{eq:conjugation-prob}
\bbp(Y_{2^k}^{-1} H_1 Y_{2^k}=H_2)\le L [G:H^\sharp]^{-\beta}\le L|G|^{-\beta/L}.
\ee
By \eqref{eq:initial-upper-bound-prob}, \eqref{eq:conjugation-prob}, and Lemma~\ref{lem:exceptional-subgroups}, we have
\begin{align}
\notag
\bbp((\vecy_{\rm l},y,\vecy_{\rm r})\in \mathcal{E}_{k+1})\le
&
2p_k+\sum_{\vecy_{\rm l}\not\in \mathcal{E}_k, \vecy_{\rm r}\not\in\mathcal{E}'_k} \bbp((Y_1,\ldots,Y_{2^k-1})=\vecy_{\rm l})\bbp((Y_{2^k+1},\ldots,Y_{2^{k+1}-1})=\vecy_{\rm r})
\\
\notag
& \left( L|G|^{-\beta/L}\frac{2}{L|G|^{-\beta_k}|G|^{-2\beta_{k+1}}/4}\right)
\\
\label{eq:second-upper-bound-prob}
\le 
&
2p_k+8|G|^{-\frac{\beta}{L}+\beta_k+2\beta_{k+1}}.
\end{align}
By \eqref{eq:second-upper-bound-prob}, to prove the claim it is enough to show that $2p_k+8|G|^{-\frac{\beta}{L}+\beta_k+2\beta_{k+1}}\le p_{k+1}$. We notice that $p_{k+1}-2p_k=|G|^{-\beta/(2L)}$, and 
\[
-\frac{\beta}{L}+\beta_k+2\beta_{k+1}=-\frac{\beta}{L}+\frac{1}{8^k}\left(1+\frac{1}{4}\right)\beta_0\le -\frac{3\beta}{4L}.
\]
Hence it is enough to show $8|G|^{-\frac{3\beta}{4L}}\le |G|^{-\frac{\beta}{2L}}$, which clearly holds for $|G|\gg_{\beta,L} 1$.
\end{proof}
\begin{lem}\label{lem:Diophantine-subfield-type}
	In the above setting, let $p:=(2^{m+1}-1)|G|^{-\beta/(2L)}$, $\beta_0:=\min(\frac{\beta}{5L},\frac{\alpha'}{2})$, and $\beta':=\frac{\beta_0}{8^{m+1}}$; then 
\[
\textstyle
\bbp((Y_1,\ldots,Y_{2^{m+1}-1})=\vecy 
\text{ s.t. } 
\exists H\in \bigcup_{j=0}^{m'} \cal'_j, \exists g\in G, \bbp(X_{\vecy}\in gH)\ge |G|^{-\beta'})\le p.
\]	
\end{lem}
\begin{proof}
We follow an identical argument as in the proof of Lemma~\ref{lem:inductive-set-up-structural}. Let 
\be\label{eq:exceptional-subfield-type}
\mathcal{E}':=
\{\vecy\in \bigoplus_{i=1}^{2^{m+1}-1} G|\h \exists H\in \bigcup_{i=1}^{m'} \cal'_i, \exists g\in G, \bbp(X_{\vecy}\in gH)\ge |G|^{-\beta'}\}.
\ee
Suppose $\vecy:=(\vecy_{\rm l},y,\vecy_{\rm r})\in\mathcal{E}'$ where $\vecy_{\rm l}$ and $\vecy_{\rm r}$ are the left and the right $2^m-1$ components of $\vecy$, respectively, and $y\in G$. Then $X_{\vecy}=X_{\vecy_{\rm l}}yX'_{\vecy_{\rm r}}$, and there are $1\le i\le m'$, $H\in\cal'_i$, and $g\in G$ such that $|G|^{-\beta'}\le \bbp(X_{\vecy_{\rm l}}yX'_{\vecy_{\rm r}}\in gH)$. As in the proof of Lemma~\ref{lem:inductive-set-up-structural}, there is $g'\in G$ such that
\be\label{eq:subgroups-large-prob-2}
\bbp(\wtx_{\vecy_{\rm l}}X_{\vecy_{\rm l}}\in g'^{-1}Hg')\ge \frac{1}{4}|G|^{-2\beta'}
\text{ and }
\bbp(X'_{\vecy_{\rm r}}\wtx'_{\vecy_{\rm r}}\in y^{-1}g'^{-1}Hg'y)\ge \frac{1}{4}|G|^{-2\beta'}.
\ee 
If $\vecy_{\rm l}\not\in \mathcal{E}_m$ where $\mathcal{E}_m$ is defined in \eqref{eq:def-exceptional-left}, then by Lemma~\ref{lem:inductive-set-up-structural} for any $\overline{H}\in \bigcup_{i=0}^m \cal_i$ and any $\overline{g}\in G$ we have $\bbp(X_{\vecy}\in \overline{g}\overline{H})<|G|^{-\beta_m}$ where $\beta_m=\frac{\beta_0}{8^m}$. This implies that
\be\label{eq:among-exceptional-subgroups}
g'^{-1}Hg'\in E'_i(\lambda_{\vecy_{\rm l}};|G|^{-\beta_m},\frac{1}{4}|G|^{-2\beta'}),
\ee 
where $E'_i$ is the set defined in Lemma~\ref{lem:exceptional-subgroups}. Similarly if $\vecy_{\rm r}\not\in \mathcal{E}'_m$ where $\mathcal{E}'_m$ is defined in \eqref{eq:def-exceptional-right}, then 
\be\label{eq:among-exceptional-subgroups-right}
y^{-1}g'^{-1}Hg'y\in E'_i(\wt{\lambda'}_{\vecy_{\rm r}};|G|^{-\beta_m},\frac{1}{4}|G|^{-2\beta'}).
\ee
Following an identical argument as in the proof of Lemma~\ref{lem:inductive-set-up-structural}, we get that
\begin{align}
\notag
\bbp((\vecy_{\rm l},y,\vecy_{\rm r})\in \mathcal{E}')\le
&
 \bbp(\vecy_{\rm l}\in \mathcal{E}_m)+\bbp(\vecy_{\rm r}\in \mathcal{E}'_m)+
	\bbp((\vecy_{\rm l},y,\vecy_{\rm r})\in \mathcal{E}'), \vecy_{\rm l}\not\in \mathcal{E}_m, \vecy_{\rm r}\not\in \mathcal{E}'_m)
\\
\notag
\le &
2p_m+\sum_{\vecy_{\rm l}\not\in \mathcal{E}_m, \vecy_{\rm r}\not\in\mathcal{E}'_m} \bbp((Y_1,\ldots,Y_{2^m-1})=\vecy_{\rm l})\bbp((Y_{2^m+1},\ldots,Y_{2^{m+1}-1})=\vecy_{\rm r})
\\
\notag
&	\left(\sum_{i=1}^{m'}\sum_{H_1,H_2}\bbp(Y_{2^m}^{-1}H_1Y_{2^m}=H_2)\right)
\end{align}
where $p_m$ is given in Lemma~\ref{lem:inductive-set-up-structural}, $H_1$ ranges in $E'_i(\lambda_{\vecy_{\rm l}};|G|^{-\beta_m},\frac{1}{4}|G|^{-2\beta'})$, and $H_2$ ranges in 
		$E_{i}'(\wt{\lambda'}_{\vecy_{\rm r}}; |G|^{-\beta_m}, \frac{1}{4}|G|^{-2\beta'})$ for the given $i$. We notice that for a given $H_1$ and $H_2$ that are conjugate of each other and are in $\cal'_i$, there is $g''\in G$ such that
\[
\bbp(Y_{2^m}^{-1} H_1 Y_{2^m}=H_2)=\bbp(Y_{2^m}\in g''N_G(H_1));
\]
and by our assumption $[N_G(H_1):H_1]\le L$. Hence 
\[
\bbp(Y_{2^m}^{-1} H_1 Y_{2^m}=H_2)\le L\max_{\overline{g}\in G} \bbp(Y_{2^m}\in \overline{g}H_1).
\]
Now we consider two cases based on whether $|H_1|\ge |G|^{\alpha}$ or not. 

{\bf Case 1.} $|H_1|\ge |G|^{\alpha}$. 

In this case, by our assumption, 
\[
\max_{\overline{g}\in G} \bbp(Y_{2^m}\in \overline{g}H_1)\le [G:H]^{-\beta}\le |G|^{-
\frac{\beta}{L}};
\]
and so by an identical analysis as in the proof of Lemma~\ref{lem:inductive-set-up-structural} and our assumption that $m'\le \log |G|$, we deduce that
\be\label{eq:initial-upper-bound-subfield}
\bbp((\vecy_{\rm l},y,\vecy_{\rm r})\in \mathcal{E}')
\le 2p_m+8L(\log |G|) |G|^{-\frac{\beta}{L}+\beta_m+2\beta'}. 
\ee
By \eqref{eq:initial-upper-bound-subfield} to prove the claim in this case, it is enough to show $
2p_m+8L(\log |G|) |G|^{-\frac{\beta}{L}+\beta_m+2\beta'}\le p.
$ We notice that $p-2p_m=|G|^{-\beta/(2L)}$, and 
$
-\frac{\beta}{L}+\beta_m+2\beta'\le -\frac{3\beta}{4L}.
$
Hence it is enough to show $8L(\log|G|) |G|^{-\frac{3\beta}{4L}}\le |G|^{-\frac{\beta}{2L}}$, which clearly holds for $|G|\gg_{\beta,L} 1$.

{\bf Case 2.} $|H_1|<|G|^{\alpha}$.

In this case, we have 
\[
\max_{\overline{g}\in G} \bbp(Y_{2^m}\in \overline{g}H_1)\le |G|^{-\beta}|H| \le|G|^{-\beta}|G|^{\alpha}\le |G|^{-\frac{\beta}{L}},
\]
where the last inequality holds as $(1-\frac{1}{L})\beta\ge \alpha$. Now we can follow the same analysis as in the first case; and the claim follows.
\end{proof}
\begin{proof}[Proof of Propodition~\ref{prop:generic-Diophantine-property}]
First we notice that we can and will let $\cal_{m+1}:=\cal_m$ and get the claim of Lemma~\ref{lem:inductive-set-up-structural} for $k=m+1$ as well. Hence by Lemma~\ref{lem:inductive-set-up-structural} and Lemma~\ref{lem:Diophantine-subfield-type} we get
\[
\bbp(\vecy\in \mathcal{E}_{m+1}\cup \mathcal{E}')\le p_{m+1}+p'=2(2^{m+1}-1)|G|^{-\frac{\beta}{2L}}\le |G|^{-\frac{\beta}{4L}},
\] 	
where $\mathcal{E}_{m+1}$ and $\mathcal{E}'$ are defined in \eqref{eq:exceptional-subgroup-left} and \eqref{eq:exceptional-subfield-type}, respectively, and the last inequality holds for $|G|\gg_{L,\beta} 1$. 

Suppose $\vecy\not\in \mathcal{E}_{m+1}\cup \mathcal{E}'$. For any proper subgroup $H$ of $G$ proper, there is $H^\sharp\in \bigcup_{i=0}^m\cal_i \cup \bigcup_{j=1}^{m'} \cal'_j$ such that $[H:H\cap H^{\sharp}]\le L$. Then for any $g\in G$
\[
\bbp(X_{\vecy}\in gH)\le L \max_{g'\in G} \bbp(X_{\vecy}\in g'H^\sharp)\le L|G|^{-\beta'}\le |G|^{-\beta'/2},
\] 
where the last inequality holds for $|G|\gg_{\beta,L} 1$.
\end{proof}

\subsection{Gaining conditional entropy: proof of Proposition~\ref{prop:VarjuMutation}} 
By the definition of the conditional R\'{e}nyi entropy, we have
 \begin{align}
 \notag
 H_2(X_1Y_1\cdots Y_{2^{m+1}-1}X_{2^{m+1}}X_{2^{m+1}+1}|Y_1,\ldots,Y_{2^{m+1}-1}) &
 = \hspace{3cm} 
 \\
 \label{eq:def-conditional-entropy}
 \sum_{\vecy\in \bigoplus_{i=1}^{2^{m+1}-1}G}  \bbp ((Y_1,\ldots,Y_{2^{m+1}-1})&=\vecy) 
H_2(X_{\vecy}X_{2^{m+1}+1}).
 \end{align}
Let 
\[
\mathcal{E}'':=\{\vecy \text{ such that } X_{\vecy} \text{ is not of } (0,\beta'/2)\text{-Diophantine type}\}
\]
where $\beta':=\frac{1}{8^{m+1}}\min(\frac{\beta}{5L},\frac{\alpha'}{2})$; and so by Proposition~\ref{prop:generic-Diophantine-property} we have
\be\label{eq:measure-exceptional-set}
\bbp((Y_1,\ldots,Y_{2^{m+1}-1})\in \mathcal{E}'')\le |G|^{-\frac{\beta}{4L}}.
\ee
For $\vecy\in \mathcal{E}''$, we use the trivial bound given in Lemma~\ref{lem:product-entropy-trivial-bound}
\be\label{eq:trivial-bound-for-exceptional-set}
H_2(X_{\vecy}X_{2^{m+1}+1})\ge \max_{i=1}^{2^{m+1}+1}H_2(X_i)\ge \min_{i=1}^{2^{m+1}+1} H_2(X_i)=:h_{\rm min};
\ee
we notice that $H_2(Xg)=H_2(X)$ for any random variable $X$ with values in $G$ and $g\in G$. For $\vecy\not\in \mathcal{E}''$, we have that 
\begin{enumerate}
\item $X_{\vecy}$ is of $(0,\beta'/2)$-Diophantine type.
\item $H_2(X_{\vecy})\ge \max_{i=1}^{2^{m+1}} H_2(X_i)\ge \max_{i=1}^{2^{m+1}} H_\infty(X_i)\ge \alpha' \log |G|$ by Lemma~\ref{lem:properties-entropy}, Lemma~\ref{lem:product-entropy-trivial-bound}, and the fact that $H_2(X_iy_i)=H_2(X_i)$ for any $i$. 
\end{enumerate}
Then either $H_2(X_{\vecy})>(1-\frac{\alpha''}{2})\log |G|$ or by Lemma~\ref{lem:GrowthOfRenyiEntropySingleScale}
\[
H_2(X_{\vecy}X_{2^{m+1}+1})\ge \frac{H_2(X_{\vecy})+H_2(X_{2^{m+1}+1})}{2}+\gamma_0 \log |G|
\]
for some positive $\gamma_0$ which depends only on $\alpha',\alpha'',\beta,$ and the function $\delta_0$. Since \[\max_{i=1}^{2^{m+1}+1} H_2(X_i)\le (1-\alpha'')\log |G|,\] in either case we get 
\be\label{eq:gaining-entropy-outside-exceptional-set}
H_2(X_{\vecy}X_{2^{m+1}+1})\ge h_{\rm min}+\gamma_0 \log |G|.
\ee
In what follows, let $p_{\vecy}:=\bbp((Y_1,\ldots,Y_{2^{m+1}-1})=\vecy)$ for simplicity. By \eqref{eq:def-conditional-entropy}, \eqref{eq:trivial-bound-for-exceptional-set}, and \eqref{eq:gaining-entropy-outside-exceptional-set}, we have
\begin{align*}
H_2(X_1Y_1\cdots Y_{2^{m+1}-1}&X_{2^{m+1}}X_{2^{m+1}+1}|Y_1,\ldots,Y_{2^{m+1}-1}) 
  \hspace{3cm} 
 \\
= & 
\sum_{\vecy \in \mathcal{E}''} p_{\vecy} h_{{\rm min}} +
\sum_{\vecy\not\in \mathcal{E}''} p_{\vecy} (h_{{\rm min}}+\gamma_0 \log |G|)
\\
= & h_{{\rm min}}+\gamma_0 \bbp((Y_1,\ldots,Y_{2^{m+1}-1})\not\in \mathcal{E}'')\log |G|
\\
\ge & 
h_{{\rm min}}+ \gamma_0 (1-|G|^{-\frac{\beta}{4L}})\log |G| &(\text{By } \eqref{eq:measure-exceptional-set})
\\
\ge &h_{{\rm min}}+\frac{\gamma_0}{2} \log |G|,
\end{align*}
for $|G|\gg_{L,\beta,\alpha',\alpha'',\delta_0} 1$; and the claim follows.
\subsection{Scales with room for improvement} In this section, we use the gain of conditional entropy (given in Proposition~\ref{prop:VarjuMutation}) at levels where we have room for improvement to prove a growth statement in the multi-scaled setting of Proposition~\ref{propvarjuproduct}. In order to use Proposition~\ref{prop:VarjuMutation}, we need to have an auxiliary random variable with some Diophantine property. We recall a result of Varj\'{u} that provides us with such a random variable.

\begin{lem}\label{lem:Varj-aux-random-variable}\cite[Lemma 17]{Var}
Suppose $\{G_i\}_i$ is a sequence of finite groups that are $L^{-1}$-quasi-random. Suppose $0<\vare<1$, $0<\delta<\vare/(8L)$, and $S\subseteq G:=\bigoplus_{i=1}^n G_i$ is such that for any proper subgroup $H$ of $G$ and $g\in G$ we have
\[
\pcal_S(gH)\le [G:H]^{-\vare}|G|^{\delta}.
\] 
Then there are $B\subseteq S$ and $J_{\rm g}\subseteq [1..n]$ (think about it as the set of {\em good} indexes) with the following properties. Let $Y:=(Y^{(1)},\ldots,Y^{(n)})$ be the random variable with the uniform distribution on $B$, and $Y^{(i)}$ be the induced random variable with values in $G_i$.
\begin{enumerate}
\item For $i\in J_{\rm g}$, $Y^{(i)}$ is of $(0,\frac{\vare}{2L})$-Diophantine type.
\item Let $J_{\rm b}:=[1..n]\setminus J_{\rm g}$ (think about it as the set of {\em bad} indexes); then $|G_{J_{\rm b}}|\le |G|^{\frac{\delta}{\vare/(2L)}}$, where $G_{J_{\rm b}}:=\bigoplus_{i\in J_{\rm b}} G_i$.
\end{enumerate}
\end{lem}
\begin{proof}
See \cite[Proof of Lemma 17]{Var}.	
\end{proof}
Let us recall some of our assumptions and earlier results that will be used in the remaining of this section. 
\begin{enumerate}
\item $\{G_i\}_{i=1}^{\infty}$ is a family of pairwise non-isomorphic finite groups that satisfy assumptions (V1)$_L$-(V3)$_L$ and (V4)$_{\delta_0}$. 
\item $0<\vare<1$ and $\delta:=\vare^5/(8L)$.
\item Suppose $|G_i|\gg_{\vare,L} 1$ such that Proposition~\ref{prop:VarjuMutation} holds for the variables $\alpha':=\delta$, $\alpha'':=1/(3L)$, $\beta:=\vare/(2L)$, $L$, and $\delta_0$; we are allowed to make this assumption thanks to Lemma~\ref{lem:required-inequality}. Let $\gamma$ be the constant given by Proposition~\ref{prop:VarjuMutation} for the same set of variables. 
\item $S\subseteq G:=\bigoplus_{i=1}^n G_i$ such that for any proper subgroup $H$ of $G$ and $g\in G$
\[
\pcal_S(gH)\le [G:H]^{-\vare}|G|^{\delta}.
\]
\item Let $A\subseteq S$ be the $(D_0,\ldots,D_{n-1})$-regular subset that is given at the end of Section~\ref{ss:regular-subset}; that means $D_i$ is either 1 or at least $|G_i|^{\delta}$ and $|A|>|G|^{-2\delta} |S|$.
\item Let $I_{\rm l}:=\{i\in [0..n-1]| D_i>|G_i|^{1-1/(3L)}\}$ and $I_{\rm s}:=[0..n-1]\setminus I_{\rm l}$.
\item Let $B\subseteq S$ be given by Lemma~\ref{lem:Varj-aux-random-variable}; and let $J_{\rm g}$ and $J_{\rm b}$ be given the sets given in same lemma.
\end{enumerate}
In the above setting, we let $X_i:=(X^{(1)}_i,\ldots,X^{(n)}_i)$ be i.i.d. random variables with distribution $\pcal_A$ for $1\le i\le 2^{m+1}+1$, and $Y_i:=(Y^{(1)}_i,\ldots,Y^{(n)}_i)$ be i.i.d. random variables with distribution $\pcal_B$ for $1\le i\le 2^{m+1}-1$. For $\vecy=(\vecy^{(1)},\ldots,\vecy^{(n)})\in \bigoplus_{i=1}^{2^{m+1}-1} G$, we let
	\[
	p_{\vecy}:=\bbp((Y_1,\ldots,Y_{2^{m+1}-1})=\vecy),
	\text{ and }
	p_{\vecy^{(i)}}:=\bbp((Y^{(i)}_1,\ldots,Y^{(i)}_{2^{m+1}-1})=\vecy^{(i)})
	\]
	and 
	\begin{align*}
	X_{\vecy}:= & X_1y_1X_2\ldots y_{2^{m+1}-1}X_{2^{m+1}} \\
	= &(X^{(1)}_1y^{(1)}_1X^{(1)}_2\cdots y^{(1)}_{2^{m+1}-1}X^{(1)}_{2^{m+1}},
	\ldots,
	X^{(n)}_1y^{(n)}_1X^{(n)}_2\cdots y^{(n)}_{2^{m+1}-1}X^{(n)}_{2^{m+1}}).
	\end{align*}
\begin{lem}\label{lem:main-gain-scales-with-room}
	In the above setting, we have
	\[
\sum_{\vecy\in \bigoplus_{j=1}^{2^{m+1}-1} G} p_{\vecy} H(X_{\vecy} X_{2^{m+1}+1})\ge 
\log |S|+\gamma \log |G_{I_{\rm s}}|-\gamma \frac{\delta}{\vare/(3L)} \log |G|.
	\]
\end{lem}
\begin{proof}
By Lemma~\ref{lem:properties-entropy}, we have
\begin{align}
\notag
H(X_{\vecy}X_{2^{m+1}+1})= 
&\sum_{i=1}^n H(X^{(i)}_{\vecy^{(i)}}X^{(i)}_{2^{m+1}+1}|\pr_{[1..i-1]}(X_{\vecy}X_{2^{m+1}+1}))
\\
\label{eq:initial-inequality-average-conditional-entropy}
\ge
&
\sum_{i=1}^n H(X^{(i)}_{\vecy^{(i)}}X^{(i)}_{2^{m+1}+1}|\{\pr_{[1..i-1]}(X_j)\}_{j=1}^{2^{m+1}+1})
\end{align}
By \eqref{eq:initial-inequality-average-conditional-entropy}, we get
\begin{align}
\notag
\sum_{\vecy\in \bigoplus_{j=1}^{2^{m+1}-1} G} p_{\vecy} H(X_{\vecy} X_{2^{m+1}+1})\ge &
\sum_{i=1}^n \sum_{\vecy\in \bigoplus_{j=1}^{2^{m+1}-1} G} p_{\vecy} H(X^{(i)}_{\vecy^{(i)}}X^{(i)}_{2^{m+1}+1}|\{\pr_{[1..i-1]}(X_j)\}_{j=1}^{2^{m+1}+1})
\\
\label{eq:second-inequality-conditional-entropy}
=
&
\sum_{i=1}^n \sum_{\vecy^{(i)}\in \bigoplus_{j=1}^{2^{m+1}-1} G_i} p_{\vecy^{(i)}} H(X^{(i)}_{\vecy^{(i)}}X^{(i)}_{2^{m+1}+1}|\{\pr_{[1..i-1]}(X_j)\}_{j=1}^{2^{m+1}+1})	
\end{align}
We notice that for a given $1\le i\le n$, we have
\begin{align}
\notag
h_i:=\sum_{\vecy^{(i)}\in \bigoplus_{j=1}^{2^{m+1}-1} G_i} p_{\vecy^{(i)}} H(X^{(i)}_{\vecy^{(i)}}X^{(i)}_{2^{m+1}+1}|
& \{\pr_{[1..i-1]}(X_j)\}_{j=1}^{2^{m+1}+1})=
\\
\label{eq:connection-previous-result-conditional-entropy}
 H(X^{(i)}_1Y^{(i)}_1X^{(i)}_2\cdots Y^{(i)}_{2^{m+1}-1} X^{(i)}_{2^{m+1}} X^{(i)}_{2^{m+1}+1}|
 &
\{\pr_{[1..i-1]}(X_j)\}_{j=1}^{2^{m+1}+1}, \{Y^{(i)}_j\}_{j=1}^{2^{m+1}-1}).
\end{align}
Hence, if $i\in J_{\rm g}\cap I_{\rm s}$ and $D_i\neq 1$, by Proposition~\ref{prop:VarjuMutation}, we have
\be\label{eq:gained-conditional-entropy}
h_i\ge \log D_i+ \gamma \log |G_i|.
\ee
By Lemma~\ref{lem:product-entropy-trivial-bound}, we have the trivial bound
\be\label{eq:trivial-bound-entropy}
h_i\ge \log D_i,
\ee
for any $i$. By \eqref{eq:second-inequality-conditional-entropy}, \eqref{eq:connection-previous-result-conditional-entropy}, \eqref{eq:gained-conditional-entropy}, and \eqref{eq:trivial-bound-entropy}, we get
\be\label{eq:total}
\sum_{\vecy\in \bigoplus_{j=1}^{2^{m+1}-1} G} p_{\vecy} H(X_{\vecy} X_{2^{m+1}+1})\ge
\log |A|+\gamma \log |G_{J_{\rm g}\cap I_{\rm s}}|.
\ee
We notice that
\be\label{eq:adding-bad-indexes}
\log |G_{I_s}| = \log |G_{J_{\rm g}\cap I_{\rm s}}|+\log |G_{J_{\rm b}\cap I_{\rm s}}|
\le \log |G_{J_{\rm g}\cap I_{\rm s}}| + \log |G_{J_{\rm b}}|
\le \log |G_{J_{\rm g}\cap I_{\rm s}}| + \frac{\delta}{\vare/(2L)} \log |G|,
\ee
and $\log |A|\ge \log |S|-2\delta \log |G|$. Hence by \eqref{eq:total} and \eqref{eq:adding-bad-indexes}, we have
\[
\sum_{\vecy\in \bigoplus_{j=1}^{2^{m+1}-1} G} p_{\vecy} H(X_{\vecy} X_{2^{m+1}+1})\ge 
\log |S|+ \gamma \log |G_{I_s}| - \gamma \delta (2+\frac{2L}{\vare}) \log |G|; 
\]
and the claim follows.
\end{proof}
\begin{cor}\label{cor:multi-scaled-expansion-scales-with-room}
	In the above setting, we have
	\[
	\textstyle 
	\log |\prod_{2^{m+2}} S|\ge \log |S|+\gamma \log |G_{I_{\rm s}}|-\gamma \frac{\delta}{\vare/(3L)} \log |G|.
	\]
\end{cor}
\begin{proof}
By Lemma~\ref{lem:main-gain-scales-with-room}, there is $\vecy\in B\times \cdots \times B$ such that 
\[
H(X_{\vecy}X_{2^{m+1}+1})\ge \log |S|+\gamma \log |G_{I_{\rm s}}|-\gamma \frac{\delta}{\vare/(3L)} \log |G|.
\]
On the other hand, $H(X_{\vecy}X_{2^{m+1}+1})\le H_0(X_{\vecy}X_{2^{m+1}+1})\le \log |\prod_{2^{m+2}} S|$; and the claim follows.
\end{proof}

\subsection{Multi-scale product result: proof of Proposition~\ref{propvarjuproduct}} 
In this section, we still work in the setting listed in the previous section, and finish the proof of Proposition~\ref{propvarjuproduct}. 

By Proposition~\ref{proplargeindices}, we either have $|\prod_3 S|\ge |G|^{1-\vare+\delta}$ in which case the claim follows or
\be\label{eq:first-bound}
\textstyle
\log |\prod_{14} S| \ge \log |G_{I_{\rm l}}|+\Theta_L(\vare^3) \log |G|.
\ee   
By Corollary~\ref{cor:multi-scaled-expansion-scales-with-room}, we have
\be\label{eq:second-bound}
\textstyle 
	\log |\prod_{2^{m+2}} S|\ge \log |S|+\gamma \log |G_{I_{\rm s}}|-\Theta_L(\gamma \vare^4) \log |G|.
\ee 
By \eqref{eq:first-bound} and \eqref{eq:second-bound}, we get
\[
\textstyle
(1+\gamma) \log |\prod_{2^{m+2}} S| \ge \log |S|+\gamma \log |G|.
\]
As $\log |S|\le (1-\vare)\log |G|$, we deduce
\[
\textstyle
(1+\gamma)\log |\prod_{2^{m+2}} S| \ge \log |S|+\frac{\gamma}{1-\vare} \log |S|\ge (1+\gamma) \log |S|+\gamma\vare \log |S|.
\]
Hence by \cite[Lemma 2.2]{Hel1} we get
\[
\textstyle
(1+\gamma)(2^{m+2}-2)(\log |\prod_3 S|-\log |S|)\ge (1+\gamma)(\log |\prod_{2^{m+2}} S|-\log |S|)\ge \gamma\vare \log |S|;
\]
and the claim follows.

\section{Super-approximation: proof of Theorem~\ref{t:SpectralGap}}~\label{sec:finishing-proof-SA}	
As it has been pointed out by Bradford (see~\cite[Theorem 1.14]{Bra}) Varj\'{u} has already proved a multi-scale version of Bourgain-Gamburd's result which in combination with Proposition~\ref{propvarjuproduct} can be formulated as follows (see~\cite[Sections 3 and 5]{Var}). 
\begin{thm}\label{thm:VarjuBourgainGamburdMachine}
Suppose $L$ is a positive integer, $\delta_0:\bbr^+\rightarrow \bbr^+$, and $\{G_i\}_{i=1}^{\infty}$ is a family of finite groups that satisfy V(1)$_L$-V(3)$_L$ and V(4)$_{\delta_0}$. Suppose $\overline{\Omega}$ is a symmetric generating set of $G:=\bigoplus_{i=1}^n G_i$. Suppose there are $\eta>0$, $C_0$, and $l<C_0\log |G|$ such that for any proper subgroup $H$ of $G$, we have $\pcal_{\overline{\Omega}}^{(2l)}(H)\le [G:H]^{-\eta}$. Then
\[
1-\lambda(\pcal_{\overline{\Omega}};G)\gg_{L,\delta_0,\eta,C_0,|\overline{\Omega}|} 1.
\]	
\end{thm}
  Let us recall that by the discussion in Section~\ref{ss:storngapproximation}, we can and will assume
  \be\label{eq:group-structure}
  \pi_f(\Gamma)\simeq \bigoplus_{\ell|f, \ell \text{ irred.}} \gcal_{\ell}(K(\ell))
  \ee
  where $K(\ell)$ is the finite field $\bbf_{p_0}[t]/\langle \ell\rangle$, $\gcal_{\ell}$ is an absolutely almost simple, simply connected, $K(\ell)$-group, and the absolute type of all $\gcal_{\ell}$'s are the same. 
  
  By Proposition~\ref{propmainescape}, there is a symmetric subset $\Omega'$ of $\Gamma$, a square-free polynomial $r_1$, and positive numbers $c_0$ and $\delta$ such that for any $f\in S_{r_1,c_0}$ (that means $f$ and $r_1$ are coprime and the degree of any irreducible factor of $f$ does not have a prime factor less than $c_0$), any purely structural subgroup $H$ of $\pi_f(\Gamma)$ and $l\gg_{\Omega} \deg f$, we have 
  \be\label{eq:escape-purely-structural-and-generation}
  \pi_f(\langle \Omega'\rangle)=\pi_f(\Gamma),\text{ and } \pcal^{(l)}_{\pi_f(\Omega')}(H)\le [\pi_f(\Gamma):H]^{-\delta};
  \ee
  moreover $\Omega'=\Omega'_0\sqcup {\Omega'_0}^{-1}$ and $\Omega'_0$ freely generates a subgroup of $\Gamma$. By \eqref{eq:escape-purely-structural-and-generation}, to prove Theorem~\ref{t:SpectralGap}, it is enough to prove 
\[
1-\lambda(\pcal_{\pi_f(\Omega')};\pi_f(\Gamma))\gg_{\Omega} 1.
\]
By \eqref{eq:group-structure}, \eqref{eq:escape-purely-structural-and-generation}, and Theorem~\ref{thm:VarjuBourgainGamburdMachine}, to prove Theorem~\ref{t:SpectralGap}, it is enough to prove the following.
\begin{enumerate}
\item There are $L$ and $\delta_0$ such that $\gcal_{\ell}(K(\ell))$ satisfies V(1)$_L$-V(3)$_L$, and V(4)$_{\delta_0}$ if $\ell$ is an irreducible polynomial that does not divide $r_1$.	
\item There are $\eta>0$, $C_0$, and $c_0'\ge c_0$ such that for any $f\in S_{r_1,c_0'}$ and any proper subgroup $H$ of $\pi_f(\Gamma)$ we have $\pcal_{\pi_f(\Omega')}^{(2l)}(H)\le [\pi_f(\Gamma):H]^{\eta}$ for some $l<C_0 \deg f$.
\end{enumerate}
 In the rest of this section, we will prove these items.
 
 \subsection{Verifying Varj\'{u}'s assumptions V(1)$_L$-V(3)$_L$, and V(4)$_{\delta_0}$ for $\gcal_{\ell}(K(\ell))$} Since $\gcal_{\ell}$'s are 
absolutely almost simple, simply connected, $K(\ell)$-groups, and all of them have the same absolute type, by \cite{LS-quasi}, they satisfy V(1)$_L$ and V(2)$_L$ for some positive integer $L$. By the groundbreaking results \cite[Corollary 2.4]{BGT} and \cite[Theorem 4]{PS}, there is a function $\delta_0$ such that $\gcal_{\ell}(K(\ell))$'s satisfy V(4)$_{\delta_0}$.

Now we introduce the families of subgroups $\cal_i$ and $\cal'_j$, and prove that they satisfy V(3)$_L$ for some positive integer $L$ that is independent of irreducible polynomials $\ell$'s.

By Theorem~\ref{thm:FinalRefinementOfLP}, for a structural subgroup $H$ of $\gcal_{\ell}(K(\ell))$, there is a proper subgroup $\bbh$ of $\gcal_{\ell}$ with complexity $O_{\Gamma}(1)$ such that $H\subseteq \bbh(K(\ell))$. Since the complexity of $\bbh$ is $O_{\Gamma}(1)$, $[\bbh(K(\ell)):\bbh^{\circ}(K(\ell))]\ll_{\Gamma}1$ and the complexity of $\bbh^{\circ}$ is also bounded by a function of $\Gamma$, where $\bbh^{\circ}$ is the connected component of the identity of $\bbh$ in the Zariski topology. For $0\le i< \dim \bbg$, initially we let 
\[
\cal_i:=\{\bbh(K(\ell))|\h \bbh \lneq \gcal_{\ell}, \dim \bbh=i, 
\bbh=\bbh^{\circ} \text{, its complexity is bounded as above}\}
\]
Next for smaller dimension subgroups, we allow slightly larger complexity to include the connected components of the intersections of larger dimension connected proper subgroups. Then by Theorem~\ref{thm:FinalRefinementOfLP} a subgroup $H$ of $\gcal_\ell(K(\ell))$ is a structural subgroup if and only if there is $H^\sharp$ in $\cal_i$ for some $i$ such that $H\preceq_L H^\sharp$ where $L:=O_{\Gamma}(1)$. Moreover $\cal_i$'s satisfy the condition (V3)$_L$,(i)-(ii).
 
By Theorem~\ref{thm:FinalRefinementOfLP}, we know that if $H$ is a proper subgroup of subfield type of $\gcal_{\ell}(K(\ell))$, then there is a subfield $F_H$ and an $F_H$-model $\bbg_H$ of $\Ad(\gcal_{\ell})$ such that \[[\bbg_H(F_H),\bbg_H(F_H)]\subseteq \Ad(H)\subseteq \bbg_H(F_H).\] This implies that $H\preceq \wt{\bbg}_H(F_H)$ where $\wt{\bbg}_H$ is a simply-connected cover of $\bbg_H$. 

\begin{lem}~\label{lem:subfield-type-subgroups-up-to-conjugation}
For a subgroup $H$ of $\gcal_{\ell}(K(\ell))$, let ${\rm Con}(H)$ be the set of all the conjugates of $H$ in $\gcal_{\ell}(K(\ell))$. For a subfield $F$ of $K(\ell)$, let
\[
n_F:=|\{{\rm Con}(\wt{\bbg}(F))|\h \wt{\bbg}\text{ is an } F\text{-model of } \gcal_{\ell}\}|.
\] 
Then $n_F\ll_{\bbg} 1$.
\end{lem}
\begin{proof}
Since $\wt\bbg$ has the same absolute type as $\gcal_{\ell}$, up to $F$-isomorphism there are only two choices; either $\wt\bbg$ is the unique simply connected $F$-split group of the given absolute type or it is the unique quasi-split outer form of the given absolute type defined over $F$. So without loss of generality, we fix an $F$-model $\wt\bbg_0$ of $\gcal_{\ell}$ and we want to show that
\[
|\{{\rm Con}(\wt{\bbg}(F))|\h \wt{\bbg}\simeq \wt{\bbg}_0\text{ as } F\text{-groups}\}|\ll_{\bbg} 1.
\] 
Notice that if $\wt\bbg$ is an $F$-model of $\gcal_{\ell}$ which is $F$-isomorphic to $\wt{\bbg}_0$, then there is an $F$-isomorphism $\phi:\wt\bbg_0\xrightarrow{\simeq} \wt{\bbg}$ which induces an automorphism of $\gcal_{\ell}$ after base change. Since $[\Aut(\gcal_{\ell}):\Ad(\gcal_{\ell})]\ll 1$, without loss of generality we can and will assume that $\phi$ induces and inner automorphism of $\gcal_{\ell}$. Hence we can and will assume that there is $g\in \gcal_{\ell}(\overline{K(\ell)})$ such that $g\wt{\bbg}_0(F)g^{-1}=\wt{\bbg}(F)\subseteq \wt{\bbg}_0(K(\ell))$. Therefore by Proposition~\ref{prop:conjugates-subfield-type-subgroups}, $\Ad(g)\in \Ad(\wt{\bbg}_0)(K(\ell))=\Ad(\gcal_{\ell})(K(\ell))$. Since $[\Ad(\gcal_{\ell})(K(\ell)):\Ad(\gcal_{\ell}(K(\ell)))]\ll_{\bbg} 1$, the claim follows.
\end{proof}
We let $\cal_i'$'s be the sets of conjugacy classes of the groups of the form $\wt\bbg(F)$ where $F$ is a proper subfield of $K(\ell)$ and $\wt\bbg$ is an $F$-model of $\gcal_{\ell}$. By Lemma~\ref{lem:subfield-type-subgroups-up-to-conjugation}, there are at most $L\log |\gcal_{\ell}(K(\ell))|$ where $L$ just depends on the absolute type of $\gcal_{\ell}$'s. By Corollary~\ref{cor:intersection-conjugate-subfield-type}, $\cal_i'$'s satisfy property (V3)$_L$-(v); and our claim follows.
\subsection{Escaping proper subgroups} The main goal of this short section is to show that Proposition~\ref{propmainescape} is good enough to show the needed escaping from an arbitrary proper subgroup of $\pi_f(\Gamma)$ for $f\in S_{r_1,c_0}$. 
\begin{lem}\label{lem:final-escaping-proper-subgroups}
In the setting described at the beginning of Section~\ref{sec:finishing-proof-SA}	, there are $\eta>0$, $C_0$, and $c_0'\ge c_0$ such that for any $f\in S_{r_1,c_0'}$ and any proper subgroup $H$ of $\pi_f(\Gamma)$ we have 
\[
\pcal_{\pi_f(\Omega')}^{(2l)}(H)\le [\pi_f(\Gamma):H]^{-\eta}
\]
for some $l<C_0\deg f$.
\end{lem}
\begin{proof}
We assume that $f\in S_{r_1,c_0'}$ for a sufficiently large $c_0'\ge c_0$ (to be specified later). We split the set of irreducible divisors of $f$ into three disjoint sets:
\[
D_1(f;H):=\{\ell|f \text{ s.t. } \ell\text{ is irreducible, } \pi_\ell(H) \text{ is a structural subgroup}\},
\]
\[
D_2(f;H):=\{\ell|f \text{ s.t. } \ell\text{ is irreducible, } \pi_\ell(H) \text{ is a subfield type subgroup}\}, \text{ and }
\]
\[
D_3(f;H):=\{\ell|f \text{ s.t. } \ell\text{ is irreducible, } \pi_\ell(H)=\pi_\ell(\Gamma)\}.
\]
Let $f_i:=\prod_{\ell\in D_i(f;H)} \ell$ and $H_i:=\prod_{\ell\in D_i(f;H)} \pi_\ell(H)$; then by Lemma~\ref{lemsubgroupproductform}  we have that 
\be\label{eq:approximating-independent-subgroups}
[\pi_{f_1}(\Gamma):H_1][\pi_{f_2}(\Gamma):H_2]=[\pi_f(\Gamma):H_1\oplus H_2\oplus H_3] \ge [\pi_f(\Gamma):H]^{1/L}.
\ee
By Proposition~\ref{propmainescape}, we have 
\be\label{eq:escaping-purely-structural-part}
\pcal_{\pi_{f_1}(\Omega')}^{(2l)}(H_1)\le [\pi_{f_1}(\Gamma):H_1]^{-\delta_0},
\ee
where $\delta_0$ is a positive number which just depends on $\Omega$. On the other hand, by Kesten's result on random walks in a free group (see~\cite[Theorem 3]{Kes}), we have that, for some $\ell_0\ll_{\Omega} \deg f_2$ and $c_1>0$, we have
\[
\| \pcal_{\pi_{f_2}(\Omega')}^{(2l_0)}\|_{\infty} \le |\pi_{f_1}(\Gamma)|^{-c_1};
\]
and so 
\be\label{eq:escaping-purely-subfield-type-1}
\pcal_{\pi_{f_2}(\Omega')}^{(2l_0)}(H_2)\le |\pi_{f_1}(\Gamma)|^{-c_1}|H_2|\le |\pi_{f_1}(\Gamma)|^{-c_1}\prod_{\ell|f_2, \text{irr. }} |\pi_{\ell}(H_2)|\le |\pi_{f_1}(\Gamma)|^{-c_1}
|\pi_{f_2}(\Gamma)|^{1/c'_0}.
\ee
So if $c_0'>2/c_1$, then by \eqref{eq:escaping-purely-subfield-type-1} implies that
\be\label{eq:escaping-purely-subfield-type-2}
\pcal_{\pi_{f_2}(\Omega')}^{(2l_0)}(H_2)\le |\pi_{f_1}(\Gamma)|^{-c_1/2}.
\ee
By \eqref{eq:escaping-purely-subfield-type-2} and the fact that $\Omega'$ is a symmetric set, we have $\pcal_{\pi_{f_2}(\Omega')}^{(l_0)}(gH_2)\le |\pi_{f_1}(\Gamma)|^{-c_1/4}$; and so for any $l\ge l_0$, we have
\be\label{eq:escaping-purely-subfield-type}
\pcal_{\pi_{f_2}(\Omega')}^{(l)}(H_2)\le |\pi_{f_1}(\Gamma)|^{-c_1/4}.
\ee
By \eqref{eq:escaping-purely-structural-part} and \eqref{eq:escaping-purely-subfield-type}, we deduce that
\begin{align}
\notag
\pcal_{\pi_{f}(\Omega')}^{(2l)}(H)\le
&
\min(\pcal_{\pi_{f_1}(\Omega')}^{(2l)}(H_1),\pcal_{\pi_{f_2}(\Omega')}^{(2l)}(H_2))
\\
\notag
\le 
&
(\pcal_{\pi_{f_1}(\Omega')}^{(2l)}(H_1)\pcal_{\pi_{f_2}(\Omega')}^{(2l)}(H_2))^{1/2}
\\
\label{eq:almost-finished}
\le 
&
|\pi_{f_1f_2}(\Gamma)|^{-\min(\delta_0,c_1/4)}\le [\pi_f(\Gamma):H_1\oplus H_2\oplus H_3]^{-\min(\delta_0,c_1/4)}.
\end{align}
By \eqref{eq:almost-finished} and \eqref{eq:approximating-independent-subgroups}, we get
\[
\pcal_{\pi_{f}(\Omega')}^{(2l)}(H) \le [\pi_f(\Gamma):H]^{-\min(\delta_0/L,c_1/(4L))};
\]
and the claim follows.
\end{proof}




\end{document}